\documentclass[11pt]{article}

\pdfoutput=1

\RequirePackage{my_packages}
\RequirePackage{my_theorems}
\RequirePackage{my_macros}
\RequirePackage{my_tikz}

\usepackage[backend = biber,        
            language = english ,
            style = numeric-comp ,   
            sorting = nyt,
            sortcites = true,
            firstinits = true,
            isbn = false,
            url = false,
            doi = false,
            maxnames = 6,
            backref=true
            ]{biblatex}

\DeclareNameAlias{sortname}{last-first}
\renewbibmacro{in:}{%
  \ifentrytype{article}{}{%
  \printtext{\bibstring{in}\intitlepunct}}}
\AtEveryBibitem{
 \clearlist{address}
 \clearfield{date}
 \clearfield{eprint}
 \clearfield{isbn}
 \clearfield{issn}
 \clearlist{location}
 \clearfield{month}
 \clearfield{series}
 \clearlist{language}
 \clearfield{note}
 
}

\DeclareNameAlias{labelname}{given-family}

\title{Stallings automata
}
\author{Jordi Delgado and Enric Ventura}
\date{\vspace{-7pt}
    Departament de Matemàtiques\\
    Universitat Politècnica de Catalunya\\[17pt]
    \today
}

\newcommand{\Addresses}{{
  \bigskip
  \footnotesize

  Jordi Delgado\\\nopagebreak
  \textsc{Departament de Matemàtiques\\Universitat Politècnica de Catalunya, Spain}\\\nopagebreak
  \url{jorge.delgado@upc.edu}

  \medskip

  Enric Ventura\\\nopagebreak
  \textsc{Departament de Matemàtiques\\Universitat Politècnica de Catalunya, Spain}\\\nopagebreak
  \url{enric.ventura@upc.edu }

}}

\begin{document}

\maketitle

\begin{abstract}
\noindent 
This chapter provides a self-contained reference for Stallings' automata theory, which allows to interpret every subgroup of the free group as a neat geometric object: its 
\emph{Stallings automaton}.
In contrast to the original topological approach ---rooted in the seminal works of Serre, Stallings, and others--- we adopt a discrete perspective, where geometric objects are abstract combinatorial structures,
and subsequently reveal the connections with the underlying topological notions. This allows us to highlight both sides of the theory: the algebraic part (closely linked to the topological point of view), and the computational part (more naturally described in terms of automata and language theory).

Once the bijection between subgroups and automata is established, we focus on some classical applications of this theory. From the algebraic point of view, we deduce the Nielsen--Schreier theorem,
the Schreier index formula, Howson's property for free groups (along with classical bounds for the rank of the intersection of subgroups), Marshall Hall theorem, and properties such as Hopfianity and residual finiteness of free groups.

On the other hand, the computability of the Stallings automaton for finitely generated subgroups, based on the now classical technique of \emph{Stallings foldings}, yields a plethora of algorithmic results. We prove the computability of the \emph{subgroup membership problem}, the \emph{finite index problem}, the \emph{subgroup conjugacy problem}, the \emph{coset intersection problem}; as well as the computability of
intersections of finitely generated subgroups, and the set of algebraic extensions of a finitely generated subgroup, among other results. 
\end{abstract}



\vspace{10pt}


\vfill
\hfill
Version 1.0
\newpage
\setcounter{tocdepth}{1}
\tableofcontents

\newpage

This \monograph\ aims to provide a self-contained, comprehensive, and reasonably detailed presentation of the theory of Stallings' automata and some of its main applications.

Although the approach is mostly combinatorial and language-theoretical, aligning with previous surveys such as \cite{kapovich_stallings_2002}, \cite{bartholdi_rational_2021} and \cite{delgado_list_2022}, our intention is to uncover the  connections with the more topological perspective in \cite{serre_trees_1980} and~\cite{stallings_topology_1983}. In some sense, this text is intended to complement and extend the mentioned works providing some extra details, and including full formal proofs for almost all the presented results. 

\section{Introduction}
The classical \emph{theory of Stallings' automata} provides a neat geometric representation of the subgroups of the free group and constitutes the modern --- and probably the most natural and fruitful --- approach to their systematic study. Moreover, if the involved subgroups are finitely generated, then this description is finitary and very well suited for algorithmic treatment.

The original result (hinted by the work of Serre in \parencite{serre_arbres_1977}, and precisely stated in the seminal paper~\parencite{stallings_topology_1983} by Stallings) interprets the subgroups of the free group~$\Free[A]$ (with basis $A$ of cardinality $n$) as covering spaces of the bouquet of $n$ circles. Despite this (mainly topological) original viewpoint, when restricted to finitely generated subgroups, the obtained bijection is easily computable, and it soon became clear that the entire approach admitted an appealing restatement in terms of automata.

More precisely, the involved automata are directed graphs with a distinguished base vertex, and with the arcs (directed edges) labeled by the elements in the chosen basis $A$. The underlying idea is that the labels of the closed paths at the basepoint of such an object constitute a subgroup of $\Free[A]$ (the subgroup $H\leqslant \Free[A]$ recognized by the automaton). Stallings' theory establishes precise conditions (determinism and coreness) on these automata in order to make this representation of subgroups bijective. The resulting automaton associated to a subgroup $H\leqslant \Free[A]$, called the Stallings' automaton of $H$ and denoted by $\stallings(H)$, turns out to be a part of the Schreier automaton $\schreier(H)$, and is finite if and only if the subgroup $H$ is finitely generated. In this case, in order to compute $\stallings(H)$ from a finite set of generators $S$ of $H$, we start from the flower automaton (each of its petals reading one of the generators in $S$) and then keep making identifications (called \emph{Stallings' foldings}) between incident arcs with the same label until reaching a unique final deterministic object, which is precisely $\stallings(H)$. As we will see, many classic results on free groups --- such as the Nielsen--Schreier theorem, the Schreier index formula, Howson's property, Marshall Hall Theorem, residual finiteness, etc. --- neatly emerge from this graphical interpretation.

Moreover, the correspondence between finitely generated subgroups and finite Stallings' automata being fully algorithmic (and quite fast) constitutes a very powerful tool for attacking algorithmic problems on subgroups of free groups: we shall see that the subgroup membership problem (\MP), the finite index problem (\FIP), the computation of bases for intersections, and checking the normality and malnormality of a given finitely generated subgroup (among many other algorithmic problems) can be successfully solved using Stallings' machinery.

\medskip

In \Cref{sec: free groups} we offer a brief introduction to free groups and establish the notation in use throughout the rest of the document. In \Cref{sec: Digraphs and automata} we introduce the main facts and notation concerning graphs, digraphs, and automata. In \Cref{sec: Stallings automata} we define Stallings' automata and prove the bijection between them and subgroups of the free group, which is the central tool in the theory. As a first fundamental application, in \Cref{sec: MP} we solve the membership problem for free groups, including the explicit writing of the candidate element in terms of the given generators for the subgroup. Other classical applications follow in the subsequent sections. In \Cref{sec: conj & cosets} we discuss basic problems related to conjugacy of subgroups, normality, and malnormality, as well as finite index subgroups and cosets. In \Cref{sec: intersections} we study intersections of subgroups and we obtain some classical bounds for its rank. In \Cref{sec: extensions}, we study extensions of free groups and give a modern proof of a classical result by Takahasi. Finally, in \Cref{sec: further}, we summarize (without proofs) several other applications and extensions of the theory. This is just a selection of some of the main applications of Stallings' techniques to the study of the lattice of subgroups of the free group. Since Stallings' seminal work~\parencite{stallings_topology_1983}, many other applications have appeared (and keep appearing) in the literature (see~\cite{delgado_list_2022}), and more are likely to arise in the future.

\subsection{Notation and conventions}

Most of the notation and terminology used throughout this \monograph\ are standard; however, we clarify certain conventions below.

The set of natural numbers, denoted by $\NN$, is assumed to contain zero, and its (countable) cardinal is sometimes denoted by $\infty$ (instead of $\aleph_0$). For $r,s\in \NN \cup \set{\infty}$, we write $[r,s] = \set{n \in \NN \cup \set{\infty} \st r \leq n \leq s}$.
The cardinal of a set $S$ is denoted by $\card S$, whereas the notation $|\cdot|$ is reserved to denote length (in different contexts). Latin letters $m,n,r,\ldots$ usually denote finite cardinals, whereas Greek letters $\kappa,\mu,\ldots$ denote arbitrary cardinals.

We denote by $\Free[A]$ a generic free group (with basis $A$), while 
$\Free[\kappa]$ (\resp $\Fn$) are used to emphasize its rank $\kappa$ (\resp finite rank $n$). The first lowercase letters of the Latin alphabet ($a,b,c,\ldots$) are commonly used to denote symbols in our formal alphabets, whereas the latter ones ($u,v,w,\ldots$) usually denote formal words or elements in the free group. 

Functions are assumed to act on the right. For example, we denote by $(x)\varphi$ ---or simply by $x \varphi$--- the image of the element $x$ by a homomorphism $\varphi$, and we denote by  $\varphi \psi$ the composition \smash{$A\xto{\varphi\phantom{\psi}\hspace{-8pt}} B \xto{\psi} C$}. Accordingly, we write, for example, $g^{h} = h^{-1} gh$ and $[g,h] = g^{-1} h^{-1} g h$. We sometimes write $A \into B$ (\resp $A \onto B$) to express that a function $A \xto{} B$ is injective (\resp surjective),
and we denote by~$(C)\varphi\preim$ the full preimage of $C \subseteq B$ by the map $\varphi\colon A \xto{} B$.


We write $H \leqslant G$ (\resp $H \normaleq G$), to denote that $H$ is a subgroup (\resp normal subgroup) of $G$; and we write $H \leqslant\fg G$, $H \leqslant\fin G$, $H \leqslant\ff G$, $H \leqslant\alg G$ to express that the subgroup $H$ is finitely generated, of finite index, a free factor of, and algebraic in $G$, respectively. 

For the different variants of graphs appearing in the paper, the following convention applies: uppercase Greek letters ($\Gamma, \Delta, \ldots$) denote unlabeled undirected graphs, the arrow accentuated versions ($\Dgri,\Dgrii,\ldots$) denote unlabeled directed graphs, and the boldface versions $(\Ati,\Atii, \ldots)$ denote labeled digraphs (including automata). 



\section{Free groups}\label{sec: free groups}

We assume that the reader is familiar with the construction and basic properties of free monoids and free groups. For later reference and to fix the notation, we provide a brief summary of the main definitions and constructions, without detailed proofs.

Throughout the paper, we denote by $A=\set{a_i}_{i\in I}$ a set called \defin{alphabet}, whose elements are called  \defin[letter]{letters}. Depending on the context, this set may be finite or not.
A \defin{word} or \defin{string} over $A$ is a finite sequence of elements in $A$. If $w=(a_{i_1},\ldots ,a_{i_l})\in A^l$ then we simply write $w=a_{i_1} \cdots a_{i_l}$, and  we say that $w$ has \defin[word length]{length} $l$, denoted by~$|w|=l$. In particular, the (unique) word of length $0$ is called the \defin{empty word} and it is denoted by $\emptyword$. We denote by $A^{\!*}$ the \defin{free monoid} on $A$, consisting of all the words on $A$ (including the empty word, working as the neutral element) endowed with the concatenation operation. Note that, for $u,v\in A^{\!*}$, $|uv|=|u|+|v|$ and that, as soon as $|A|\geq 2$, the monoid $A^*$ is not commutative. 

\begin{rem}
If the alphabet $\Alfi$ is finite and nonempty, then the free monoid $\Alfi^{\!*}$ has countable cardinal $\aleph_0$; otherwise, $\Alfi$ and $\Alfi^{\!*}$ have the same cardinality. In particular, a nontrivial free monoid is always infinite.
\end{rem}


We say that a nonempty word $u\in A^*$ is a \defin[factor of a word]{factor} (\resp \defin{proper factor}) of a word $w\in A^*$ if $w=vuv'$ for some words $v,v'\in A^*$ (\resp at least one of them being non-trivial); specifically, if $v=1$ (\resp  $v'=1$), $u$ is called a \defin{prefix} (\resp a \defin{suffix}) of~$w$. 


If $A$ is an alphabet, the subsets of $A^{\!*}$ are called \defin[language]{languages} (over $A$), or $A$-languages. A language over a finite alphabet $A$ is said to be \defin[rational language]{rational} if it can be  obtained from the singletons in $A$ using finitely many times the operators of union, product ($ST=\{st \mid s\in S,\, t\in T\}$) and star ($S^*=\langle S\rangle$, \ie the submonoid generated by $S$). A direct consequence of a fundamental theorem of Kleene (see for example \parencite{sakarovitch_elements_2009}) is that the set of rational $A$-languages is closed under finite intersections and taking complements.

In the context of groups (and free groups in particular), a particular type of alphabet plays a prominent role. Given an alphabet $A$, we  denote by~$A^{-1}$ the \defin{set of formal inverses} of~$A$. Formally, $A^{-1}$ can be defined as a new set $A'$ equipotent and disjoint with $A$, together with a bijection $ A \xto{} A'$. Then, for every $a\in A$, we write $a^{-1}$ the image of~$a$, and we call it the \defin{formal inverse} of $a$. So, $A^{-1}=\{ a^{-1} \st a \in A \}$ and $A\cap A^{-1}= \varnothing$. Then, the set $A^{\pm} =A\sqcup A^{-1}$, called the \defin{involutive closure} of $A$, is equipped with an involution $^{-1}$ (by defining $(a^{-1})^{-1} = a$), which can be extended to $(A^{\pm})^*$ in the natural way: 
the \defin{inverse of a word} $w =a_{i_1} \cdots a_{i_l}$ is $w^{-1} = a_{i_l}^{-1}\cdots a_{i_1}^{-1}$.
An alphabet is called \defin[involutive alphabet]{involutive} if it is the involutive closure of some other alphabet.

In such an involutive ambient $(A^{\pm})^{*}$, a word consisting of two mutually inverse letters (\ie a word of the form $a a^{-1}$, for some $a\in A^{\pm}$) is called a \defin{cancellation}. A word in $(A^{\pm})^{*}$ (which we just refer to as a \emph{word on $A$}) is said to be \defin[freely reduced word]{(freely) reduced} if it contains no cancellation (\ie it has no factor of the form $a a^{-1}$, where $a \in A^{\pm}$). It is well known that the word obtained from $w\in (A^{\pm})^{*}$ by successively removing cancellations (in any possible order) is unique (apply \Cref{prop: unicitat cami reduit} below to an automaton $\Ati$ with a single vertex); we call it the \defin{free reduction} of $w$, and we denote it by $\red{w}$. Similarly, we write $\red{S} = \set{\red{w} \st w \in S}$, for any subset $S\subseteq (A^{\pm})^{*}$, and we denote by $\Red_A$ the \defin{set of reduced words on $A$},
 \begin{equation*}
\Red_A \,=\, \red{(A^{\pm})^{*}} \,=\, (A^{\pm})^{*} \setmin \textstyle{\bigcup_{a \in A^{\pm}} (A^{\pm})^{*} a a^{-1} (A^{\pm})^{*}, }
 \end{equation*}
which is clearly a rational language.

In a similar vein, a word $w \in (A^{\pm})^*$ is said to be \defin[cyclically reduced word]{cyclically reduced} if all of its cyclic permutations are reduced (that is, if $w^2$ is reduced). The \defin{cyclic reduction} of a word $w \in \Free[A]$
is obtained after successively removing from $\red{w}$ the first and last letters if they are inverses of each other. 

For an arbitrary alphabet $A$, the \defin[${\Free[A]}$]{free group on $A$}, denoted by $\Free[A]$, can be defined as the quotient 
 \begin{equation} \label{eq: free group def}
\Free[A] \,=\, (A^{\pm})^* / \mathcal{C} \,,
 \end{equation}
where $\mathcal{C}$ is the congruence modulo cancellation (generated by the pairs $\set{(a a^{-1}, \emptyword) \st a \in A^{\pm}}$, with the product inherited from concatenation in $(A^{\pm})^*$). It is easy to see this resulting monoid is, in fact, a group. As indicated above, every equivalence class contains a unique reduced word, easily computable from any given representative; this allows the alternative definition of $\Free[A]$ as $\Red_A$ together with the operation given by $u\cdot v=\red{uv}$. In particular, a word $w\in (A^\pm)^*$ represents the trivial element in $\Free[A]$ if and only if $\red{w}=\emptyword$; this immediately solves the word problem for free groups. Since letters and inverse letters (viewed as words of length 1) are clearly reduced, we deduce that the natural map $\iota_A\colon A\xto{} \Free[A]$ is injective.

We say that a group is \emph{free on $A$} (or \defin[free basis]{free with basis $A$}) if it is isomorphic to $\Free[A]$; and we say that a group is \defin[free group]{free} if it is isomorphic to $\Free[A]$, for some set $A$. We further recall that Benois' Theorem (see~\parencite{benois_parties_1969}) allows us to understand rational subsets of $\Free[A]$ as reductions of rational $A^{\pm}$-languages.

\begin{defn}\label{def: rank}
The \defin[rank of a group]{rank} of a group $G$, denoted by $\rk(G)$, is the smallest cardinality of a generating set for $G$.
\end{defn}

Furthermore, it is not difficult to see that two free groups $\Free[A]$ and $\Free[B]$ are isomorphic if and only if $\card A=\card B$; see~\Cref{cor: mateix cardinal} below. That is, the cardinal of $A$ is an algebraic invariant of $\Free[A]$; and we will see that this cardinal is precisely the rank of $\Free[A]$. Accordingly we will often denote by $\Free[\kappa]$ a (the) generic free group of rank $\kappa$.

\begin{rem}
Although most of them can be overcome by assuming the axiom of choice, in order to avoid set-theoretical subtleties, throughout this text we will assume that ${\card A=\kappa \leq \aleph_0}$. That is, we will consider free groups $\Free[\kappa] \isom \Free[A]$ of at most countable rank.
\end{rem}

\begin{rem}
Note that $\Free[\hspace{1pt}0]$ is the trivial group, and $\Free[\hspace{1pt}1]$ is isomorphic to the group~$\ZZ$ of integers. It is easy to see that these are the only abelian (and somewhat unrepresentative) free groups; see \eg \cite[Ch.\,1]{johnson_presentations_1997}.
\end{rem}


The universal property below turns out to be fundamental as it characterizes (group) freeness, and can consequently be taken as an alternative (categorical) definition for a free group (see \eg \parencite{bogopolski_introduction_2008}).

\begin{thm} \label{thm: free cat}
A group $F$ is free on $A \subseteq F$ if and only if every function from $A$ to an arbitrary group~$G$ extends to a unique  homomorphism $F\xto{} G$. \qed
\end{thm}

\begin{figure}[H]
\centering
\begin{tikzcd}
A \arrow[d, "\iota_A"', hook] \arrow[r, "\varphi"] & G \\ F\arrow[ru, "\exists !\,\, \widetilde{\varphi}"', dashed] &  
\end{tikzcd}
\caption{Categorical definition of a free group $F$}
\end{figure}

We note that the extendability in \Cref{thm: free cat} corresponds to the set $A$ being \defin[freely independent set]{freely independent} in~$F$ (\ie a basis of $\gen{A} \leqslant F$), and the uniqueness of the extension corresponds to $A$ generating the entire $F$. We say that a subgroup $H\leqslant \Free[A]$ is a \defin{free factor of $\Free[A]$}, and we write $H\leqff \Free[A]$, when it is generated by a subset of some basis of $\Free[A]$; that is, if and only if $H =K*L$, where $K,L\leqslant \Free[A]$. A \defin{proper free factor} is a free factor different from the full ambient group.

\medskip

The family of free groups is, of course, of great algebraic relevance, as it, in some sense, encapsulates all the information about every possible group. This is made precise through the following classic result, elementary to prove yet of fundamental importance in Group Theory.

\begin{thm} \label{thm: G = F/N}
Every group $G$ (of rank $\kappa$) is isomorphic to a quotient of a free group (of rank $\kappa$). That is, for every group $G$ there exists a cardinal $\kappa$ and a normal subgroup $N \normaleq \Free[\kappa]$ such that $G \isom \Free[\kappa]/N$. \qed
\end{thm}

We recall that a subgroup $H \leqslant G$ is said to be \defin[normal subgroup]{normal} in $G$, denoted by $H \normaleq G$, if $g^{-1} H g = H$ for every $g \in G$. For a subset $S \subseteq G$, the \defin{normal closure} of $S$ in~$G$, denoted by $\ncl{S}_G$ (or just $\ncl{S}$ if the ambient group is clear from the context) is the smallest normal subgroup of $G$ containing $S$; then we say that the elements in $S$ \defin[generate as a normal subgroup]{generate $\ncl{S}$ as a normal subgroup}. Finally, the \defin{normal rank} of a normal subgroup $N \normaleq G$, denoted by~$\nrk(N)$ is the smallest cardinality of a set $S$ such that $\ncl{S} = N$.

If, in \Cref{thm: G = F/N} we fix a basis for $\Free[\kappa]$ and a generating set for $N$ \emph{as a normal subgroup}, we reach the concept of presentation, also a cornerstone in Combinatorial, Geometric and Algorihmic Group Theory.

\begin{defn} \label{def: presentation}
Let $G$ be a group. A \defin[group presentation]{presentation} for $G$ is a pair $(A,R)$ such that $A$ is a set, $R \subseteq \Free[A]$, and $G\isom \Free[A] / \ncl{R}$. Then, 
we write $G = \pres{A}{R}$.
\end{defn}

From \Cref{thm: G = F/N}, it is clear that every group admits a (in fact infinitely many) presentation $G = \pres{A}{R}$, where $A$ is a generating set for the group $G$, and $R$ is a complete set of \defin[relations of a group]{relations}
for $G$. Accordingly, $A$ and $R$ are called the sets of \defin[generators of a group]{generators} and \defin[relators of a presentation]{relators} of the presentation $\pres{A}{R}$. Also, we say that $G$ is 
\defin[finitely generated group]{finitely generated} (f.g.)
if it admits a presentation with a finite number of generators; and
\defin[finitely presented group]{finitely presented} (f.p.) 
if it  admits a presentation  with a finite number of generators and relators.

\begin{rem}
    If $G = \pres{A}{R}$ and the context does not induce confusion, we usually abuse language and denote elements in $G$ by words in $(A^{\pm})^*$ representing them. That is, if $w \in (A^{\pm})^*$, we sometimes denote the element $w\ncl{R} \in G$ just by $w\in G$. In the same vein, if $S \subseteq (A^{\pm})^*$, we usually write $\gen{S}$ (or $\gen{S}_G$) to denote the subgroup of $G$ generated by the elements in $G$ represented by the words in $S$. 
\end{rem}

\begin{rem}
A group is free if and only if it admits a presentation with no relators. Concretely $\Free[A]$ admits the presentation $\pres{A}{-}$, which is called the \defin{standard presentation for $\Free[A]$}.
\end{rem}

\medskip

\section{Digraphs and automata}\label{sec: Digraphs and automata}

In this section, we present the basic graphical object  that is used throughout the text.

\begin{defn} \label{def: digraph}
A \defin{directed multigraph} (a \defin{digraph}, for short) is a tuple $\Dgri = (\Verts,\Edgs,\init,\term)$, where $\Verts =\Verts \Dgri$ is a nonempty set called the set of \defin[vertex]{vertices} of $\Dgri$, $\Edgs =\Edgs \Dgri$ is a set disjoint from $\Verts$ called the set of \defin[directed edge]{directed edges} or \defin[arc]{arcs} of $\Dgri$, and $\init$ and $\term$ are functions $\Edgs \xto{} \Verts$ called \defin[initial vertex]{initial} and \defin[terminal vertex]{terminal} \defin{incidence functions} of $\Dgri$, respectively. We say that $\Dgri$ is \defin[finite digraph]{finite} (\resp \defin[countable digraph]{countable}) if the cardinal $\card{(\Verts \Dgri \sqcup \Edgs \Dgri)}$ is so.
\end{defn}


If $\edgi \in \Edgs$ is an arc of $\Dgri$, then $\edgi \init$ is called the \defin[origin of an arc]{origin}  or \defin{initial vertex} of $\edgi$, and~$\edgi \term$ is called the \defin[end of an arc]{end} or \defin{terminal vertex} of~$\edgi$. Then, we also say that $\edgi$ is an arc from $\edgi \init$ to $\edgi \term$, and we write $\edgi \equiv \edgi\init \xarc{\ \,} \edgi \term$. Two arcs $\edgi,\edgii \in \Edgs$ are said to be \defin[reverse arcs]{reverse} of each other if $\edgi \init =\edgii \term$ and $\edgi \term = \edgii \init$.


\begin{rem}
Note that no restrictions apply to the incidence functions in \Cref{def: digraph}; in particular, we allow both the possibility of arcs having the same origin and end (called \defin[loop]{loops}), and of different arcs sharing the same respective origins and ends (called \defin{parallel arcs}).
\end{rem}

\begin{defn} \label{def: walk}
A \defin[walk in a digraph]{walk} in a digraph $\Dgri = (\Verts,\Edgs,\init,\term)$ is a finite alternating sequence of the form
\begin{equation} \label{eq: walk}
\walki =\verti_0 \edgi_1 \verti_1 \cdots \edgi_{l} \verti_{l} \, ,
\end{equation}
where $l\in \NN$ and, for every meaningful $i\in [0,l]$, $\verti_i \in \Verts \Dgri$, and $\edgi_i \in \Edgs\Dgri$ are such that $\edgi_i \init= \verti_{i-1}$ and $\edgi_i \term = \verti_{i}$. The number of occurrences (with possible repetitions) of arcs in $\walki =\verti_0 \edgi_1 \verti_1 \cdots \edgi_{l} \verti_{l}$, namely $l$, is called the \defin[walk length]{length} of $\walki$, denoted by $|\walki |=l$. 
\end{defn}

Note that walks of length $0$, called \defin[trivial walk]{trivial walks}, can be identified with vertices in $\Dgri$, whereas nontrivial walks are
characterized by the complete subsequence of arcs they contain. Accordingly, we write trivial walks simply as vertices, and, when $l\geqslant 1$, we sometimes abbreviate the walk \eqref{eq: walk} as $\edgi_1 \cdots \edgi_{l}$. 

The vertices $\verti_0$ and $\verti_l$ are the \defin[walk endpoints]{endpoints} of $\walki =\verti_0 \edgi_1 \verti_1 \cdots \edgi_{l} \verti_{l}$, called the \defin[origin of a walk]{origin} and the \defin[end of a walk]{end} of $\walki$, denoted by $\verti_0 = \walki \init$ and $\verti_l =\walki \term$ respectively; and we say that $\walki$ is a walk \defin[walk from $\verti_0$ to~$\verti_l$]{from $\verti_0$ to~$\verti_l$} (a $(\verti_0,\verti_l)$-walk), which is denoted by $\walki \equiv \verti_0 \xwalk{\ \,} \verti_{l}$.

The \defin{set of walks} in $\Dgri$ is denoted by \index{$\Walks(\Dgri)$}, and if $\Verti,\Vertii \subseteq \Verts \Dgri$, we denote by $\Walks_{\Verti,\Vertii}(\Dgri)$ (or simply by $\Walks_{\Verti \Vertii}$) the set of walks in~$\Dgri$ from a vertex in $\Verti$ to a vertex in $\Vertii$, \ie
 \begin{equation} \label{eq: Walks(P,Q)}
\Walks_{\Verti \Vertii}(\Dgri)
\,=\,
\Set{\walki \in \Walks(\Dgri) \st \walki \equiv \verti \xwalk{\ \,} \vertii \text{\  for } \verti \in \Verti \text{ and } \vertii\in \Vertii}.
 \end{equation}
If $\verti,\vertii \in \Verts \Dgri$ then we abbreviate
$\Walks_{\verti \vertii}(\Dgri) = \Walks_{\set{\verti}\set{\vertii}}(\Dgri)$. A walk whose initial and terminal vertices are equal (say, to $\verti$) is called a \defin{closed walk} or a \defin{$\verti$-walk}. The \defin{set of $\mathsf{p}$-walks} in $\vv{{\Gamma }}$ 
is (also) denoted by $\Omega _{\mathsf{p}}(\vv{{\Gamma }})$. 

If there exists a walk from $\verti$ to $\vertii$ (\ie if $\Walks_{\verti \vertii} \neq \varnothing$), then we write $\verti \xwalk{\ \,} \vertii$. 
A digraph $\Dgri$ is called \defin[strongly connected digraph]{strongly connected} if for every $\verti, \vertii \in \Verts \Dgri$, $\verti \xwalk{\ \,} \vertii$.

\begin{defn}
If $\walki$ and $\walkii$ are walks in a digraph $\Dgri$ such that $\walki \term = \walkii \init$, then the \defin[product of walks]{product} of $\walki$ and $\walkii$, denoted by $\walki \walkii$, is the walk (of length $|\walki|+|\walkii|$) obtained after concatenating $\walki$ and $\walkii$ in the natural way. It is clear that concatenation of walks is associative and, hence, $\Walks(\Dgri)$ is a semigroupoid (\wrt concatenation). Moreover, for every vertex $\verti \in \Verts \Dgri$, $\Walks_{\verti}(\Dgri)$ is a monoid with neutral element the trivial walk $\verti$.
\end{defn}

\subsection{Labeled digraphs}

\begin{defn} \label{defn: A-digraph}
Let $\Alfi$ be an alphabet. An \defin[labeled digraph]{$A$-labeled digraph} (an \defin{$\Alfi$-digraph}, for short) is
a tuple $\Ati = (\Dgri,\lab)$, were $\smash{\Dgri = (\Verts,\Edgs, \init, \term)}$ is a digraph (called the \defin{underlying digraph} of $\Ati$), and $\lab\colon \Edgs \xto{} \Alfi$.
We denote labeled digraphs by boldface capital Greek letters ($\Ati, \Atii,\Atiii, \ldots$).
\end{defn}

Note that if $\Ati =(\Dgri,\lab)$ is an $A$-digraph, then the arc-labelling $\lab$ extends naturally to a homomorphism (of semigroupoids) $\lab^* \colon \Walks(\Ati) \xto{} A^*$; 
we usually abuse language and write $\lab^* =\lab$ to lighten notation.
If $ (\walki)\lab =\wordi$, then we say that the walk $\walki$ \emph{reads} (or \emph{spells}) the word $\wordi$, and that the word~$\wordi$ \defin[label of a walk]{labels} the walk $\walki$. Furthermore, we write $\walki \equiv \verti \xwalk{\wordi} \vertii$ to express that $\walki$ is a walk from $\verti$ to $\vertii$ with label $(\walki)\lab = \wordi$, and $\verti \xwalk{\wordi} \vertii$ to express that the word $\wordi$ is readable as the label of some walk from $\verti$ to~$\vertii$.

\begin{defn}
Let $\Ati = (\Dgri,\lab)$ be an $A$-labeled digraph and let $\Verti,\Vertii \subseteq \Verts \Ati$. Then, 
 the set of labels of walks from $\Verti$ to $\Vertii$ in $\Ati$ is called the
\defin{language recognized by~$\Ati$ from $\Verti$ to $\Vertii$}, denoted by $\Lang_{\Verti \Vertii}(\Ati) = (\Walks_{\Verti \Vertii}(\Ati))\lab \subseteq A^*$. 

We extend to languages the notational conventions already introduced for walks; for example, $\Lang_{\verti}(\Ati) = \set{\wordi \in A^*\st \verti \xwalk{\wordi} \verti}\subseteq A^*$ is the \defin{language recognized by $\Ati$ at vertex $\verti$}.
\end{defn}
 
It is clear that,  
for every $\verti \in \Verts \Ati$, the walk labeling restricts to a surjective (not necessarily injective) homomorphism of monoids
 \begin{equation}
\begin{array}{rcl}
\lab \colon \Walks_{\verti}(\Ati) & \onto & \Lang_{\verti}(\Ati) \leqslant A^{*}. \\
\walki & \mapsto & (\walki)\lab
\end{array}
 \end{equation}




An arc with label $a \in A$ is called an \defin{$a$-arc}. We write $\edgi \equiv \verti \xarc{a} \vertii$ to express that $\edgi$ is an arc from $\verti$ to $\vertii$ with label $a \in A$. We denote by $\Edgs_{a}(\Ati)$ the set of $a$-arcs in~$\Ati$. Clearly, $\Edgs \Ati = \bigsqcup_{a\in A} \Edgs_a \Ati$.

The \defin{$a$-outdegree} (\resp \defin{$a$-indegree}) of a vertex~$\verti \in \Verts \Ati$, denoted by $\deg_a^+(\verti)$ (resp., by~$\deg_a^-(\verti)$) is the number of $a$-arcs with origin (\resp end) the vertex $\verti$; whereas the \defin[total $a$-degree]{(total) $a$-degree} of $\verti$ is $\deg^{\pm}_{a}(\verti)=\deg_{a}^{+}(\verti)+\deg_{a}^{-}(\verti)$. We define the \defin{indegree}, \defin{outdegree}, and \defin{total degree} of $\verti$ as 
$\deg^{-}(\verti) = \sum_{a \in \Alfi} \deg_{a}^{-}(\verti)$,
$\deg^{+}(\verti) = \sum_{a \in \Alfi} \deg_{a}^{+}(\verti)$,  
and $\deg^{\pm}(\verti) = 
\sum_{a \in \Alfi} \deg^{\pm}_{a}(\verti) =
\deg^{+}(\verti) + \deg^{-}(\verti) 
$.

A vertex $\verti$ in an $\Alfi$-digraph $\Ati$ is said to be \defin[saturated vertex]{saturated} if for every letter $a\in \Alfi$ there is at least one $a$-arc with origin $\verti$, \ie if $\deg_{a}^{+}(p) \geq 1$. Otherwise (\ie when $\deg_{a}^{+}(p)=0$), we say that $\verti$ is unsaturated (or \defin[deficient vertex]{$a$-deficient}, if we want to allude the missing label). The \defin{$a$-deficit} of $\Ati$, denoted by $\dfc_{a}(\Ati)$, is the number (cardinal) of $a$-deficient vertices in $\Ati$. 
An $\Alfi$-digraph is said to be \defin[saturated $\Alfi$-digraph]{saturated}\footnote{The term `complete' is standard in automata theory; we use the term `saturated' instead in order to avoid confusion with the notions of `complete graph' and `complete digraph'.} if all its vertices are so (that is if $\dfc_{a}(\Ati) = 0$, for all $a \in A$) and \defin[unsaturated $A$-digraph]{unsaturated} otherwise.

An $A$-digraph $\Ati$ is said to be \defin[deterministic vertex]{deterministic at a vertex} $\verti \in \Verts \Ati$ if no two arcs with the same label depart from $\verti$ (\ie if $\deg_{a}^{+}(\verti) \leq 1$ for every $a \in A$). An $A$-digraph is said to be \defin[deterministic $\Alfi$-digraph]{deterministic} if it is deterministic at every vertex; that is, if for every vertex $\verti \in \Verts \Ati$, and every pair of arcs $\edgi,\edgi'$ leaving $\verti$, $(\edgi)\lab = (\edgi') \lab$ implies $\edgi =\edgi'$; see~\Cref{fig: nondet}.

\begin{figure}[H]
\centering
  \begin{tikzpicture}[shorten >=1pt, node distance=.5cm and 1.75cm, on grid,auto,>=stealth']
   \node[state,semithick, fill=gray!20, inner sep=2pt, minimum size = 10pt] (1) {$\scriptstyle{\verti}$};
   \node[state] (2) [above right = of 1] {};
   \node[state] (3) [below right = of 1] {};

   \path[->]
        (1) edge[]
            node[pos=0.45,above] {$\alfi$}
            (2)
            edge[]
            node[pos=0.45,below] {$\alfi$}
            (3);
\end{tikzpicture}
\hspace{10pt}
\begin{tikzpicture}[shorten >=1pt, node distance=.5cm and 1.5cm, on grid,auto,>=stealth']
   \node[state,semithick, fill=gray!20, inner sep=2pt, minimum size = 10pt] (1) {$\scriptstyle{\verti}$};
   \node[state] (2) [right = 1.5 of 1] {};

   \path[->]
        (1) edge[loop above,min distance=12mm,in= 90-35,out=90+35]
            node[] {\scriptsize{$a$}}
            (1)
            edge
            node[pos=0.5,below] {$\alfi$}
            (2);
\end{tikzpicture} 
\hspace{15pt}
\begin{tikzpicture}[shorten >=1pt, node distance=.5cm and 2cm, on grid,auto,>=stealth']
   \node[state,semithick, fill=gray!20, inner sep=2pt, minimum size = 10pt] (1) {$\scriptstyle{\verti}$};
   \node[state] (2) [right = of 1] {};

   \path[->]
        (1) edge[bend right]
            node[pos=0.5,below] {$\alfi$}
            (2)
        (1) edge[bend left]
            node[pos=0.5,above] {$\alfi$}
            (2);
   \end{tikzpicture}
\hspace{15pt}
\begin{tikzpicture}[shorten >=1pt, node distance=.5cm and 1.5cm,
 on grid,auto,>=stealth']
   \node[state,semithick, fill=gray!20, inner sep=2pt, minimum size = 10pt] (1) {$\scriptstyle{\verti}$};

   \path[->]
        (1) edge[loop right,min distance=12mm,in=35,out=-35]
            node[] {\scriptsize{$a$}}
            (1)
             edge[loop left,min distance=12mm,in=180-35,out=180+35]
            node[] {\scriptsize{$a$}}
            (1);
\end{tikzpicture}   
\caption{Nondeterministic situations at vertex $\verti$}
\label{fig: nondet}
\end{figure}

\begin{rem} \label{rem: unique walk}
It is clear by induction that, if $\Ati$ is deterministic, then for every vertex $\verti$ in $\Ati$ and every word $\wordi\in \Alfi^*$ there is at most one walk in $\Ati$ reading $\wordi$ from $\verti$; we denote by $\verti \wordi$ its final vertex in case such a walk exists; otherwise, $\verti \wordi$ is undefined. In particular, if $\Ati$ is deterministic, then $\lab \colon \Walks_{\verti}(\Ati) \xto{} \Lang_{\verti}(\Ati)$ is an isomorphism of monoids.
\end{rem}

When two particular sets of vertices ---working as initial and terminal vertices by default--- are distinguished in a labeled digraph, it is usually called an automaton. In this context vertices are called \emph{states}\index{state of an automaton} and arcs are called \emph{transitions}\index{transition of an automaton}. We will focus on automata having a single vertex working both as initial and terminal vertex. 

\begin{defn}
A \defin[pointed $A$-automaton]{(pointed) $A$-automaton} is a pair $\Ati_{\!\verti} =(\Ati,\verti)$, where $\Ati$ is a labeled $A$-digraph and $\verti \in \Verts \Ati$ is a distinguished vertex, called the \defin[basepoint of an automaton]{basepoint} of $\Ati$, and usually denoted by $\bp$. The \defin{language recognized} by $\Ati_{\!\bp}$ is~$\Lang(\Ati_{\!\bp})=\Lang_{\bp}(\Ati)$. If the basepoint is clear from the context, we simply write~${\Ati_{\!\bp} =\Ati}$.
\end{defn}

If not stated otherwise, all the automata appearing in the rest of the \monograph\ are assumed to be pointed.

\subsection{Involutive labeled digraphs}

Recall that for a given alphabet $A$, the involutive closure of $A$ is the set $A^{\pm}$ consisting of the elements in $A$ and their formal inverses, \ie $A^{\pm} = A \sqcup A^{-1}$.


\begin{defn} \label{def: inv A-digraph}
An \defin{involutive $A$-digraph} is an $A^{\pm}$-digraph $\Ati =(\Verts,\Edgs,\init,\term,\lab)$ together with a map ${}^{-1} \colon \Edgs \xto{} \Edgs$ such that for every $\edgi \in \Edgs$,
 \begin{enumerate*}[ind]
\item $(\edgi^{-1})^{-1} = \edgi$, 
\item $\edgi^{-1}$ is reversed to $\edgi$, and 
\item\label{item: inv labs} $(\edgi^{-1})\lab =((\edgi)\lab)^{-1}$.
\end{enumerate*}
\end{defn}
Note that condition \ref{item: inv labs} implies that $\edgi^{-1} \neq \edgi$ for every arc $\edgi \in \Edgs$. Hence, the map ${}^{-1} \colon \Edgs \xto{} \Edgs$ is an involution without fixed points, which we call \defin{inversion of arcs} in $\Ati$. Accordingly, $\edgi^{-1}$ is called the \defin{inverse arc} of $\edgi$, for every $\edgi \in \Edgs$.

In summary,  in an involutive $A$-digraph $\Ati$, for every arc $\edgi \equiv \verti \xarc{a\,} \vertii$ (where $a\in A^{\pm}$) there exists a unique arc \smash{$\edgi^{-1} \equiv \verti \xcra{\ a^{-1}\!\!} \vertii$} (different from~$\edgi$) which we call the \emph{inverse} of $\edgi$, such that $(\edgi^{-1})^{-1} = \edgi$. More compactly, an involutive $A$-digraph
  is an $A^{\pm}$-digraph where
  arcs appear by mutually reversed pairs with mutually inverse labels:
  \vspace{3pt}
\begin{figure}[H]
\centering
\begin{tikzpicture}[shorten >=3pt, node distance=.3cm and 2cm, on grid,auto,>=stealth']
   \node[state] (0) {};
   \node[state] (1) [right = of 0] {};
   \path[->]
        ([yshift=0.4ex]0.east) edge[]
            node[pos=0.5,above] {$a$}
            ([yshift=0.3ex]1)
   ([yshift=-0.4ex]1.west) edge[dashed]
            node[pos=0.5,below] {$a^{-1}$}
            ([yshift=-0.3ex]0);
\end{tikzpicture}
\caption{Two inverse arcs in an involutive $A$-digraph}
\label{fig: involutive a-arc}
\end{figure}

\begin{rem}
If $\Ati$ is an involutive $A$-digraph then, for every $a \in A^{\pm}$ and every vertex $\verti \in \Verts \Ati$, $\deg_{a}^{+}(\verti) = \deg_{a^{-1}}^{-}(\verti)$.
\end{rem}

An arc in an involutive $A$-digraph $\Ati$ is called \defin[positive arc]{positive} (\resp \defin[negative arc]{negative}) if it is labeled by a letter in $A$ (\resp $A^{-1}$).  
We denote by~$\Edgs^{+} (\Ati)$ the set of positive arcs in~$\Ati$. The \defin[positive part of an automaton]{positive part} of $\Ati$, denoted by $\Ati^{+}$, is the $A$-digraph obtained after removing all the negative arcs from $\Ati$. It is clear that involutive digraphs are fully characterized by their positive part. Hence, we usually represent involutive $\Alfi$-digraphs
through their positive part,
with the convention that their (positive) arcs read the inverse label when crossed backwards. That is, if $\Ati$ is involutive, then $\Edgs \Ati =\bigsqcup_{\edgi \in \Edgs^+(\Ati)} \set{\edgi, \edgi^{-1}}$. We denote the corresponding equivalence relation by $\sim$.

\begin{exm}
The involutive $A$-digraph with one vertex and one (positive) $a$-loop for each letter $a \in A$ is called the \defin[bouquet]{$A$-bouquet}, denoted by $\bouquet_A$ or $\bouquet_n$, where $n = \card A$.
\end{exm}

\begin{figure}[H]
    \centering
    \begin{tikzpicture}[shorten >=1pt, node distance=.5cm and 1.5cm, on grid,auto,>=stealth']
   \node[state,accepting] (1) {};

   \path[->]
        (1) edge[blue,loop right,min distance=12mm,in=330+35,out=330-35]
            node[right] {\scriptsize{$a$}}
            (1)
            edge[red,loop right,min distance=12mm,in=90+35,out=90-35]
            node[above] {\scriptsize{$b$}}
            (1)
            edge[black,loop right,min distance=12mm,in=210+35,out=210-35]
            node[left] {\scriptsize{$c$}}
            (1);
\end{tikzpicture} 
\vspace{-15pt}
\caption{(The positive part of) $\bouquet_{3}$, a threefold bouquet}
\label{fig: bouquet}
\end{figure}
\vspace{-5pt}

\begin{defn}
Let $\Ati =(\Verts,\Edgs, \init,\term,\lab)$ be an involutive digraph. The \defin{underlying (undirected) graph} of~$\Ati$
is the undirected graph obtained after removing the labeling, and identifying every pair of mutually inverse arcs in $\Ati$ into a single edge with the same endpoints. Formally, 
it
is the undirected graph with vertex set $
\Verts \Ati$, edge set $
\Edgs\Ati /{\sim}$, and incidence map
$\nu\colon \Edgs/{\sim} \xto{} V \cup \binom{V}{2}$, $\set{\edgi,\edgi^{-1}} \mapsto \set{\edgi \init,\edgi \term}$. Note that~$\card (\Edgs\Ati /{\sim}) = \frac{1}{2} \card\Edgs{\Ati} = \card \Edgs^{+} \Ati$.
\end{defn}

\begin{rem} \label{rem: undirected graphs}
Note that every undirected graph is the underlying graph of some involutive automaton. That is, `\emph{unlabeled} involutive digraph' can be seen as an alternative definition for `undirected graph' according to this interpretation.
\end{rem}

Involutive digraphs inherit terminology from its underlying graph; for example, we say that an involutive (sub-)digraph is \defin[connected automaton]{connected}, \defin[vertex-transitive automaton]{vertex-transitive}, \defin[regular automaton]{($k$-)regular}, a \emph{path}, a \emph{cycle}, a \emph{tree}, or a \emph{spanning tree}\index{spanning tree (involutive digraph)} if its underlying graph is so. Similarly, the \defin[undirected degree]{(undirected) degree} of a vertex $\verti$ in an involutive digraph $\Ati$, denoted by $\deg_{\Ati}(\verti)$ (or just by $\deg(\verti)$ if $\Ati$ is clear) is defined to be its degree in the underlying graph; that is,~$\deg(\verti) = \frac{1}{2} \deg^{\pm}(\verti) = \deg^{+}(\verti) = \deg^{-}(\verti)$.

Finally, we define the \defin[rank of a graph]{rank} of a connected involutive automaton $\Ati$, denoted by~$\rk(\Ati)$, to be the rank of its underlying graph (\ie the number of edges outside any spanning tree). In particular, if $\Ati$ is finite, then $\rk (\Ati) = 1-\card \Verts\Ati + \card \Edgs^+\Ati$.

For the rest of the section we shall assume that $\Ati 
$ is an involutive $A$-digraph.

\begin{defn}\label{def: inverse walk}
Let $\walki =\verti_0 \edgi_1 \verti_1 \cdots \verti_{l-1} \edgi_l \verti_l$ be a walk in $\Ati$. Then, the \defin[inverse of a walk]{inverse} of $\walki$ is the walk $\walki^{-1} =\verti_l \edgi_l^{-1} \verti_{l-1} \cdots \verti_1 \edgi_1^{-1}\verti_0$. It is clear that $|\walki| =|\walki^{-1}|$ and $(\walki^{-1})\lab =((\walki) \lab)^{-1}$.
\end{defn}

\begin{defn}\label{def: reduced walk}
An \defin{elementary backtracking} in $\Ati$ is a walk consisting of two successive arcs inverse of each other; that is, a walk of the form $\edgi \edgi^{-1} $ where $\edgi \in \Edgs \Ati$. A walk is said to be \defin[reduced walk]{reduced} if it does not contain any backtracking. We use the notation $\walki \equiv \verti \xrwalk{\ \,} \vertii$ to denote that $\walki$ is a reduced walk from $\verti$ to $\vertii$.
A closed walk is said to be \defin[cyclically reduced walk]{cyclically reduced} if it is trivial or it is reduced and $\edgi_l \neq \edgi_0^{-1}$.
\end{defn}

\begin{rem}
Note that every backtracking in $\bm{{\Gamma }}$ reads a cancellation in
$(A^{\pm})^{*}$. Hence, if the label $(\omega )\ell $ is a reduced word
then the walk $\omega $ is also reduced. However, the converse is not necessarily true; in a nondeterministic environment, consecutive arcs that are not inverses of each other may still have inverse labels.
\end{rem}

\begin{defn}
Two walks $\walki,\walkii$ in $\Ati$ are said to be \defin[equivalent walks]{equivalent (modulo backtracking)}, denoted by $\walki \equiv \walkii$, if one can be obtained from the other after a finite sequence of backtracking insertions and removals. It is immediate to check that $\equiv$ is an equivalence relation in $\Walks(\Ati)$ preserving the endpoints.
\end{defn}

\begin{lem}\label{prop: unicitat cami reduit}
Each equivalence class $[\walki ]\in \Walks(\Ati)/\equiv$ contains one and only one reduced walk, denoted by~$\red{\walki}$, and called the \defin[reduction of a walk]{reduction} of the walks in $[\walki]$.
\end{lem}

\begin{proof}
By successively applying backtracking removals to any representative, it is clear that every equivalence class contains a reduced walk. To see uniqueness, suppose that $\walki$ and $\walki'$ are two reduced, equivalent walks, and let us show $\walki=\walki'$. Let $\walki =\walki_0 \leftrightsquigarrow
\walki_1 \leftrightsquigarrow \cdots \leftrightsquigarrow \walki_{n-1} \leftrightsquigarrow \walki_n =\walki'$ be a finite sequence of backtracking insertions and removals minimizing the total sum of lengths $N=\sum_{i=0}^n |\walki_i |$. 

Suppose that $n\geq 2$. Since $|\walki_0 |<|\walki_1 |$ and $|\walki_{n-1} |>|\walki_n |$, there must exist $j\in [1,n-1]$ such that $|\walki_{j-1} |<|\walki_j |>|\walki_{j+1} |$. Let us focus at the backtracking removals $\walki_{j-1} \leftsquigarrow \walki_j \rightsquigarrow \walki_{j+1}$ and at the occurrences of arcs being canceled along them. If there is one (or two) in common, then $\walki_{j-1} =\walki_{j+1}$ contradicting the minimality of $N$. Hence, these two pairs of occurrences are disjoint, \ie $\walki_j =\walkii \edgi\edgi^{-1}\walkii' \edgii\edgii^{-1}\walkii''$, for certain walks $\walkii,\walkii',\walkii''\in \Walks(\Ati)$ and certain arcs $\edgi,\edgii\in \Edgs\Ati$; now, replacing $\walki_{j-1} \leftsquigarrow \walki_j \rightsquigarrow \walki_{j+1}$ to either $\walki_{j-1} =\walkii \walkii' \edgii\edgii^{-1}\walkii'' \rightsquigarrow \walkii \walkii' \walkii'' \leftsquigarrow \walkii \edgi \edgi^{-1}\walkii'\walkii'' =\walki_{j+1}$ or $\walki_{j-1}=\walkii \edgi\edgi^{-1}\walkii'\walkii'' \rightsquigarrow \walkii\walkii'\walkii''\leftsquigarrow \walkii\walkii' \edgii\edgii^{-1}\walkii'' =\walki_{j+1}$ in the sequence above, we would reduce the value of $N$, contradicting again its minimality.

Hence, $n\leq 1$. But the case $n=1$ is not possible because both $\walki$ and $\walki'$ are reduced; therefore, $n=0$ and $\walki=\walki'$ as we wanted to see.  
\end{proof}




We denote by $\red{\Walks}(\Ati)$  the \defin{set of reduced walks}  in $\Ati$, and we extend the notation to $\red{\Walks}_{\Verti\Vertii}(\Ati)$, $\red{\Walks}_{\verti\vertii}(\Ati)$, $\red{\Walks}_{\verti}(\Ati)$, etc., in the natural way. It is clear that equivalence of walks is compatible with concatenation; this defines a groupoid in $\Walks(\Ati)/{\equiv}$, called the \defin{fundamental groupoid} of $\Ati$. The uniqueness of reduced walks provides a natural isomorphism between $\Walks(\Ati)/{\equiv}$ and $\red{\Walks}(\Ati)$ with the operation of `concatenation and reduction'.  
Restricting the fundamental groupoid to walks closed at a certain vertex, we recover the classical notion of fundamental group.

\begin{defn}
Let $\Ati$ be an involutive digraph, and let $\verti \in \Verts \Ati$. The \defin{fundamental group} of $\Ati$ at $\verti$ is $\pi_{\verti}(\Ati) =\red{\Walks}_{\verti}(\Ati) \isom \Walks_{\verti}(\Ati)/{\equiv}$.
\end{defn}

If $\walkii \equiv \verti \xrwalk{\ } \vertii$ is a reduced walk in $\Ati$, then the map $\pi_{\verti}(\Ati) \xto{} \pi_{\vertii}(\Ati), \ \walki \mapsto \walkii^{-1} \walki \walkii$ is clearly a group isomorphism. So, for a connected $\Ati$, the isomorphism class of $\pi_{\verti}(\Ati)$ does not depend on the vertex $\verti$.


It is also clear by construction that the fundamental group $\pi_{\verti}(\Ati)$ only involves (walks in) the connected component of $\Ati$ containing $\verti$. More precisely, only vertices visited by reduced $\verti$-walks contribute to the fundamental group $\pi_{\verti}(\Ati)$; this motivates the following definition. 

\begin{defn}
An involutive digraph $\Ati$ is said to be \defin[coreness \wrt a vertex]{core \wrt a vertex~$\verti$} (or \mbox{\defin[$\verti$-coreness]{$\verti$-core}}) if every vertex ---and so, every arc--- in $\Ati$ appears in some reduced \mbox{$\verti$-walk}. The (pointed) \defin[pointed core of a digraph]{$\verti$-core} of $\Ati$, denoted by $\core_{\verti}(\Ati)$, is the maximal $\verti$-core subdigraph of~$\Ati$ (as a pointed automaton with basepoint $\verti$). Similarly, $\Ati$ is said to be a \defin[strict coreness]{strict core} if every vertex in $\Ati$ appears in some nontrivial cyclically reduced walk in $\Ati$. The \defin{strict core} of $\Ati$, denoted by $\core^{*}(\Ati)$, is the maximal strict core subdigraph of~$\Ati$. Hence $\Ati$ is a $\verti$-core (\resp a strict core) if and only if $\Ati =\core_{\verti}(\Ati)$ (\resp $\Ati = \core^{*}(\Ati)$). Note that $\core_{\verti}(\Ati)$ is a pointed automaton whereas $\core^{*}(\Ati)$ is just an involutive digraph with no distinguished vertices (and independent from the eventual basepoint in~$\Ati$).
\end{defn}

If $\Ati$ can be obtained by identifying a vertex $\verti$ of some tree $T$ with a vertex of some involutive digraph $\Atii$ disjoint with $T$, then we say that $T$ is a \defin{hanging tree} of $\Ati$, and  $\Atii$ is obtained after removing the hanging tree $T$ from $\Ati$. Note that for every $\verti \in \Verts(\Ati)$, $\core^*(\Ati) \subseteq \core_{\verti}(\Ati)$; if $\Ati$ is a forest then $\core^*(\Ati)$ is empty, and otherwise, 
$\core^{*}(\Ati)$ (\resp $\core_{\verti}(\Ati)$) is obtained after removing from $\Ati$ all the hanging trees (\resp all the nontrivial hanging trees not containing~$\verti$). It follows that the $\verti$-core of $\Ati$ consists of the strict core $\core^{*}(\Ati)$ with a (possibly trivial) hanging tree ending at $\verti$; moreover, if $\Ati$ is connected then
the fundamental groups of $\Ati$, $\core_{\verti}(\Ati)$, and $\core^{*}(\Ati)$ (if nonempty) at any vertex are isomorphic.

\begin{rem}
If $\Ati$ is finite, then $\Ati$ is a strict core (\resp a $\verti$-core) if and only if it has no vertices of degree one (\resp no vertices of degree one except maybe $\verti$); in this case, $\core^{*}(\Ati)$ (\resp $\core_{\verti}(\Ati)$) can be obtained from~$\Ati$ by successively eliminating the vertices of degree one (\resp the vertices of degree one different from~$\verti$). In particular, the strict core (\resp the $\verti$-core) of a finite digraph is always computable. 
\end{rem}

If $\Ati$ is a pointed automaton and the basepoint is clear from the context, we will just write $\core(\Ati) =\core_{\bp}(\Ati)$ and call it the \defin[core of an automaton]{core} of $\Ati$; in this case, we say that $\Ati$ is \defin[coreness of an automaton]{core} instead of $\bp$-core.

The crucial theorem below states that a precise algebraic description of the fundamental group $\pi_{\verti}(\Ati)$ can be obtained in terms of any spanning tree of (the $\verti$-core of) the connected component of $\Ati$ containing $\verti$. It is easy to see that if $T$ is a spanning tree of $\Ati$ then, for every pair of vertices $\verti,\vertii \in \Verts \Ati$, there is a unique reduced walk $\verti \xrwalk{\ }\vertii$ using only arcs in $T$; we denote it by $\verti \xrwalk{\scriptscriptstyle{T}} \vertii$ or by $T[\verti,\vertii]$. Then, for every arc $\edgi \in \Edgs\Ati \setmin \Edgs T$, we call $\walki_{\edgi}=\verti \xrwalk{\scriptscriptstyle{T}} \! \bullet \! \xarc{\,\edgi\ } \!\bullet\! \xrwalk{\scriptscriptstyle{T}} \verti$ the \defin{$T$-petal} associated to $\edgi$. Note that $\walki_{\edgi^{-1}}=\walki_{\edgi}^{-1}$.

\begin{thm}\label{thm: B_T}
Let $\Ati$ be an involutive connected digraph,\footnote{Note that this is nothing more than a standard undirected graph, according to the interpretation in \Cref{rem: undirected graphs}.} let $\verti \in \Verts\Ati$, and let $T$ be a spanning tree of $\Ati$. Then, $\mathcal{B}_T=\{\omega_{\edgi} \st \edgi\in \Edgs^+\Ati \setmin \Edgs T\}$ is a free basis for $\pi_{\verti}(\Ati)$. In particular, the fundamental group $\pi_{\verti}(\Ati)$ is a free group of rank $\rk(\pi_{\verti}(\Ati)) = \rk(\Ati)$.
\end{thm}

\begin{proof}
Let $\walki\in \pi_{\verti}(\Ati)$ be a reduced walk. Highlighting the visits of $\walki$ to the arcs in $\Edgs\Ati \setmin \Edgs T$, we can write  
\vspace{5pt}
 \begin{equation*}
\walki \equiv \verti \xrwalk{T} \!\!\bullet\!\!  \xarc{\raisebox{0.6ex}{$\scriptstyle{\edgi_1^{\varepsilon_1}}$}} \!\!\bullet\!\! \xrwalk{T}  \!\!\bullet\!\!  \xarc{\raisebox{0.6ex}{$\scriptstyle{\edgi_2^{\varepsilon_2}}$}} \!\!\bullet\!\! \xrwalk{T} \!\!\bullet \ \cdots \ \bullet\!\! \xrwalk{T} \!\!\bullet\!\! \xarc{\raisebox{0.6ex}{$\scriptstyle{\edgi_l^{\varepsilon_l}}$}}\!\!\bullet\!\! \xrwalk{T} \verti,
\vspace{5pt}
 \end{equation*}
where $l\geq 0$, $\edgi_1,\ldots ,\edgi_l \in \Edgs^+\Ati \setmin \Edgs T$ (with possible repetitions), $\varepsilon_j =\pm 1$, and where the tree subwalks $\bullet\!\xrwalk{T} \!\bullet$ may be trivial. Observe that, up to reduction, $\walki$ coincides with  
\vspace{5pt}
 \begin{equation*} \label{eq: cami ampliat}
\verti \! \xrwalk{T} \!\! \bullet \!\! \xarc{\raisebox{0.6ex}{$\scriptstyle{\edgi_1^{\varepsilon_1}}$}} \!\!\bullet \!\! \xrwalk{T} \verti \xrwalk{T}\!\! \bullet \!\! \xarc{\raisebox{0.6ex}{$\scriptstyle{\edgi_2^{\varepsilon_2}}$}} \!\!\bullet\!\! \xrwalk{T} \verti \ \cdots \ \verti \xrwalk{T} \!\!\bullet\!\! \xarc{\raisebox{0.6ex}{$\scriptstyle{\edgi_l^{\varepsilon_l}}$}}\!\!\bullet\!\! \xrwalk{T} \verti
\vspace{5pt}
 \end{equation*}
and so, $\walki \equiv \walki_{\edgi_1}^{\varepsilon_1} \walki_{\edgi_2}^{\varepsilon_2} \cdots \walki_{\edgi_l}^{\varepsilon_l}$; this shows that $\mathcal{B}_T$ generates $\pi_{\verti}(\Ati)$. 

To see that $\mathcal{B}_T$ is free, consider a non-trivial product $\walki =\walki_{\edgi_1}^{\varepsilon_1} \walki_{\edgi_2}^{\varepsilon_2} \! \cdots \walki_{\edgi_l}^{\varepsilon_l}$, where $l\geq 1$, $\edgi_1,\ldots ,\edgi_l \in \Edgs^+\Ati \setmin \Edgs T$, and $\varepsilon_1,\ldots ,\varepsilon_l =\pm 1$; assuming it is reduced, \ie $\edgi_i=\edgi_{i+1}$ implies $\varepsilon_i=\varepsilon_{i+1}$, we have to see that $\walki\neq 1$. In fact, 
\vspace{5pt}
 \begin{align*}
\walki & =\walki_{\edgi_1}^{\varepsilon_1} \walki_{\edgi_2}^{\varepsilon_2} \! \cdots \walki_{\edgi_l} ^{\varepsilon_l} \notag \\[3pt]
& =\verti \! \xrwalk{T} \!\! \bullet \!\! \xarc{\raisebox{0.5ex}{$\scriptstyle{\edgi_1^{\varepsilon_1}}$}} \!\!\bullet \!\! \xrwalk{T} \verti \xrwalk{T}\!\! \bullet \!\! \xarc{\raisebox{0.5ex}{$\scriptstyle{\edgi_2^{\varepsilon_2}}$}} \!\!\bullet\!\! \xrwalk{T} \verti \ \cdots \ \verti \xrwalk{T} \!\!\bullet\!\! \xarc{\raisebox{0.5ex}{$\scriptstyle{\edgi_l^{\varepsilon_l}}$}}\!\!\bullet\!\! \xrwalk{T} \verti \notag \\[3pt] & =\verti \xrwalk{T} \!\!\bullet\!\! \xarc{\raisebox{0.5ex}{$\scriptstyle{\edgi_1^{\varepsilon_1}}$}} \!\!\bullet\!\! \xrwalk{T} \!\!\bullet\!\! \xarc{\raisebox{0.5ex}{$\scriptstyle{\edgi_2^{\varepsilon_2}}$}} \!\!\bullet\!\! \xrwalk{T} \!\!\bullet \ \cdots \ \bullet\!\! \xrwalk{T}  \!\!\bullet\!\! \xarc{\raisebox{0.5ex}{$\scriptstyle{\edgi_l^{\varepsilon_l}}$}}\!\!\bullet\!\! \xrwalk{T} \verti
  \end{align*}
  
\vspace{5pt}
\noindent
is non-trivial because this last walk is nonempty ($l\geq 1$) and reduced: the arcs $\edgi_i^{\varepsilon_i}$ lie outside~$T$ and, by hypothesis, any consecutive two (\ie when the corresponding subwalk $\bullet\! \xrwalk{\scriptscriptstyle{T}} \!\bullet$ is trivial) do not present backtracking. So, $\mathcal{B}_T$ is a free basis for $\pi_{\verti}(\Ati)$. In particular, $\pi_{\verti}(\Ati)$ is a free group of rank $\card \mathcal{B}_T =\rk(\Ati)$. 
\end{proof}

\begin{cor}\label{cor: incl => ff}
Let $\Ati$ and $\Atii$ be involutive digraphs. If $\Ati$ is a subdigraph of~$\Atii$ and $\verti \in \Verts \Ati$, then $\pi_{\verti}(\Ati)$ is a free factor of $\pi_{\verti}(\Atii)$.
\end{cor}

\begin{proof}
Take a maximal tree $T$ of $\Ati$ and extend it to a maximal tree $T'$ of $\Atii$. Clearly, the basis $\mathcal{B}_{T'}$ for $\Atii$ extends the basis $\mathcal{B}_T$ for $\Ati$ (see~\Cref{thm: B_T}) and so, $\pi_{\verti}(\Ati)$ is a free factor of $\pi_{\verti}(\Atii)$.
\end{proof}

\begin{rem}
The converse of \Cref{cor: incl => ff} is far from true, 
given that the number of involutive subdigraphs of a finite involutive digraph is finite, whereas the number of free factors of a non-cyclic free group is not.
\end{rem}

\begin{defn}\label{def: recognized}
The \defin{reduced label} of a walk $\walki \in \Walks(\Ati)$ is $(\walki )\rlab =\red{(\walki) \lab} \in \Free[A]$. We write $\red{\Lang}_{\Verti \Vertii} (\Ati) = \red{\Lang_{\Verti \Vertii} (\Ati)}$, $\red{\Lang}_{\verti \vertii} (\Ati) = \red{\Lang_{\verti \vertii} (\Ati)}$, $\red{\Lang}_{\verti} (\Ati) = \red{\Lang_{\verti} (\Ati)}$, etc.

As seen below, for any vertex $\verti \in \Verts\Ati$, the set $\red{\Lang}_{\verti}(\Ati)$ of reduced labels of $\verti$-walks in $\Ati$ is a subgroup of $\Free[A]$, which we call the \defin{subgroup recognized by $\Ati$ at~$\verti$}, and we denote by $\gen{\Ati}_{\verti}$; that is,
 \begin{equation}
\gen{\Ati}_{\verti} \,=\, \red{\Lang}_{\verti}(\Ati) \,=\, \set{(\walki )\rlab \st \walki \in \Walks_{\verti}(\Ati) }\,\leqslant\, \Free[A].
 \end{equation}
If $\Ati$ is an involutive pointed $A$-automaton with basepoint $\bp$, then we say that $\gen{\Ati} = \gen{\Ati}_{\bp}$ is the \defin[recognized subgroup]{subgroup recognized} by $\Ati$.
\end{defn}

Below, we see that certain sets of reduced labels admit a neat algebraic description, as it is straightforward to check.

\begin{lem} \label{lem: recognized}
Let $\Ati$ be a connected involutive $A$-digraph, and let $\verti,\vertii \in \Verts \Ati$. Then,
 \begin{enumerate}[ind]
\item the set $\gen{\Ati}_{\verti} =\red{\Lang}_{\verti}(\Ati)$ is a subgroup of $\Free[A]$;
\item \label{item: coset recognized} the set $\red{\Lang}_{\verti \vertii}(\Ati)$ of reduced labels of $(\verti,\vertii)$-walks in $\Ati$ is a right-coset of $\gen{\Ati}_{\verti}$ and a left-coset of $\gen{\Ati}_{\vertii}$; more precisely, $\red{\Lang}_{\verti \vertii}(\Ati) \,=\, \gen{\Ati}_{\verti} \, u_{\verti \vertii} \,=\, u_{\verti \vertii} \, \gen{\Ati}_{\vertii}$, for any $u_{\verti \vertii} \in \red{\Lang}_{\verti \vertii}(\Ati)$;
\item the subgroup $\gen{\Ati}_{\vertii}$ recognized by $\Ati$ at $\vertii$ is a conjugate of the subgroup $\gen{\Ati}_{\verti}$ recognized by $\Ati$ at $\verti$; more precisely: $\gen{\Ati}_{\vertii} \,=\, u_{\verti \vertii}^{-1} \, \gen{\Ati}_{\verti} \, u_{\verti \vertii}$, for any $u_{\verti \vertii} \in \red{\Lang}_{\verti \vertii}(\Ati)$. \qed
\end{enumerate}
\end{lem}


The relation between the fundamental group $\pi_\verti(\Ati)$ and the recognized subgroup $\gen{\Ati}_{\verti}$ of an involutive $A$-digraph $\Ati$ at vertex $\verti$ is encapsulated in the following natural group epimorphism:
 \begin{equation} \label{eq: pi onto gen}
\begin{array}{rcl}
\red{\lab}\colon \pi_{\verti}(\Ati) & \onto & \gen{\Ati}_{\verti} \leqslant \Free[A] \\ \walki & \mapsto & \red{(\walki) \lab}.
\end{array}
 \end{equation}

An immediate consequence of the homomorphism \eqref{eq: pi onto gen} is that the subgroup recognized by an automaton does not change after taking the core.

\begin{cor} \label{cor: <core> = <ati>}
For any involutive $A$-digraph $\Ati$ and any vertex $\verti \in \Verts \Ati$, we have that $\gen{\core_{\verti}(\Ati)}_{\verti} = \gen{\Ati}_{\verti}$. \qed
\end{cor}

Observe that, by \Cref{rem: unique walk}, this homomorphism is injective whenever $\Ati$ is deterministic. A description of the kernel of $\red{\lab}$ in the general situation will be obtained in the next section with the help of Stallings' automata (see \Cref{cor: ker rlab}).

\begin{lem}\label{lem: deterministic gamma}
If $\Ati$ is deterministic then $\red{\lab} \colon \pi(\Ati) \xto{} \gen{\Ati}$ is an isomorphism of groups; in particular, $\gen{\Ati}$ is free and $\rk (\gen{\Ati}) =\rk(\pi(\Ati)) =\rk(\Ati)$. \qed
\end{lem}

Finite, deterministic, involutive automata exhibit certain rigidity that will be later exploited algebraically.

\begin{lem} \label{lem: = defc}
If an involutive $A$-digraph $\Ati$ is deterministic and finite then, for every $a \in A$, $\dfc_{a}(\Ati) = \dfc_{a^{-1}}(\Ati)$.
\end{lem}

\begin{proof}
Note that if $\Ati$ is involutive and deterministic then, for every vertex $\verti \in \Verts \Ati$, $0\leq \deg_{a}^+(\verti) = \deg_{a^{-1}}^-(\verti) \leq 1$. Since $\Ati$ (and hence $\Verts \Ati$) is finite,
 \[
\dfc_{a}(\Ati)
\,=\,
\card\Verts \Ati - \sum_{\verti \in \Verts{\Ati}}
\deg_{a}^+(\verti)
\,=\,
\card\Verts \Ati - \sum_{\verti \in \Verts{\Ati}}
\deg_{a^{-1}}^-(\verti)
\,=\,
\dfc_{a^{-1}}(\Ati). \tag*{\qedhere}
 \]
\end{proof}

Note that the finiteness condition is essential in~\Cref{lem: = defc}, as the following example shows.
\begin{figure}[H] 
  \centering
  \begin{tikzpicture}[shorten >=1pt, node distance=1.2 and 1.2, on grid,auto,>=stealth']
   \node[state] (0) {};
   \node[state] (1) [right = of 0]{};
   \node[state] (2) [right = of 1]{};
   \node[state] (3) [right = of 2]{};
   \node[state] (4) [right = of 3]{};
   \node[] (5) [right = 0.8 of 4]{};
   \node[] (dots) [right = 0.2 of 5]{$\cdots$};

   \path[->]
        (0) edge[loop above,red,min distance=10mm,in=55,out=125]
            node[left = 0.1] {\scriptsize{$b$}}
            (0)
            edge[blue]
            node[below] {\scriptsize{$a$}}
            (1);

    \path[->]
        (1) edge[loop above,red,min distance=10mm,in=55,out=125]
            node[left = 0.1] {}
            (1)
            edge[blue]
            (2);

    \path[->]
        (2) edge[loop above,red,min distance=10mm,in=55,out=125]
            (2)
            edge[blue]
            (3);

    \path[->]
        (3) edge[loop above,red,min distance=10mm,in=55,out=125]
            (3)
            edge[blue]
            (4);
    \path[->]
        (4) edge[loop above,red,min distance=10mm,in=55,out=125]
            (4)
            edge[blue]
            (5);
            
\end{tikzpicture}
\caption{An infinite deterministic automaton $\Ati$ with $\defc[\alfi]{\Ati}=0$ but $\defc[\alfi^{\text{-}1}]{\Ati}=1$}
\label{fig: Infinte automaton with different deficit}
\end{figure}



We conclude this section by introducing the concept that will allow us to bijectively link $A$-automata to subgroups of $\Free[A]$.

\begin{defn}
A pointed involutive $A$-digraph $\Ati$ is said to be \defin[reduced automaton]{reduced} if it is deterministic and core. We denote by $\red{\Atts}_{\hspace{1pt}\bp}(A)$ the set of reduced pointed involutive $A$-automata. 
\end{defn}

\subsection{Homomorphisms of digraphs}

\begin{defn} \label{def: hom digraphs}
Let $\Ati =(\Dgri,\lab, {}^{-1})$ and $\Ati' =(\Dgri',\lab',{}^{-1})$ be two involutive \mbox{$A$-digraphs}. A \defin[homomorphism of involutive $A$-digraphs]{homomorphism (of involutive $A$-digraphs)} between $\Ati$ and $\Ati'$, denoted by ${\varphi \colon \Ati \xto{} \Ati'}$, is a pair of maps $\varphi =(\varphi_{_\Verts}\colon \Verts\Ati \xto{} \Verts\Ati',\varphi_{_\Edgs}\colon \Edgs\Ati \xto{} \Edgs\Ati')$ such that, for every $\edgi \in \Edgs \Ati$:
 \begin{enumerate*}[ind]
\item $\edgi\varphi_{_\Edgs} \init =\edgi \init \varphi_{_\Verts}$ and $\edgi\varphi_{_\Edgs} \term =\edgi\term \varphi_{_\Verts}$,
\item $(\edgi \varphi_{_\Edgs})\lab' =(\edgi)\lab$, and
\item $(\edgi^{-1}) \varphi_{_\Edgs}=(\edgi\varphi_{_\Edgs})^{-1}$.
 \end{enumerate*}
A \defin[homomorphism of pointed $A$-automata]{homomorphism of (pointed) $A$-automata} ${\varphi \colon \Ati_{\!\bp} \xto{} \Ati'_{\bp'}}$ is a homomorphism of the underlying $A$-digraphs preserving the basepoints, $\bp \varphi_{_\Verts} = \bp'$.
\end{defn}


\begin{rem}
Note that if $\Ati'$ is deterministic, then any homomorphism of $A$-digraphs $\varphi \colon \Ati \xto{} \Ati'$ is characterized by its restriction $\varphi_{_\Verts} \colon \Verts \Ati \xto{} \Verts\Ati'$, which we will identify with the homomorphism $\varphi =(\varphi_{_\Verts}
, \varphi_{_\Edgs})$ itself.   
\end{rem}

A homomorphism $\varphi\colon \Ati \xto{} \Ati'$ of $A$-digraphs extends naturally to the level of walks, $\varphi \colon \Walks(\Ati) \xto{} \Walks(\Ati')$, $\walki =\verti_0 \edgi_1 \verti_1 \cdots \edgi_l \verti_l \mapsto \walki \varphi =(\verti_0 \varphi )(\edgi_1 \varphi )(\verti_1\varphi )\cdots (\edgi_l\varphi )(\verti_l\varphi)$; and to the level of fundamental groupoids, $\varphi^{*} \colon \red{\Walks}(\Ati) \xto{} \red{\Walks}(\Ati')$, $\walki \mapsto \red{\walki \varphi}$. Of course, $\varphi^{*}$ is a groupoid homomorphism commuting with the labeling maps, $(\walki\varphi^{*})\red{\lab'} =(\walki)\rlab$, $\forall \walki \in \pi(\Ati)$. Moreover, restricting to closed walks at a certain vertex $\verti \in \Verts \Ati$, we obtain a group homomorphism $\varphi_{\verti} \colon \pi^*_{\verti}(\Ati) \xto{} \pi_{\verti \varphi}(\Ati')$ between the corresponding fundamental groups. 
 

Observe that if $\varphi\colon \Ati \xto{} \Ati'$ is a homomorphism of pointed involutive $A$-automata then $\gen{\Ati}\leqslant \gen{\Ati'}$ and the following diagram (of group homomorphisms), where $\iota$ denotes the natural inclusion, commutes.
\vspace{5pt}
\begin{equation} \label{eq: label comm}
\begin{tikzcd}[column sep=20pt,ampersand replacement=\&]
\pi(\Ati) \arrow[r, "\varphi^{*}"] \arrow[d, "\red{\lab}"', ->>] \& \pi(\Ati') \arrow[d, "\red{\lab'}", ->>] \\ \gen{\Ati} \arrow[r, hook,"\iota"] \& \gen{\Ati'}
\end{tikzcd}
 \end{equation}
Note, however, that the inclusion $\gen{\Ati} \leqslant \gen{\Ati'}$ does not guarantee the existence of a homomorphism $\varphi\colon \Ati \xto{} \Ati'$. This converse implication will be true under certain extra conditions for the involved automata. 

\begin{prop} \label{prop: main prop}
Let $\Ati$ and $\Ati'$ be reduced pointed $A$-automata. Then, $\gen{\Ati} \leqslant \gen{\Ati'}$ if and only if there exists an $A$-homomorphism $\Ati \xto{} \Ati'$ and, in this case, the homomorphism is unique. 
\end{prop}

\begin{proof}
Uniqueness is an immediate consequence of the connectedness of $\Ati$ and the determinism of $\Ati'$: in fact, suppose that $\varphi_1,\varphi_2 \colon \Ati \xto{} \Ati'$ are $A$-homomorphisms of reduced pointed $A$-automata (in particular, $\bp\varphi_1=\bp\varphi_2=\bp'$); given a vertex $\verti \in \Verts \Ati$, choose a walk $\walki \equiv \bp \xwalk{\ } \verti$ in $\Ati$ and consider its images by $\varphi_1$ and $\varphi_2$: these will be two walks in $\Ati'$ with the same origin, namely $\bp'$, and the same label, namely $(\walki)\lab$, so, by the determinism of $\Ati'$, they must also have the same terminal vertices, $\verti \varphi_1 =\verti \varphi_2$ (see \Cref{rem: unique walk}); this shows that $\varphi_1=\varphi_2$.

The implication to the left has already been observed above. For the converse, suppose that $\gen{\Ati} \leqslant \gen{\Ati'}$, and let us construct an $A$-homomorphism $\varphi \colon \Ati \xto{} \Ati'$. Necessarily, we take $\bp \varphi =\bp'$. For every vertex $\verti \neq \bp$, consider a reduced $\bp$-walk in $\Gamma$ passing through $\verti$, say $\walki \colon \bp \xrwalk{u} \verti \xrwalk{v} \bp$  (such a walk does always exist because $\Ati$ is a core). Moreover, since $\Ati$ is deterministic, we have that $u$ and $v$ are both reduced words, and without cancellation in the product $uv$. In particular, $uv\in \gen{\Ati}$ and, by hypothesis, $uv\in \gen{\Ati'}$. This means that $uv$ is the reduced label of a $\bp'$-walk in $\Ati'$ and, since $\Ati'$ is deterministic, $uv$ is also the label of a (unique) reduced $\bp'$-walk $\walki'$ in $\Ati'$, which can then be decomposed as 
 \begin{equation*}
\walki' \colon \bp' \xrwalk{u} \verti' \xrwalk{v} \bp' \,.
 \end{equation*}
Define $\verti\varphi =\verti'$. To see this is well-defined, let $\bp \xrwalk{w} \verti \xrwalk{z} \bp$ be another reduced $\bp$-walk in $\Ati$ passing through $\verti$, consider $\bp' \xrwalk{w} \verti'' \xrwalk{z} \bp'$ the corresponding reduced $\bp'$-walk in $\Gamma'$, and let us see that $\verti'=\verti''$. Observe that $\bp \xrwalk{_{\scriptstyle{u}}} \verti \xrwalk{_{\scriptstyle{z}}} \bp$ is a \mbox{$\bp$-walk} in $\Gamma$ \emph{with possible backtracking at $\verti$}; denoting $\vertii$ the vertex where this backtracking ends, we have 
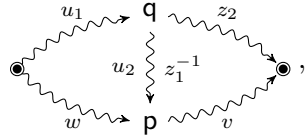
\begin{figure}[H]
\centering
\begin{tikzpicture}[shorten >=1pt, node distance=1cm and 2cm, on grid, decoration={snake, segment length=2mm, amplitude=0.5mm,post length=1.5mm},>=stealth']
  \node[state,accepting] (0) {};
  \node[] (q) [above right = 0.75 and 1.75 of 0] {$\vertii$};
  \node[] (p) [below right = 0.75 and 1.75 of 0] {$\verti$};
  \node[state,accepting] (0') [below right = 0.75 and 1.75 of q] {};
  \node[] (c) [right = 0.25 of 0'] {,};
   
    \path[->]
        (0) edge[snake it,bend left, out = 20,in= 170]
            node[pos=0.5,above = 0.1] {$u_1$}
            (q);
            
    \path[->]
        (q) edge[snake it,bend left, out = 10,in= 160]
            node[pos=0.5,above = 0.1] {$z_2$}
            (0');
            
    \path[->]
        (0) edge[snake it,bend left, out = -20,in= 190]
            node[pos=0.5,below = 0.1] {$w$}
            (p);
            
    \path[->]
        (p) edge[snake it,bend left, out = -10,in= 200]
            node[pos=0.5,below = 0.1] {$v$}
            (0');
            
    \path[->]
        (q) edge[snake it]
            node[pos=0.5,left = 0.1] {$u_2$}
            node[pos=0.5,right = 0.1] {$z_1^{-1}$}
            (p);
\end{tikzpicture}
\end{figure}
\noindent where $u=u_1 u_2$, $z=z_1 z_2$, $u_2=z_1^{-1}$, and the $\bp$-walk $\bp \xrwalk{u_1\,} \vertii \xrwalk{z_2\,} \bp$ is reduced. Hence, $uv=u_1 u_2 v$, $wz=w z_1 z_2$ and $u_1 z_2$ must be readable (without cancellation) as the labels of reduced $\bp'$-walks at $\Ati'$: 
 
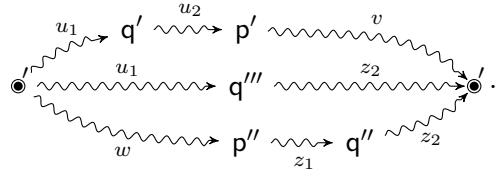
\vspace{3pt}
\begin{figure}[H]
\centering
  \begin{tikzpicture}[shorten >=1pt, node distance=1cm and 2cm, on grid, decoration={snake, segment length=2mm, amplitude=0.5mm,post length=1.5mm},>=stealth']

  \node[state,accepting,outer sep=5pt] (0) {};
  \node[] () [above right = 0.1 and 0.1 of 0] {$'$};
  \node[] (q') [above right = 0.75 and 1.5 of 0] {$\vertii'$};
  \node[] (p') [right = 1.5 of q'] {$\verti'$};
  \node[] (q''') [right = 3 of 0] {$\vertii'''$};
  \node[] (p'') [below right = 0.75 and 3 of 0] {$\verti''$};
  \node[] (q'') [right = 1.5 of p''] {$\vertii''$};
  \node[state,accepting] (0') [right =  6 of 0]{};
  \node[] () [above right = 0.1 and 0.1 of 0'] {$'$};
  \node[] (c) [right = 0.25 of 0'] {.};

  \path[->]
        (0) edge[snake it,bend left, out = 20,in= 170]
            node[pos=0.5,above=0.1] {$u_1$}
            (q');
                   
    \path[->]
        (0) edge[snake it, bend left, out = -20,in= 190]
            node[pos=0.5,below = 0.1] {$w$}
            (p'');
            
    \path[->]
        (p') edge[snake it, bend left, out = 10,in= 150]
            node[pos=0.5,above = 0.1] {$v$}
            (0');
            
    \path[->]
        (q') edge[snake it]
            node[pos=0.52,above = 0.1] {$u_2$}
            (p');
            
    \path[->]
        (p'') edge[snake it]
            node[pos=0.52,below = 0.1] {$z_1$}
            (q'');
            
    \path[->]
        (0) edge[snake it]
            node[pos=0.5,above = 0.05] {$u_1$}
            (q''');
            
    \path[->]
        (q''') edge[snake it]
            node[pos=0.5,above = 0.05] {$z_2$}
            (0');
            
    \path[->]
        (q'') edge[snake it,bend left, out = -10,in= 200]
            node[pos=0.5,below=0.1] {$z_2$}
            (0');
\end{tikzpicture}
\end{figure}
Finally, by the determinism of $\Ati'$, it is clear that $\vertii'=\vertii'''=\vertii''$ and so $\verti'= \verti''$ (using $u_2 =z_1 ^{-1}$), as we wanted to see. Thus, $\varphi\colon \Verts\Ati \xto{} \Verts\Ati'$ is well-defined. 

To finish the proof, let us see that $\varphi\colon \Ati \xto{} \Ati'$ is a morphism of $A$-automata. In fact, since $\Ati$ is core, given an arc $\edgi \equiv \verti \xarc{a\,} \vertii$ a $\Ati$, we know there exists a reduced $\bp$-walk in $\Ati$ using the arc $\edgi$:
 \begin{equation*}
\bp \xrwalk{u} \verti \xarc{\,a\ } \vertii \xrwalk{v} \bp \,.
 \end{equation*}
Hence, $uav \in \gen{\Ati} \leqslant \gen{\Ati'}$. Now, by definition of $\varphi$, we have in $\Ati'$ the reduced $\bp'$-walk 
 \begin{equation*}
\bp' \xrwalk{u} \verti \varphi \xarc{\,a\ } \vertii\varphi \xrwalk{v} \bp' \,.
 \end{equation*}
 
 \vspace{5pt}
 \noindent
In particular, the arc $\verti \varphi \xarc{a\, } \vertii \varphi$ exists in $\Ati'$ and so, $\varphi$ is an $A$-homomorphism, as we wanted to prove. 
\end{proof}

The following corollaries are immediate, and they express a property that will be crucial later. 

\begin{cor}
If  $\Ati$ is a reduced $A$-automaton, then the unique homo\-mor\-phism of $A$-automata $\Ati \xto{} \Ati$ is the identity. \qed
\end{cor}

\begin{cor} \label{cor: gen iff isom}
If two reduced $A$-automata have the same recognized subgroup then they are isomorphic. That is, for reduced $A$-automata $\Ati$ and $\Ati'$, we have
 \[
\pushQED{\qed} \gen{\Ati} = \gen{\Ati'} \,\Leftrightarrow\, \Ati \isom \Ati' \,. \qedhere \popQED
 \]
\end{cor}

\section{Stallings' automata}\label{sec: Stallings automata}

The notion of recognized subgroup
(see \Cref{def: recognized}.) provides a natural connection between geometric objects and subgroups of the free group. If we denote by $\Atts_{\bp}(A)$ the set of (isomorphism classes of) pointed involutive $A$-automata, and by $\Sgps(\Free[A])$ the family of subgroups of the free group $\Free[A]$, we have the following well defined map:

 \begin{equation} \label{eq: Ati ->> <Ati>}
\begin{array}{rcl}
\Atts_{\bp}(A)& \onto & \Sgps(\Free[A]). \\ \Ati & \mapsto & \gen{\Ati}
\end{array}
 \end{equation}
It is easy to see that this map is onto, that is, every $H\leqslant \Free[A]$ is the subgroup recognized by some pointed involutive $A$-automaton. In fact, let $S=\set{w_i}_{i\in I} \subseteq \Free[A]$ be a set of reduced words generating $H$ and, for every $w_i =a_{i,\scriptscriptstyle{1}}a_{i,\scriptscriptstyle{2}}\cdots a_{i,\scriptscriptstyle{l_i}}\in S$ ($a_{i,j} \in A^{\pm}$), consider the involutive closure of the directed cycle reading~$w_i$:
 \vspace{5pt}
\begin{figure}[H]
\centering
\begin{tikzpicture}[shorten >=1pt, node distance=0.2 and 1.5, on grid,auto,>=stealth']
    \node[state, accepting] (0) {};
    \node[state] (a) [above right =  of 0] {};
    \node[state] (1) [above right =  of a] {};
    \node[state] (b) [right =  of 1] {};
    \node[state] (d) [below right =  of 0] {};
    \node[state] (4) [below right =  of d] {};
    \node[state] (c) [right =  of 4] {};

    \path[->]
        (0) edge[]
            node[above] {$a_{i,\scriptscriptstyle{1}}$}
            (a);

     \path[->]
        (a) edge[]
            node[above left] {$a_{i,\scriptscriptstyle{2}}$}
            (1);

     \path[->]
        (1) edge[]
            node[above] {$a_{i,\scriptscriptstyle{3}}$}
            (b);

     \path[->,dashed]
        (b) edge[bend left,out=90,in=90,min distance=10mm]
            (c);

    \path[->]
        (c) edge[]
            node[midway,below] {$a_{i,\scriptscriptstyle{l_i}-2}$}
            (4);

    \path[->]
        (4) edge[]
            node[below] {$a_{i,\scriptscriptstyle{l_i}-1}$}
            (d);

    \path[->]
        (d) edge[]
            node[below] {$a_{i,\scriptscriptstyle{l_i}}$}
            (0);
\end{tikzpicture}
\vspace{-5pt}
\caption{The petal $\flower(w_i)$}
\label{fig: petal}
 \end{figure}
\noindent named the \defin[petal automaton]{petal} associated to $w_i$, and denoted $\flower(w_i)$; see~\Cref{fig: petal}. Then, define the so-called \defin{flower automaton} $\flower(S)$ as the (pointed involutive) $A$-automaton obtained after identifying the basepoints of the petals corresponding to the elements from $S$; see~\Cref{fig: flower}.
 \vspace{-10pt}
\begin{figure}[H]
\centering
\begin{tikzpicture}[shorten >=1pt, node distance=2cm and 2cm, on grid,auto,>=stealth',
decoration={snake, segment length=2mm, amplitude=0.5mm,post length=1.5mm}]
  \node[state,accepting] (1) {};
  \path[->,thick]
        (1) edge[loop,out=160,in=200,looseness=8,min distance=25mm,snake it]
            node[left=0.2] {$w_1$}
            (1);
            (1);
  \path[->,thick]
        (1) edge[loop,out=140,in=100,looseness=8,min distance=25mm,snake it]
            node[left=0.15] {$w_2$}
            (1);
  \path[->,thick]
        (1) edge[loop,out=20,in=-20,looseness=8,min distance=25mm,snake it]
            node[right=0.2] {$w_p$}
            (1);
\foreach \n [count=\count from 0] in {1,...,3}{
      \node[dot] (1\n) at ($(1)+(45+\count*15:0.75cm)$) {};}
\end{tikzpicture}
\vspace{-15pt}
\caption{The flower automaton $\flower(w_1,w_2,\ldots,w_p)$}
\label{fig: flower}
\end{figure}
It is clear that the subgroup recognized by $\flower(S)$ is, precisely, $\gen{\flower(S)}=\gen{S}=H\leqslant \Free[A]$; hence, the map~\eqref{eq: Ati ->> <Ati>} is onto. Observe, though, that this map is very far from being injective: for example, different families of generators for $H$ provide very different flower automata, all of them mapping to $H$ through~\eqref{eq: Ati ->> <Ati>}. 

Our goal is to restrict the map~\eqref{eq: Ati ->> <Ati>} to a bijection. It is easy to see that both non-determinism  and hanging trees not containing the basepoint are sources of redundancy to this respect. The main result of Stallings in~\cite{stallings_topology_1983} is that these are indeed the only sources of redundancy. \Cref{cor: gen iff isom} tells us that there is at most one (isomorphic class of) reduced automaton recognizing any given subgroup $H\leqslant \Free[A]$, \ie the restriction of the map~\eqref{eq: Ati ->> <Ati>} to the set of reduced $A$-automata is injective. In order to see that this restriction is also surjective, we will show a witness for every subgroup $H\leqslant \Free[A]$ (and explicitly construct it in the finitely generated case). 

Let $\Ati =(\Verts,\Edgs, \init,\term,\lab, \bp,^{-1})$ be a pointed involutive $A$-automaton; our goal is to construct a reduced automaton $\red{\Ati}$ such that $\gen{\Ati}=\gen{\red{\Ati}}$ (which we know it will be unique modulo isomorphism). Consider the equivalence relation $\sim$ in $\Verts$ given by 
 \begin{equation}
\verti \sim \vertii \ \Leftrightarrow \ \exists\ \text{$(\verti, \vertii)$-walk $\walki$ in $\Ati$ with $(\walki)\rlab =1$},
 \end{equation}
and write $\pi=\pi_{\Verts}\colon \Verts \xto{} \Verts /{\sim}$. Then, $\Ati_{\!1}= \Ati/{\sim}= (\Verts/{\sim}, \Edgs, \init \pi,\term \pi,\bp \pi,^{-1})$ is a pointed involutive automaton which is \defin[vertex-deterministic automaton]{vertex-deterministic} (meaning that if two arcs $\edgi, \edgii\in \Edgs$ satisfy $\edgi\init\pi =\edgii\init\pi$ and $(\edgi)\lab =(\edgii)\lab$, then $\edgi\term\pi =\edgii\term\pi$): indeed, if $\edgi\init \pi =\edgii\init \pi$ there is a $(\edgi\init, \edgii\init)$-walk $\walki$ in $\Ati$ with $(\walki)\rlab =1$, and then $\edgi^{-1}\walki \edgii$ is a $(\edgi\term, \edgii\term)$-walk also satisfying $(\edgi^{-1}\walki \edgii)\rlab=1$; hence, $\edgi\term\pi=\edgii\term\pi$. Therefore, the pointed involutive $A$-automaton $\Ati_{\!2}$ obtained from $\Ati_{\!1}$ by identifying each set of parallel arcs with a common label into a single arc (keeping the label) is deterministic. It is clear that $\red{\Ati} =\core(\Ati_{\!2})$ is a (the) reduced $A$-automaton such that $\gen{\overline{\Ati}} =\gen{\Ati_2}=\gen{\Ati_1} =\gen{\Ati}$. This argument completes the proof that the desired restriction of \eqref{eq: Ati ->> <Ati>} is a bijection.

\begin{defn}
For any $\Ati \in \Atts_{\bp}(A)$, its \defin[reduction of an automaton]{reduction}, denoted by $\red{\Ati}$, is the (unique modulo isomorphism) reduced automaton recognizing $\gen{\Ati}$. Similarly, for any ${H\leqslant \Free[A]}$, the \defin[Stallings' automaton]{Stallings' automaton of $H$ \wrt $A$}, denoted by $\stallings(H,A)$, is the unique (modulo isomorphism) reduced $A$-automaton whose recognized subgroup is $H$. Note that $\rk(\red{\Ati}) =\rank(\gen{\Ati})=\rk(\stallings(\gen{\Ati},A))$.  
\end{defn}


\begin{thm}[\citenr{stallings_topology_1983}]\label{thm: Stallings bijection}
The map
  \begin{equation} \label{eq: Stallings bijection1}
\begin{array}{rcl}
\red{\Atts}_{\hspace{1pt}\bp}(A) & \xto{} & \Sgps(\Free[A]) \\ \Ati & \mapsto & \gen{\Ati}
\end{array}
 \end{equation}
is a bijection whose inverse is $\stallings(H,A) \mapsfrom H$. \qed
\end{thm}

As an immediate consequence, the fundamental result below follows now in a very transparent way. With different techniques, it was first proved by \citeauthor{nielsen_om_1921} for the finitely generated case, and a few years later by \citeauthor{schreier_untergruppen_1927} with full generality. 

\begin{thm}[Nielsen \cite{nielsen_om_1921}, Schreier\cite{schreier_untergruppen_1927}] \label{thm: Nielsen-Schreier}
Every subgroup $H$ of a free group is free of rank $\rk(\stallings(H,A))$.  
\end{thm}

\begin{proof}
By \Cref{thm: Stallings bijection}, every subgroup $H\leqslant \Free[A]$ is of the form $H=\gen{\stallings(H,A)}$; this is isomorphic to $\pi(\stallings(H,A))$ by \Cref{lem: deterministic gamma}; and this last group is free of rank $\rk(\stallings(H,A))$, by \Cref{thm: B_T}.  
\end{proof}

In fact, bijection \eqref{eq: Stallings bijection1} allows not only to guarantee the freeness of subgroups of $\Free[A]$, but also to exactly describe the (isomorphic classes of the) subgroups of a given free group of countable rank.

\begin{cor} \label{cor: ranks of subgroups}
Let $\Free[\kappa]$ be a free group of rank $\kappa \in [2,\aleph_0]$. Then, for every cardinal $\mu \in [0,\aleph_0]$ there exists a subgroup $H\leqslant \Free[\kappa]$ such that $H\isom \Free[\mu]$. \qed
\end{cor}

Below, we present two classical examples of infinite-rank subgroups of $\Free[2] = \Free_{\set{a,b}}$ through their corresponding Stallings' automata.

\begin{exm}
The normal closure of $b$, denoted by $\ncl{b} \leqslant \Free[2]$, is a subgroup of infinite rank, as it follows from its (infinite rank) Stallings' automaton $\stallings(\ncl{b})$ depicted in \Cref{fig: Stallings <<b>>}. The only spanning tree available (in blue) provides the basis 
 $\set{a^kba^{-k}\st k\in \ZZ}$ for $\ncl{b}$.
\begin{figure}[H] 
\centering
  \begin{tikzpicture}[shorten >=1pt, node distance=1.2 and 1.2, on grid,auto,>=stealth']
   \node[state,accepting] (0) {};
   \node[state] (1) [right = of 0]{};
   \node[state] (2) [right = of 1]{};
   \node[state] (3) [right = of 2]{};
   \node[] (4) [right = 1.5 of 3]{$\cdots$};
   \node[] (c) [right = .5 of 4]{};

   \node[state] (-1) [left = of 0]{};
   \node[state] (-2) [left = of -1]{};
   \node[state] (-3) [left = of -2]{};
   \node[] (-4) [left = 1.5 of -3]{$\cdots$};

   \path[->]
        (0) edge[loop above,red,min distance=10mm,in=55,out=125]
            node[] {\scriptsize{$b$}}
            (0)
            edge[blue,thick]
            node[below] {\scriptsize{$a$}}
            (1);

    \path[->]
        (1) edge[loop above,red,min distance=10mm,in=55,out=125]
            (1)
            edge[blue,thick]
            (2);

    \path[->]
        (2) edge[loop above,red,min distance=10mm,in=55,out=125]
            (2)
            edge[blue,thick]
            (3);

    \path[->]
        (3) edge[loop above,red,min distance=10mm,in=55,out=125]
            (3)
            edge[blue,thick]
            (4);

    \path[->]
        (-1) edge[loop above,red,min distance=10mm,in=55,out=125]
            (-1)
            edge[blue,thick]
            (0);

    \path[->]
        (-2) edge[loop above,red,min distance=10mm,in=55,out=125]
            (-2)
            edge[blue,thick]
            (-1);

    \path[->]
        (-3) edge[loop above,red,min distance=10mm,in=55,out=125]
            (-3)
            edge[blue,thick]
            (-2);

    \path[->]
        (-4) edge[blue,thick]
            (-3);
\end{tikzpicture}
\caption{The (infinite) Stallings' automaton of $\ncl{b} \leqslant \Free_{\set{a,b}}$}
\label{fig: Stallings <<b>>}
\end{figure}
\end{exm}

\begin{exm}
The commutator subgroup $\Comm{\Free[2]}\normaleq  \Free[2]$, is also of infinite rank, as its Stallings' automaton is nothing else than the Cayley graph $\cayley(\ZZ^2,\set{a,b})$, depicted in \Cref{fig: commutator}. We propose the reader to describe a basis for the commutator $\Comm{\Free[2]}$ using his favorite spanning tree.
\end{exm}

\vspace{-5pt}
\begin{figure}[H]
\centering
\begin{tikzpicture}[shorten >=1pt, node distance=1.2cm and 2cm, on grid,auto,auto,>=stealth']

\newcommand{\dx}{0.75}
\newcommand{\dy}{0.7}
\node[] (0)  {};

\foreach \x in {-2,...,2}
\foreach \y in {-2,...,2} 
{
\node[state] (\x!\y) [above right = \y*\dy and \x*\dx of 0] {};
}
\node[state,accepting] (0!0){};

 \foreach \x in {-3,...,3} 
{
\node (\x!3) [above right = 3*\dy and \x*\dx of 0] {};
\node (\x!-3) [above right = -3*\dy and \x*\dx of 0] {};
}

\foreach \y in {-2,...,2} 
{
\node (3!\y) [above right = \y*\dy and 3*\dx of 0] {};
\node (-3!\y) [above right = \y* \dy and -3*\dx of 0] {};
}

\draw[draw=none,red] (0!0) edge node[left,pos=0.45] {$b$} (0!1);
\draw[draw=none,blue] (0!0) edge node[above,pos=0.45] {$a$} (1!0);

\foreach \x [evaluate = \x as \xx using int(\x+1)] in {-3,...,2}
  \foreach \y in {-2,...,2}
{
\ifthenelse{\x=-3 \OR \x=2 \OR \y=-3 \OR \y=3}
{\draw[->,blue,dotted] (\x!\y) edge (\xx!\y);}
{\draw[->,blue] (\x!\y) edge (\xx!\y);}
}

\foreach \y [evaluate = \y as \yy using int(\y+1)] in {-3,...,2}
  \foreach \x in {-2,...,2}
{
\ifthenelse{\x=-3 \OR \x=3 \OR \y=-3 \OR \y=2}
{\draw[->,red,dotted] (\x!\y) edge (\x!\yy);}
{\draw[->,red] (\x!\y) edge (\x!\yy);}
}

\end{tikzpicture}
\vspace{-10pt}
\caption{The Stallings' automaton of the commutator subgroup $\Comm{\Free_{\set{a,b}}}$}
\label{fig: commutator}
\end{figure}

\begin{cor}
 For every $\kappa \in [2,\aleph_0]$, the free group $\Free[\kappa]$ has uncountably many different subgroups. \qed  
\end{cor}

Let us now focus on the computational nature of the bijection~\eqref{eq: Stallings bijection1}, which will be crucial to our purposes. First, observe that, in \Cref{thm: Stallings bijection}, (isomorphism classes of) \emph{finite} reduced $A$-automata correspond, precisely, to \emph{finitely generated} subgroups. Moreover, under finiteness assumptions, the bijection \eqref{eq: Stallings bijection1} is computable.





\begin{prop} \label{prop: B_T comp}
Let $\Ati$ be a reduced $A$-automaton and let $T$ be a spanning tree of~$\Ati$. Then, 
 \begin{equation}
B_T \,=\, (\mathcal{B}_T)\red{\lab} \,=\, \set{(\omega_{\edgi})\red{\lab} \st \edgi\in \Edgs^+\Ati \setmin \Edgs T}
 \end{equation}
is a basis for $\gen{\Ati}$. In particular, $\gen{\Ati}$ is finitely generated if and only if $\Ati$ is finite, and, in this case, a basis for $\gen{\Ati}$, and hence the rank $\rk(\gen{\Ati})=1- \card\Verts\Ati+\card\Edgs^{+}\Ati$, are finite and computable.
\end{prop}

\begin{proof}
The first claim follows immediately from \Cref{lem: deterministic gamma} and \Cref{thm: B_T}. Moreover, if~$\Ati$ is finite it is clear that all the steps are computable: first compute a spanning tree $T$ for $\Ati$ and, for each arc $\edgi\in \Edgs^+\Ati \setmin \Edgs T$ (there are finitely many such arcs), compute the reduced word $w_{\edgi}= (\walki_{\edgi}) \red{\lab}= (\bp \xwalk{\scriptscriptstyle{T}} \! \bullet \! \xarc{\,\edgi\ } \!\bullet\! \xwalk{\scriptscriptstyle{T}} \bp)\rlab \in \Free[A]$ reading the corresponding $\edgi$-petal. Since $\pi(\Ati)\isom \gen{\Ati}$ via $\rlab$, this collection of words $B_T =(\mathcal{B}_T)\red{\lab}=\{ w_{\edgi} \mid \edgi\in \Edgs^+\Ati \setmin \Edgs T\}$ forms a free basis for $\gen{\Ati}\leqslant \Free[A]$.
\end{proof}

\begin{exm}
Consider the following reduced $\{a,b\}$-automaton $\Ati$, with the spanning tree~$T$ given by the horitzontal arcs (represented by boldfaced arrows).
 
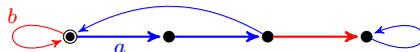
\begin{figure}[H]
\centering
\begin{tikzpicture}[shorten >=1pt, node distance=1.2cm and 2cm, on grid,auto,auto,>=stealth']
        
\newcommand{\dx}{1.3}
\newcommand{\dy}{1.2}

\node[state,accepting] (a1) [right = \dy-1/3 of 0] {};
\node[state] (a2) [right = \dx of a1] {};
\node[state] (a3) [right = \dx of a2] {};
\node[state] (a4) [right = \dx of a3] {};

\path[->]
(a1) edge[red,loop above,min distance=10mm,in=205,out=155]
node[above = 0.1] {$b$}
(a1)
(a1) edge[blue,thick] node[below] {$a$} (a2)
(a2) edge[blue,thick] (a3)
(a3) edge[blue,bend right=30] (a1)
(a3) edge[red,thick] (a4)
(a4) edge[blue,loop above,min distance=10mm,in=25,out=-25] (a4);
\end{tikzpicture}
 \end{figure}

\vspace{-5pt}
\noindent According to the discussion above, the subgroup $\gen{\Ati}\leqslant \Free[\set{a,b}]$ is free with basis $B_T =\set{\,b, a^3,a^2 b a b^{-1} a^{-2}\,}$ and so, of rank 3.
\end{exm}

Conversely, given a finite pointed $A$-automaton $\Ati$, we will show how to construct its reduction $\red{\Ati}$. In order to do so, we need to introduce an important automata transformation, called elementary Stallings' folding.

\begin{defn}\label{def: foldings}
Let $\Ati$ be an involutive $A$-automaton. An \defin[Stallings' folding]{elementary Stallings' folding} on $\Ati$ is a transformation consisting in identifying two arcs $\edgi_1,\edgi_2 \in \Edgs \Ati$ violating determinism (\ie such that $\edgi_1 \init = \edgi_2 \init$, $(\edgi_1 )\lab=(\edgi_2 )\lab$, but~$\edgi_1\neq \edgi_2$) and their corresponding inverses. If $\edgi_1$ and $\edgi_2$ are not parallel (\ie $\edgi_1 \term \neq \edgi_2 \term$), we say that the folding is \defin[open folding]{open}; otherwise, we say that the folding is \defin[closed folding]{closed}; see~\Cref{fig: foldings}).
\end{defn}
\begin{figure}[H] 
\centering
\begin{tikzpicture}[shorten >=1pt, node distance=1cm and 1.5cm, on grid,>=stealth']
\begin{scope}
   \node[state] (1) {};
   \node[state] (2) [above right = 0.5 and 1 of 1] {};
   \node[] (21) [above right = 0.3 and 0.7 of 2] {};
   \node[] (22) [below right = 0.3 and 0.7 of 2] {};
   \node[state] (3) [below right = 0.5 and 1 of 1] {};
   \node[] (31) [above right = 0.3 and 0.7 of 3] {};
   \node[] (32) [below right = 0.3 and 0.7 of 3] {};
   \node[] (neq) [right = 1 of 1]{\rotatebox[origin=c]{90}{$\neq$}};

   \path[->]
        (1) edge[]
            node[pos=0.52,above left] {$a$}
            (2)
            edge[]
            node[pos=0.5,below left] {$a$}
            (3)
        (2) edge[gray]
            (21)
            edge[gray]
            (22)
        (3) edge[gray!80]
            (31)
            edge[gray!80]
            (32);

    \node[] (i) [right = 1.8 of 1]{};
    \node[] (f) [right = 0.8 of i] {};
    \path[->]
        (i) edge[bend left] (f);

   \node[state] (1') [right= 0.3 of f] {};
   \node[state] (N) [right =  1.2 of 1'] {};
   \node[] (21') [above right = 0.7 and 0.3 of N] {};
   \node[] (22') [above right = 0.3 and 0.7 of N] {};
   \node[] (32') [below right = 0.3 and 0.7 of N] {};
   \node[] (33') [below right = 0.7 and 0.3 of N] {};
   \path[->]
        (1') edge
            node[pos=0.48,above] {$a$}
            (N)
        (N) edge[gray]
            (21')
            edge[gray]
            (22')
            edge[gray!80]
            (32')
            edge[gray!80]
            (33');

\foreach \n [count=\count from 0] in {1,...,3}{
       \node[dot,gray] (2d\n) at ($(2)+(-10+\count*10:0.4cm)$) {};}

\foreach \n [count=\count from 0] in {1,...,3}{
       \node[dot,gray!80] (2d\n) at ($(3)+(-10+\count*10:0.4cm)$) {};}

\foreach \n [count=\count from 0] in {1,...,3}{
       \node[dot,gray] (2d\n) at ($(N)+(34+\count*10:0.4cm)$) {};}

\foreach \n [count=\count from 0] in {1,...,3}{
       \node[dot,gray!80] (2d\n) at ($(N)+(-53+\count*10:0.4cm)$) {};}
  \end{scope}
  
 \begin{scope}[xshift=6cm]
  \node[state] (1)  {};
   \node[state] (2) [right = 1.2 of 1] {};
  \node[] (21) [above right = 0.7 and 0.3 of 2] {};
  \node[] (22) [above right = 0.3 and 0.7 of 2] {};
  \node[] (32) [below right = 0.3 and 0.7 of 2] {};
  \node[] (33) [below right = 0.7 and 0.3 of 2] {};
  \path[->]
        (2) edge[gray]
            (22)
            edge[gray]
            (32);

  \path[->]
        (1) edge[bend left]
            node[pos=0.5,above] {$a$}
            (2);
  \path[->]
        (1) edge[bend right]
            node[pos=0.5,below] {$a$}
            (2);

    \node[] (i) [right = 2 of 1] {};
    \node[] (f) [right = 0.8 of i] {};
    \path[->]
        (i) edge[bend left] (f);

  \node[state] (1') [right= 0.3 of f] {};
  \node[state] (N) [right =  1.2 of 1'] {};
  \node[] (21') [above right = 0.7 and 0.3 of N] {};
  \node[] (22') [above right = 0.3 and 0.7 of N] {};
  \node[] (32') [below right = 0.3 and 0.7 of N] {};
  \node[] (33') [below right = 0.7 and 0.3 of N] {};
  \path[->]
        (1') edge
            node[pos=0.48,above] {$a$}
            (N)
        (N) edge[gray]
            (22')
            edge[gray]
            (32');

\foreach \n [count=\count from 0] in {1,...,3}{
      \node[dot,gray] (2d\n) at ($(2)+(-12+\count*10:0.4cm)$) {};}

\foreach \n [count=\count from 0] in {1,...,3}{
      \node[dot,gray] (2d\n) at ($(N)+(-12+\count*10:0.4cm)$) {};} 
 \end{scope}
\end{tikzpicture}
 \vspace{-5pt}
\caption{An open elementary folding (left) and a closed elementary folding (right)}
\label{fig: foldings}
\end{figure}
Note that, formally, an elementary Stallings' folding $\phi$ is a special type of surjective homomorphism of $A$-automata $\phi \colon \Ati \onto \Ati'$, which we usually denote by \smash{$\Ati \! \xtr{\!\phi\!} \Ati'$}. We write $\Ati \! \xtr{\!} \Ati'$ to express that $\Ati'$ is the automaton obtained after performing some elementary Stallings' folding on $\Ati$. An easy vertex and arc count yields the remark below, which will be important later.

\begin{rem} \label{rem: rank folding}
If $\Ati$ is finite and $\Ati \xtr{} \Ati'$ is an elementary Stallings' folding, then 
 \begin{equation}
\rk(\Ati') \,=\, \bigg\{ \begin{array}{lll} \rk(\Ati) & & \text{si $\Ati \xtr{} \Ati'$ is open,} \\ \rk(\Ati)-1 & & \text{si $\Ati \xtr{} \Ati'$ is closed.}
\end{array}     
 \end{equation}
\end{rem}

It is also clear how elementary Stallings' foldings affect the corresponding fundamental and recognized subgroups.

\begin{lem}\label{lem: grrec}
Let $\Ati \!\xtr{\phi}\! \Ati'$ be an elementary Stallings' folding. Then,
\begin{enumerate*}[ind]
\item $\pi(\Ati)\isom \pi(\Ati')$ if and only if $\phi$ is open; and
\item $\gen{\Ati'}=\gen{\Ati}$.\qed
\end{enumerate*}
\end{lem}




\begin{defn}
    Let $\Ati$ be an involutive $A$-automaton. A \defin{folding sequence} on~$\Ati$ is a (finite or infinite) sequence of elementary foldings $\Phi =(\phi_i)_{i\geq 0}$ such that $\Ati =\Ati_{\!0}$ and \smash{$\Ati_{\!i-1} \xtr{\! \phi_i \!} \Ati_{\!i}$}, for every $i \neq 0 $. If $\Phi$ is finite, $\Phi =(\phi_i)_{i=1}^{k}$, then we say that $\Ati_{\!k}$ is the \defin[result of a folding sequence]{result} of a \defin[multiple folding]{(multiple) folding} $\Phi$ of length $k$, where we abuse language and let $\Phi = \phi_1 \phi_2 \cdots \phi_k$. A finite folding sequence is \defin[maximal folding sequence]{maximal} if its result $\Ati_{\!k}$ is deterministic (and hence no more elementary foldings are available).

A \defin{reduction process} for $\Ati$ is a maximal finite folding sequence on $\Ati$ followed by a $\core$ operation,
 \begin{figure}[H]
\vspace{-10pt}
 \begin{equation} 
\Ati = \Ati_{\hspace{-1pt}0} \xtr{\!\!\! \phi_1 \!\!}\Ati_{\hspace{-1pt}1} \xtr{\!\!\! \phi_2 \!\!} \cdots \xtr{\!\!\! \phi_m \!\!} \Ati_{\hspace{-1pt}m} \xto{\!\! \core} \!\core(\Ati_{\hspace{-1pt}m}) \isom \red{\Ati}
 \end{equation}
 \vspace{-12pt}
\caption{A reduction process: a maximal folding sequence followed by a $\core$ operation} \label{fig: eq: folding sequence}
\end{figure}
\end{defn}

\begin{prop}\label{prop: reduction is computable}
If $\Ati$ is finite, then the result of any reduction process on~$\Ati$ is~$\red{\Ati}$. Hence, the reduction $\red{\Ati}$ is computable.
\end{prop}

\begin{proof}
By \Cref{lem: grrec}, the result of any reduction process applied to $\Ati$ is a reduced $A$-automaton recognizing $\gen{\Ati}$ hence, by \Cref{thm: Stallings bijection}, it must coincide with $\red{\Ati}$, independently from the specific reduction process followed. We can keep detecting available foldings (open or closed, in no particular order) and we carry them out until no more elementary foldings are available (\ie until we reach a deterministic automaton). Observe that, a particular elementary folding may create new available ones (and even momentarily increase the number of them available); however, each elementary folding operation decreases the total number of edges in the automaton by exactly one and so, starting with a finite $\Ati$, it is guaranteed that any such sequence of foldings will eventually stop, reaching a deterministic $A$-automaton.
\end{proof}



Observe that, by \Cref{lem: grrec} and \Cref{cor: <core> = <ati>}, all the automata appearing in a reduction process for $\Ati$ recognize the same subgroup $\gen{\Ati}$. Note also that, since~$\red{\Ati}$ is the only reduced automaton recognizing $\gen{\Ati}$, the result of a folding process is independent of the folding sequence \emph{and even of the starting finite automaton}, as long as it recognizes the subgroup $\gen{\Ati}$.

The computability of the inverse map in \eqref{eq: Stallings bijection1} for finitely generated subgroups follows immediately from \Cref{prop: reduction is computable}. Indeed, given a finite set of generators $S$ for a finitely generated subgroup $H\leqslant \Free[A]$, it is enough to construct the flower automaton $\flower(S)$ and apply a folding process to it until reaching $\red{\flower(S)} =\stallings(H,A)$. As noted above, the resulting automaton $\stallings(H,A)$ depends canonically on the subgroup $H$  (it is even independent from the finite set of generators $S$ for $H$ used to start the process). 

To conclude the considerations about the compatibility of the inverse map $H \mapsto \stallings(H,A)$ let us mention that if the words in $S$ are reduced then it is not difficult to see that every automaton in any folding sequence for $\flower(S)$ is already core; hence the final core operation in the computation of $\stallings(H,A)$ can be omitted,

\begin{figure}[H]
\vspace{-5pt}
 \begin{equation} 
\flower(S) =\Ati_{\hspace{-1pt}0} \xtr{\!\!\! \phi_1 \!\!}\Ati_{\hspace{-1pt}1} \xtr{\!\!\! \phi_2 \!\!} \cdots \xtr{\!\!\! \phi_m \!\!} \Ati_{\hspace{-1pt}m} = \stallings(H,A)
 \end{equation}
\vspace{-10pt}
\caption{Computation of $\stallings(H,A)$ from a finite set $S$ of reduced generators for $H$}\label{fig: eq: folding sequence}
\end{figure}
\vspace{-15pt}

\begin{cor}\label{cor: basis are computable}
Given a finite set $S \subseteq \Free[A]$, a basis for (and hence the rank of) the subgroup $\gen{S}$ is computable. \qed
\end{cor}

\begin{exm}\label{ex: Stallings}
Let $\Free[2]$ be the free group on $A=\{a,b\}$ and let ${H \leqslant \Free[2]}$ be the subgroup generated by $S=\set{w_1,w_2,w_3}$, where $w_1=a^3$, $w_2=abab^{-1}$, and $w_3=a^{-1}bab^{-1}$. In order to compute a basis (and so, the rank) of~$H$, let us compute first $\stallings(H,A)$. We start building the flower automaton $\Ati_0=\flower(S)$, and then keep applying elementary foldings (in any order) until reaching $\stallings(H,A)$: in \Cref{fig: Stallings sequence} we can see a possible folding sequence 
to compute $\stallings(H,A)$, where, at each step, we have depicted with boldfaced arrows the two arcs getting folded in the following step (observe that the first step correspond, in fact, to three elementary foldings done simultaneously). 
\vspace{7pt}
\begin{figure}[H]
\centering
\begin{tikzpicture}[shorten >=1pt, node distance=1cm and 1cm, on grid,auto,auto,>=stealth']
\begin{scope}[rotate=90]
\newcommand{\rad}{0.85}
\newcommand{\Rad}{1.45}
\node[state,accepting] (bp)  {};

\foreach \x in {0,...,5}
{
\node[state] (\x) at ($(90+\x*360/6:\rad cm)$) {};
}

\node[state] (34) at ($(180+120:\Rad cm)$) {};
\node[state] (50) at ($(180+2400:\Rad cm)$) {};
\node[] (l0) [above = 1.2 of bp] {$\Ati_{\hspace{-1pt}0} \!=\! \flower(S)$};

\draw[->,blue] (bp) edge (1);
\draw[->,blue] (1) edge (2);
\draw[->,blue] (2) edge (bp);

\draw[->,blue] (bp) edge node[above,pos=0.45] {\scriptsize{$a$}}  (0);
\draw[->,red] (0) edge[thick] (50);
\draw[->,blue] (50) edge[thick] (5);
\draw[->,red] (bp) edge[thick] node[right,pos=0.5] {\scriptsize{$b$}} (5);

\draw[->,red] (bp) edge[thick] (4);
\draw[->,blue] (34) edge[thick] (4);
\draw[->,red] (3) edge[thick] (34);
\draw[->,blue] (3) edge (bp);

\node[] (i) [right = 1.25 of bp]{};
    \node[] (f) [right = 0.9 of i] {};
    \path[->]
        (i) edge[bend left] (f);
\end{scope}

\begin{scope}[xshift=3.25cm]         

\newcommand{\rad}{0.85}
\newcommand{\Rad}{1.45}
\node[state,accepting] (0)  {};
\node[state] (1) [right = \rad of 0] {};
\node[state] (2) [right = \rad of 1] {};
\node[state] (3) [right = \rad of 2] {};
\node[state] (5) at ($(180+30:\rad cm)$) {};
\node[state] (6) at ($(180-30:\rad cm)$) {};
\node[] (l2) [above right = 0.85 and \rad/2 of 1] {$\Ati_{\hspace{-1pt}1}$};

\draw[->,blue] (0) edge[thick] (6);
\draw[->,blue] (6) edge (5);
\draw[->,blue] (5) edge (0);
\draw[->,red] (0) edge (1);
\draw[->,blue] (2) edge (1);
\draw[->,red] (3) edge (2);
\draw[->,blue,bend left] (0) edge[thick] (3);
\draw[->,blue,bend left] (3) edge (0);

\node[] (i) [right = 2.9 of 0]{};
    \node[] (f) [right = 0.9 of i] {};
    \path[->]
        (i) edge[bend left] (f);
\end{scope}

\begin{scope}[xshift=7.3cm]         

\newcommand{\rad}{0.85}
\newcommand{\Rad}{1.45}
\node[state,accepting] (0)  {};
\node[state] (1) [right = \rad of 0] {};
\node[state] (2) [right = \rad of 1] {};
\node[state] (3) [right = \rad of 2] {};
\node[state] (4) [below right = \rad and \rad/2 of 1] {};
\node[] (l2) [above right = 0.85 and \rad/2 of 1] {$\Ati_{\hspace{-1pt}2}$};

\draw[->,red] (0) edge (1);
\draw[->,blue] (2) edge (1);
\draw[->,red] (3) edge (2);
\draw[->,blue,bend left] (0) edge (3);
\draw[->,blue,bend left] (3) edge[thick] (0);
\draw[->,blue,bend left] (3) edge (4);
\draw[->,blue,bend left] (4) edge[thick] (0);

\node[] (i) [below right = \rad and \rad/2 of 2]{};
    \node[] (f) [below = 0.9 of i] {};
    \path[->]
        (i) edge[bend left] (f);
\end{scope}

\begin{scope}[xshift=7cm,yshift=-2.3cm]         

\newcommand{\rad}{0.85}
\newcommand{\Rad}{1.45}
\node[state,accepting] (0)  {};
\node[state] (1) [right = \rad of 0] {};
\node[state] (2) [right = \rad of 1] {};
\node[state] (3) [right = \rad of 2] {};
\node[] (l2) [above right = 0.75 and \rad/2 of 1] {$\Ati_{\hspace{-1pt}3}$};

\draw[->,red] (0) edge (1);
\draw[->,blue] (2) edge (1);
\draw[->,red] (3) edge (2);
\draw[->,blue,bend left] (0) edge (3);
\draw[->,blue,bend left] (3) edge[thick] (0);
\draw[->,blue,loop right,min distance=8 * \rad mm,in=30,out=-30] (3) edge[thick] (3);

\node[] (i) [left = 0.5 of 0]{};
    \node[] (f) [left = 0.9 of i] {};
    \path[->]
        (i) edge[bend right] (f);
\end{scope}

\begin{scope}[xshift=3.5cm,yshift=-2.3cm]         

\newcommand{\rad}{0.85}
\newcommand{\Rad}{1.45}
\node[state,accepting] (0)  {};
\node[state] (1) [right = \rad of 0] {};
\node[state] (2) [right = \rad of 1] {};
\node[] (l2) [above = 0.55 of 1] {$\Ati_{\hspace{-1pt}4}$};
\draw[->,blue,loop right,min distance=8 * \rad mm,in=160,out=100] (0) edge[thick] (0);
\draw[->,blue,loop right,min distance=8 * \rad mm,in=-160,out=-100] (0) edge[thick] (0);

\draw[->,red] (0) edge (1);
\draw[->,blue] (2) edge (1);
\draw[->,red,bend left] (0) edge (2);

\node[] (i) [left = 0.8 of 0]{};
    \node[] (f) [left = 0.9 of i] {};
    \path[->]
        (i) edge[bend right] (f);
\end{scope}

\begin{scope}[xshift=-0.3cm,yshift=-2.3cm]         

\newcommand{\rad}{0.85}
\newcommand{\Rad}{1.45}
\node[state,accepting] (0)  {};
\node[state] (1) [right = \rad of 0] {};
\node[state] (2) [right = \rad of 1] {};
\node[] (l2) [above = 0.55 of 1] {$\Ati_{\hspace{-1pt}5}$};
\draw[->,blue,loop right,min distance=8 * \rad mm,in=150,out=210] (0) edge (0);

\draw[->,red] (0) edge[thick] (1);
\draw[->,blue] (2) edge (1);
\draw[->,red,bend left] (0) edge[thick] (2);

\node[] (i) [below = 0.4 of 0]{};
    \node[] (f) [below = 0.9 of i] {};
    \path[->]
        (i) edge[bend right] (f);
\end{scope}

\begin{scope}[xshift=0.3cm,yshift=-4cm]         

\newcommand{\rad}{0.85}
\newcommand{\Rad}{1.45}
\node[state,accepting] (0)  {};
\node[state] (1) [right = \rad of 0] {};
\node[] (l6) [above left= 0.5 and 0.3 of 1] {$\Ati_{\hspace{-1pt}6} \!=\! \stallings(H)$};

\draw[->,blue,loop right,min distance=8 * \rad mm,in=150,out=210] (0) edge node[left,pos=0.5] {\scriptsize{$a$}} (0);
\draw[->,blue,loop right,min distance=8 * \rad mm,in=-30,out=30] (1) edge (1);

\draw[->,red] (0) edge node[above,pos=0.45] {\scriptsize{$b$}} (1);
\end{scope}
\end{tikzpicture}
\vspace{-5pt}
\caption{A folding sequence for the subgroup $H=\gen{a^3,\, abab^{-1},\, a^{-1}bab^{-1}}\leqslant \Free_{\set{a,b}}$} 
\label{fig: Stallings sequence}
\end{figure} 
Finally, taking ${\bp \R{\xarc{ \R{b}\, }} \bullet}$ as a maximal tree for $\stallings(H,A)$, from \Cref{thm: B_T} we have that $\set{a, bab^{-1}}$ is a basis of $H$; in particular, $\rk (H)=2$.
\end{exm}

Including the computational results just obtained in the statement of \Cref{thm: Stallings bijection}, we reach the main result in the present section.

\begin{thm}[\citenr{stallings_topology_1983}]\label{thm: Stallings bijection2}
The map
 \begin{equation}\label{eq: Stallings bijection2}
\begin{array}{rcl} \mathsf{St}\colon \Sgps(\Free[A]) & \xto{} & \red{\Atts}_{\hspace{1pt}\bp}(A) \\ H & \mapsto & \stallings(H,A)
\end{array}
 \end{equation}
is a bijection whose inverse is $\gen{\Ati} \mapsfrom \Ati$. Moreover, a subgroup $H\leqslant \Free[A]$ is finitely generated if and only if $\stallings(H,A)$ is finite; in this case, the above bijection is computable in both directions. \qed
\end{thm}

We emphasize the naturalness and efficiency of the computation in both directions. Quite remarkably, it was proven in \cite{touikan_fast_2006} that the computation of the Stallings' automaton can be performed in `almost linear' time on the sum of lengths of the input generators.

As can be guessed by the experienced readers, the $A$-automaton $\stallings(H,A)$ is a concise representation of a very classical construction: the Schreier digraph associated to a subgroup $H$ \wrt the ambient set of generators $A$.

\begin{defn}\label{def: Sch}
Let $G$ be a group, $H$ be a subgroup of $G$, and $S\subseteq G$ be a set of generators for $G$. Then, the \defin[Schreier automaton]{(right) Schreier automaton of $H$ with respect to~$S$}, denoted  by $\schreier(H,S)$, is the involutive $S$-automaton with vertex set $H\backslash G$ (the set of right cosets of $G$ modulo $H$), an arc $Hg \xarc{s\,} Hgs$ for every coset $Hg\in H\backslash G$ and every element $s\in S^{\pm}$, and the coset $H$ as a basepoint. In the particular case $H=\Trivial$, we recover the classical notion of Cayley automaton (or graph), $\cayley(G,S)=\schreier_G(\Trivial ,S)$.
\end{defn}

\begin{rem}
Note that, if $H\normaleq G$, then $\schreier_G(H,S)=\cayley(G/H,S)$. 
\end{rem}

\begin{exm}
The Cayley automaton of the free group $\Free[2] =\Free[\{a,b\}]$ \wrt the basis $\{a,b\}$ is depicted in \Cref{fig: Sch(F_2)}. It is clear that $\cayley(\Free[A],A)$ is the infinite $A^{\pm}$-regular involutive automaton tree.
\end{exm}

\begin{figure}[H]
\centering
  \begin{tikzpicture}[shorten >=1pt, node distance=1.2 and 1.2, on grid,auto,>=stealth',transform shape]
  
  \node[state,accepting] (bp) {};

  \begin{scope}
  
  \newcommand{\dx}{1.6}
  \newcommand{\dy}{1.6}
  \newcommand{\dxx}{\dx*0.5}
  \newcommand{\dxxx}{\dx*0.3}
  \newcommand{\dxxxx}{\dx*0.18}
  \newcommand{\dyy}{\dy*0.5}
  \newcommand{\dyyy}{\dy*0.3}
  \newcommand{\dyyyy}{\dy*0.18}

   \node[smallstate] (1) {};
   \node[smallstate] (a) [right = \dx of 1]{};
   \node[smallstate] (aa) [right = \dxx of a]{};
   \node[tinystate] (aaa) [right = \dxxx of aa]{};
   \node[emptystate] (aaaa) [right = \dxxxx of aaa]{};
   
   \node[tinystate] (ab) [above = \dyy of a]{};
   \node[tinystate] (aab) [above = \dyyy of aa]{};
   \node[emptystate] (aabb) [above = \dyyyy of aab]{};
   \node[emptystate] (aaab) [above = \dyyyy of aaa]{};
   \node[emptystate] (aaaab) [above = \dyyyy of aaaa]{};
   
   \node[tinystate] (aB) [below = \dyy of a]{};
   \node[tinystate] (aaB) [below = \dyyy of aa]{};
   \node[emptystate] (aaBB) [below = \dyyyy of aaB]{};
   \node[emptystate] (aaaB) [below = \dyyyy of aaa]{};
   \node[emptystate] (aaaaB) [below = \dyyyy of aaaa]{};
   
   \node[tinystate] (aba) [right = \dxxx of ab]{};
   \node[tinystate] (abA) [left = \dxxx of ab]{};
   \node[emptystate] (aaba) [right = \dxxxx of aab]{};
   \node[emptystate] (aabA) [left = \dxxxx of aab]{};

   \node[tinystate] (aBa) [right = \dxxx of aB]{};
   \node[tinystate] (aBA) [left = \dxxx of aB]{};
   \node[emptystate] (aaBa) [right = \dxxxx of aaB]{};
   \node[emptystate] (aaBA) [left = \dxxxx of aaB]{};

   \node[tinystate] (abb) [above = \dyyy of ab]{};
   \node[tinystate] (aBB) [below = \dyyy of aB]{};
   
   \node[emptystate] (abbA) [left = \dxxxx of abb]{};
   \node[emptystate] (abbb) [above = \dyyyy of abb]{};
   \node[emptystate] (abba) [right = \dxxxx of abb]{};
   
   \node[emptystate] (aBBA) [left = \dxxxx of aBB]{};
   \node[emptystate] (aBBB) [below = \dyyyy of aBB]{};
   \node[emptystate] (aBBa) [right = \dxxxx of aBB]{};
   
   \node[emptystate] (abAA) [left = \dxxxx of abA]{};
   \node[emptystate] (abAb) [above = \dyyyy of abA]{};
   \node[emptystate] (abAB) [below = \dxxxx of abA]{};
   
   \node[emptystate] (aBAA) [left = \dxxxx of aBA]{};
   \node[emptystate] (aBAb) [above = \dyyyy of aBA]{};
   \node[emptystate] (aBAB) [below = \dxxxx of aBA]{};
   
   \node[emptystate] (abaa) [right = \dxxxx of aba]{};
   \node[emptystate] (abab) [above = \dyyyy of aba]{};
   \node[emptystate] (abaB) [below = \dxxxx of aba]{};
   
   \node[emptystate] (aBaa) [right = \dxxxx of aBa]{};
   \node[emptystate] (aBab) [above = \dyyyy of aBa]{};
   \node[emptystate] (aBaB) [below = \dxxxx of aBa]{};
     
   \path[->] (a) edge[red] (ab);
   \path[->] (aa) edge[red] (aab);
   \path[->,densely dotted] (aab) edge[red] (aabb);
   \path[->,densely dotted] (aaa) edge[red] (aaab);
   
   \path[->] (aB) edge[red] (a);
   \path[->] (aaB) edge[red] (aa);
   \path[->,densely dotted] (aaBB) edge[red] (aaB);
   \path[->,densely dotted] (aaaB) edge[red] (aaa);
   
   \path[->] (ab) edge[red] (abb);
   \path[->] (aBB) edge[red] (aB);
   
   \path[->] (1) edge[blue] node[pos = 0.45,above] {\scriptsize{$a$}}(a);
   \path[->] (a) edge[blue] (aa);
   \path[->] (aa) edge[blue] (aaa);
   \path[->,densely dotted] (aaa) edge[blue] (aaaa);
   
   \path[->] (abA) edge[blue] (ab);
   \path[->] (ab) edge[blue] (aba);
   \path[->] (ab) edge[blue] (aba);
   
   \path[->,densely dotted] (aabA) edge[blue] (aab);
   \path[->,densely dotted] (aab) edge[blue] (aaba);
   
   \path[->] (aBA) edge[blue] (aB);
   \path[->] (aB) edge[blue] (aBa);
   \path[->] (aB) edge[blue] (aBa);
   
   \path[->,densely dotted] (aaBA) edge[blue] (aaB);
   \path[->,densely dotted] (aaB) edge[blue] (aaBa);
   
   \path[->,densely dotted] (abb) edge[red] (abbb);
   \path[->,densely dotted] (abbA) edge[blue] (abb);
   \path[->,densely dotted] (abb) edge[blue] (abba);
   
   \path[->,densely dotted] (aBBB) edge[red] (aBB);
   \path[->,densely dotted] (aBBA) edge[blue] (aBB);
   \path[->,densely dotted] (aBB) edge[blue] (aBBa);
   
   \path[->,densely dotted] (abA) edge[red] (abAb);
   \path[->,densely dotted] (abAA) edge[blue] (abA);
   \path[->,densely dotted] (abAB) edge[red] (abA);
   
   \path[->,densely dotted] (aBA) edge[red] (aBAb);
   \path[->,densely dotted] (aBAA) edge[blue] (aBA);
   \path[->,densely dotted] (aBAB) edge[red] (aBA);
   
   \path[->,densely dotted] (aba) edge[red] (abab);
   \path[->,densely dotted] (aba) edge[blue] (abaa);
   \path[->,densely dotted] (abaB) edge[red] (aba);
   
   \path[->,densely dotted] (aBa) edge[red] (aBab);
   \path[->,densely dotted] (aBa) edge[blue] (aBaa);
   \path[->,densely dotted] (aBaB) edge[red] (aBa);
  \end{scope}
  
  \begin{scope}[rotate=90]
  
  \newcommand{\dx}{1.6}
  \newcommand{\dy}{1.6}
  \newcommand{\dxx}{\dx*0.5}
  \newcommand{\dxxx}{\dx*0.3}
  \newcommand{\dxxxx}{\dx*0.18}
  \newcommand{\dyy}{\dy*0.5}
  \newcommand{\dyyy}{\dy*0.3}
  \newcommand{\dyyyy}{\dy*0.18}

   \node[smallstate] (1) {};
   \node[smallstate] (a) [right = \dx of 1]{};
   \node[smallstate] (aa) [right = \dxx of a]{};
   \node[tinystate] (aaa) [right = \dxxx of aa]{};
   \node[emptystate] (aaaa) [right = \dxxxx of aaa]{};
   
   \node[tinystate] (ab) [above = \dyy of a]{};
   \node[tinystate] (aab) [above = \dyyy of aa]{};
   \node[emptystate] (aabb) [above = \dyyyy of aab]{};
   \node[emptystate] (aaab) [above = \dyyyy of aaa]{};
   \node[emptystate] (aaaab) [above = \dyyyy of aaaa]{};
   
   \node[tinystate] (aB) [below = \dyy of a]{};
   \node[tinystate] (aaB) [below = \dyyy of aa]{};
   \node[emptystate] (aaBB) [below = \dyyyy of aaB]{};
   \node[emptystate] (aaaB) [below = \dyyyy of aaa]{};
   \node[emptystate] (aaaaB) [below = \dyyyy of aaaa]{};
   
   \node[tinystate] (aba) [right = \dxxx of ab]{};
   \node[tinystate] (abA) [left = \dxxxx of ab]{};
   \node[emptystate] (aaba) [right = \dxxxx of aab]{};
   \node[emptystate] (aabA) [left = \dxxxx of aab]{};

   \node[tinystate] (aBa) [right = \dxxx of aB]{};
   \node[tinystate] (aBA) [left = \dxxx of aB]{};
   \node[emptystate] (aaBa) [right = \dxxxx of aaB]{};
   \node[emptystate] (aaBA) [left = \dxxxx of aaB]{};

   \node[tinystate] (abb) [above = \dyyy of ab]{};
   \node[tinystate] (aBB) [below = \dyyy of aB]{};
   
   \node[emptystate] (abbA) [left = \dxxxx of abb]{};
   \node[emptystate] (abbb) [above = \dyyyy of abb]{};
   \node[emptystate] (abba) [right = \dxxxx of abb]{};
   
   \node[emptystate] (aBBA) [left = \dxxxx of aBB]{};
   \node[emptystate] (aBBB) [below = \dyyyy of aBB]{};
   \node[emptystate] (aBBa) [right = \dxxxx of aBB]{};
   
   \node[emptystate] (abAA) [left = \dxxxx of abA]{};
   \node[emptystate] (abAb) [above = \dyyyy of abA]{};
   \node[emptystate] (abAB) [below = \dxxxx of abA]{};
   
   \node[emptystate] (aBAA) [left = \dxxxx of aBA]{};
   \node[emptystate] (aBAb) [above = \dyyyy of aBA]{};
   \node[emptystate] (aBAB) [below = \dxxxx of aBA]{};
   
   \node[emptystate] (abaa) [right = \dxxxx of aba]{};
   \node[emptystate] (abab) [above = \dyyyy of aba]{};
   \node[emptystate] (abaB) [below = \dxxxx of aba]{};
   
   \node[emptystate] (aBaa) [right = \dxxxx of aBa]{};
   \node[emptystate] (aBab) [above = \dyyyy of aBa]{};
   \node[emptystate] (aBaB) [below = \dxxxx of aBa]{};
     
   \path[->] (ab) edge[blue] (a);
   \path[->] (aab) edge[blue] (aa);
   \path[->,densely dotted] (aabb) edge[blue] (aab);
   \path[->,densely dotted] (aaab) edge[blue] (aaa);
   
   \path[->] (a) edge[blue] (aB);
   \path[->] (aa) edge[blue] (aaB);
   \path[->,densely dotted] (aaB) edge[blue] (aaBB);
   \path[->,densely dotted] (aaa) edge[blue] (aaaB);
   
   \path[->] (abb) edge[blue] (ab);
   \path[->] (aB) edge[blue] (aBB);
   
   \path[->] (1) edge[red]  node[pos = 0.45,below] {\rotatebox[origin=c]{-90}{\scriptsize{$b$}}} (a);
   \path[->] (a) edge[red] (aa);
   \path[->] (aa) edge[red] (aaa);
   \path[->,densely dotted] (aaa) edge[red] (aaaa);
   
   \path[->] (abA) edge[red] (ab);
   \path[->] (ab) edge[red] (aba);
   \path[->] (ab) edge[red] (aba);
   
   \path[->,densely dotted] (aabA) edge[red] (aab);
   \path[->,densely dotted] (aab) edge[red] (aaba);
   
   \path[->] (aBA) edge[red] (aB);
   \path[->] (aB) edge[red] (aBa);
   \path[->] (aB) edge[red] (aBa);
   
   \path[->,densely dotted] (aaBA) edge[red] (aaB);
   \path[->,densely dotted] (aaB) edge[red] (aaBa);
   
   \path[->,densely dotted] (abbb) edge[blue] (abb);
   \path[->,densely dotted] (abbA) edge[red] (abb);
   \path[->,densely dotted] (abb) edge[red] (abba);
   
   \path[->,densely dotted] (aBB) edge[blue] (aBBB);
   \path[->,densely dotted] (aBBA) edge[red] (aBB);
   \path[->,densely dotted] (aBB) edge[red] (aBBa);
   
   \path[->,densely dotted] (abAb) edge[blue] (abA);
   \path[->,densely dotted] (abAA) edge[red] (abA);
   \path[->,densely dotted] (abA) edge[blue] (abAB);
   
   \path[->,densely dotted] (aBAb) edge[blue] (aBA);
   \path[->,densely dotted] (aBAA) edge[red] (aBA);
   \path[->,densely dotted] (aBA) edge[blue] (aBAB);
   
   \path[->,densely dotted] (abab) edge[blue] (aba);
   \path[->,densely dotted] (aba) edge[red] (abaa);
   \path[->,densely dotted] (aba) edge[blue] (abaB);
   
   \path[->,densely dotted] (aBab) edge[blue] (aBa);
   \path[->,densely dotted] (aBa) edge[red] (aBaa);
   \path[->,densely dotted] (aBa) edge[blue] (aBaB);
  \end{scope}
  
  \begin{scope}[rotate=180]
  \newcommand{\dx}{1.6}
  \newcommand{\dy}{1.6}
  \newcommand{\dxx}{\dx*0.5}
  \newcommand{\dxxx}{\dx*0.3}
  \newcommand{\dxxxx}{\dx*0.18}
  \newcommand{\dyy}{\dy*0.5}
  \newcommand{\dyyy}{\dy*0.3}
  \newcommand{\dyyyy}{\dy*0.18}

  \node[smallstate] (1) {};
  \node[smallstate] (a) [right = \dx of 1]{};
  \node[smallstate] (aa) [right = \dxx of a]{};
  \node[tinystate] (aaa) [right = \dxxx of aa]{};
  \node[emptystate] (aaaa) [right = \dxxxx of aaa]{};
   
  \node[tinystate] (ab) [above = \dyy of a]{};
  \node[tinystate] (aab) [above = \dyyy of aa]{};
  \node[emptystate] (aabb) [above = \dyyyy of aab]{};
  \node[emptystate] (aaab) [above = \dyyyy of aaa]{};
  \node[emptystate] (aaaab) [above = \dyyyy of aaaa]{};
   
  \node[tinystate] (aB) [below = \dyy of a]{};
  \node[tinystate] (aaB) [below = \dyyy of aa]{};
  \node[emptystate] (aaBB) [below = \dyyyy of aaB]{};
  \node[emptystate] (aaaB) [below = \dyyyy of aaa]{};
  \node[emptystate] (aaaaB) [below = \dyyyy of aaaa]{};
   
  \node[tinystate] (aba) [right = \dxxx of ab]{};
  \node[tinystate] (abA) [left = \dxxx of ab]{};
  \node[emptystate] (aaba) [right = \dxxxx of aab]{};
  \node[emptystate] (aabA) [left = \dxxxx of aab]{};

  \node[tinystate] (aBa) [right = \dxxx of aB]{};
  \node[tinystate] (aBA) [left = \dxxx of aB]{};
  \node[emptystate] (aaBa) [right = \dxxxx of aaB]{};
  \node[emptystate] (aaBA) [left = \dxxxx of aaB]{};

  \node[tinystate] (abb) [above = \dyyy of ab]{};
  \node[tinystate] (aBB) [below = \dyyy of aB]{};
   
  \node[emptystate] (abbA) [left = \dxxxx of abb]{};
  \node[emptystate] (abbb) [above = \dyyyy of abb]{};
  \node[emptystate] (abba) [right = \dxxxx of abb]{};
   
  \node[emptystate] (aBBA) [left = \dxxxx of aBB]{};
  \node[emptystate] (aBBB) [below = \dyyyy of aBB]{};
  \node[emptystate] (aBBa) [right = \dxxxx of aBB]{};
   
  \node[emptystate] (abAA) [left = \dxxxx of abA]{};
  \node[emptystate] (abAb) [above = \dyyyy of abA]{};
  \node[emptystate] (abAB) [below = \dxxxx of abA]{};
   
  \node[emptystate] (aBAA) [left = \dxxxx of aBA]{};
  \node[emptystate] (aBAb) [above = \dyyyy of aBA]{};
  \node[emptystate] (aBAB) [below = \dxxxx of aBA]{};
   
  \node[emptystate] (abaa) [right = \dxxxx of aba]{};
  \node[emptystate] (abab) [above = \dyyyy of aba]{};
  \node[emptystate] (abaB) [below = \dxxxx of aba]{};
   
  \node[emptystate] (aBaa) [right = \dxxxx of aBa]{};
  \node[emptystate] (aBab) [above = \dyyyy of aBa]{};
  \node[emptystate] (aBaB) [below = \dxxxx of aBa]{};
 
  \path[->] (ab) edge[red] (a);
  \path[->] (aab) edge[red] (aa);
  \path[->,densely dotted] (aabb) edge[red] (aab);
  \path[->,densely dotted] (aaab) edge[red] (aaa);
   
  \path[->] (a) edge[red] (aB);
  \path[->] (aa) edge[red] (aaB);
  \path[->,densely dotted] (aaB) edge[red] (aaBB);
  \path[->,densely dotted] (aaa) edge[red] (aaaB);
   
  \path[->] (abb) edge[red] (ab);
  \path[->] (aB) edge[red] (aBB);
   
  \path[->] (a) edge[blue] (1);
  \path[->] (aa) edge[blue] (a);
  \path[->] (aaa) edge[blue] (aa);
  \path[->,densely dotted] (aaaa) edge[blue] (aaa);
   
  \path[->] (ab) edge[blue] (abA);
  \path[->] (aba) edge[blue] (ab);
  \path[->] (aba) edge[blue] (ab);
   
  \path[->,densely dotted] (aab) edge[blue] (aabA);
  \path[->,densely dotted] (aaba) edge[blue] (aab);
   
  \path[->] (aB) edge[blue] (aBA);
  \path[->] (aBa) edge[blue] (aB);
  \path[->] (aBa) edge[blue] (aB);
   
  \path[->,densely dotted] (aaB) edge[blue] (aaBA);
  \path[->,densely dotted] (aaBa) edge[blue] (aaB);
   
  \path[->,densely dotted] (abbb) edge[red] (abb);
  \path[->,densely dotted] (abb) edge[blue] (abbA);
  \path[->,densely dotted] (abba) edge[blue] (abb);
   
  \path[->,densely dotted] (aBB) edge[red] (aBBB);
  \path[->,densely dotted] (aBB) edge[blue] (aBBA);
  \path[->,densely dotted] (aBBa) edge[blue] (aBB);
   
  \path[->,densely dotted] (abAb) edge[red] (abA);
  \path[->,densely dotted] (abA) edge[blue] (abAA);
  \path[->,densely dotted] (abA) edge[red] (abAB);
   
  \path[->,densely dotted] (aBAb) edge[red] (aBA);
  \path[->,densely dotted] (aBA) edge[blue] (aBAA);
  \path[->,densely dotted] (aBA) edge[red] (aBAB);
   
  \path[->,densely dotted] (abab) edge[red] (aba);
  \path[->,densely dotted] (abaa) edge[blue] (aba);
  \path[->,densely dotted] (aba) edge[red] (abaB);
   
  \path[->,densely dotted] (aBab) edge[red] (aBa);
  \path[->,densely dotted] (aBaa) edge[blue] (aBa);
  \path[->,densely dotted] (aBa) edge[red] (aBaB);
  \end{scope}
  
  \begin{scope}[rotate=270]
  \newcommand{\dx}{1.6}
  \newcommand{\dy}{1.6}
  \newcommand{\dxx}{\dx*0.5}
  \newcommand{\dxxx}{\dx*0.3}
  \newcommand{\dxxxx}{\dx*0.18}
  \newcommand{\dyy}{\dy*0.5}
  \newcommand{\dyyy}{\dy*0.3}
  \newcommand{\dyyyy}{\dy*0.18}

  \node[smallstate] (1) {};
  \node[smallstate] (a) [right = \dx of 1]{};
  \node[smallstate] (aa) [right = \dxx of a]{};
  \node[tinystate] (aaa) [right = \dxxx of aa]{};
  \node[emptystate] (aaaa) [right = \dxxxx of aaa]{};
   
  \node[tinystate] (ab) [above = \dyy of a]{};
  \node[tinystate] (aab) [above = \dyyy of aa]{};
  \node[emptystate] (aabb) [above = \dyyyy of aab]{};
  \node[emptystate] (aaab) [above = \dyyyy of aaa]{};
  \node[emptystate] (aaaab) [above = \dyyyy of aaaa]{};
   
  \node[tinystate] (aB) [below = \dyy of a]{};
  \node[tinystate] (aaB) [below = \dyyy of aa]{};
  \node[emptystate] (aaBB) [below = \dyyyy of aaB]{};
  \node[emptystate] (aaaB) [below = \dyyyy of aaa]{};
  \node[emptystate] (aaaaB) [below = \dyyyy of aaaa]{};
   
  \node[tinystate] (aba) [right = \dxxx of ab]{};
  \node[tinystate] (abA) [left = \dxxx of ab]{};
  \node[emptystate] (aaba) [right = \dxxxx of aab]{};
  \node[emptystate] (aabA) [left = \dxxxx of aab]{};

  \node[tinystate] (aBa) [right = \dxxx of aB]{};
  \node[tinystate] (aBA) [left = \dxxx of aB]{};
  \node[emptystate] (aaBa) [right = \dxxxx of aaB]{};
  \node[emptystate] (aaBA) [left = \dxxxx of aaB]{};

  \node[tinystate] (abb) [above = \dyyy of ab]{};
  \node[tinystate] (aBB) [below = \dyyy of aB]{};
   
  \node[emptystate] (abbA) [left = \dxxxx of abb]{};
  \node[emptystate] (abbb) [above = \dyyyy of abb]{};
  \node[emptystate] (abba) [right = \dxxxx of abb]{};
   
  \node[emptystate] (aBBA) [left = \dxxxx of aBB]{};
  \node[emptystate] (aBBB) [below = \dyyyy of aBB]{};
  \node[emptystate] (aBBa) [right = \dxxxx of aBB]{};
   
  \node[emptystate] (abAA) [left = \dxxxx of abA]{};
  \node[emptystate] (abAb) [above = \dyyyy of abA]{};
  \node[emptystate] (abAB) [below = \dxxxx of abA]{};
   
  \node[emptystate] (aBAA) [left = \dxxxx of aBA]{};
  \node[emptystate] (aBAb) [above = \dyyyy of aBA]{};
  \node[emptystate] (aBAB) [below = \dxxxx of aBA]{};
   
  \node[emptystate] (abaa) [right = \dxxxx of aba]{};
  \node[emptystate] (abab) [above = \dyyyy of aba]{};
  \node[emptystate] (abaB) [below = \dxxxx of aba]{};
   
  \node[emptystate] (aBaa) [right = \dxxxx of aBa]{};
  \node[emptystate] (aBab) [above = \dyyyy of aBa]{};
  \node[emptystate] (aBaB) [below = \dxxxx of aBa]{};
 
  \path[->] (a) edge[blue] (ab);
  \path[->] (aa) edge[blue] (aab);
  \path[->,densely dotted] (aab) edge[blue] (aabb);
  \path[->,densely dotted] (aaa) edge[blue] (aaab);
   
  \path[->] (aB) edge[blue] (a);
  \path[->] (aaB) edge[blue] (aa);
  \path[->,densely dotted] (aaBB) edge[blue] (aaB);
  \path[->,densely dotted] (aaaB) edge[blue] (aaa);
   
  \path[->] (ab) edge[blue] (abb);
  \path[->] (aBB) edge[blue] (aB);
   
  \path[->] (a) edge[red] (1);
  \path[->] (aa) edge[red] (a);
  \path[->] (aaa) edge[red] (aa);
  \path[->,densely dotted] (aaaa) edge[red] (aaa);
   
  \path[->] (ab) edge[red] (abA);
  \path[->] (aba) edge[red] (ab);
  \path[->] (aba) edge[red] (ab);
   
  \path[->,densely dotted] (aab) edge[red] (aabA);
  \path[->,densely dotted] (aaba) edge[red] (aab);
   
  \path[->] (aB) edge[red] (aBA);
  \path[->] (aBa) edge[red] (aB);
  \path[->] (aBa) edge[red] (aB);
   
  \path[->,densely dotted] (aaB) edge[red] (aaBA);
  \path[->,densely dotted] (aaBa) edge[red] (aaB);
   
  \path[->,densely dotted] (abb) edge[blue] (abbb);
  \path[->,densely dotted] (abb) edge[red] (abbA);
  \path[->,densely dotted] (abba) edge[red] (abb);
   
  \path[->,densely dotted] (aBBB) edge[blue] (aBB);
  \path[->,densely dotted] (aBB) edge[red] (aBBA);
  \path[->,densely dotted] (aBBa) edge[red] (aBB);
   
  \path[->,densely dotted] (abA) edge[blue] (abAb);
  \path[->,densely dotted] (abA) edge[red] (abAA);
  \path[->,densely dotted] (abAB) edge[blue] (abA);
   
  \path[->,densely dotted] (aBA) edge[blue] (aBAb);
  \path[->,densely dotted] (aBA) edge[red] (aBAA);
  \path[->,densely dotted] (aBAB) edge[blue] (aBA);
   
  \path[->,densely dotted] (aba) edge[blue] (abab);
  \path[->,densely dotted] (abaa) edge[red] (aba);
  \path[->,densely dotted] (abaB) edge[blue] (aba);
   
  \path[->,densely dotted] (aBa) edge[blue] (aBab);
  \path[->,densely dotted] (aBaa) edge[red] (aBa);
  \path[->,densely dotted] (aBaB) edge[blue] (aBa);
  \end{scope}
   
\end{tikzpicture}
\caption{The Cayley automaton of $\Free_{\set{a,b}}$}
\label{fig: Sch(F_2)}
\end{figure}

\vspace{-5pt}
\begin{defn}
Let $A$ be an alphabet and let $a\in A$. The \defin{$a$-Cayley branch} of $\Free[A]$ is the connected component of $\cayley(\Free[A],A) \setmin F$ containing $\bp$, where $F$ is the set of arcs incident to $\bp$ except $\bp \xarc{a} \bullet$ and its inverse. 
\end{defn}
\begin{figure}[H] 
\centering
  \begin{tikzpicture}[shorten >=1pt, node distance=1.2 and 1.2, on grid,auto,>=stealth',transform shape]
  
  \node[state,accepting] (bp) {};
  \node[blue] (l) [above left = 0.13 and 0.8 of 1] {\scriptsize{$a$}};

  \begin{scope}[rotate=180]
  \newcommand{\dx}{1.6}
  \newcommand{\dy}{1.6}
  \newcommand{\dxx}{\dx*0.6}
  \newcommand{\dxxx}{\dx*0.3}
  \newcommand{\dxxxx}{\dx*0.18}
  \newcommand{\dyy}{\dy*0.5}
  \newcommand{\dyyy}{\dy*0.3}
  \newcommand{\dyyyy}{\dy*0.18}

  \node[smallstate] (1) {};
  \node[smallstate] (a) [right = \dx of 1]{};
  \node[smallstate] (aa) [right = \dxx of a]{};
  \node[tinystate] (aaa) [right = \dxxx of aa]{};
  \node[emptystate] (aaaa) [right = \dxxxx of aaa]{};
   
  \node[tinystate] (ab) [above = \dyy of a]{};
  \node[tinystate] (aab) [above = \dyyy of aa]{};
  \node[emptystate] (aabb) [above = \dyyyy of aab]{};
  \node[emptystate] (aaab) [above = \dyyyy of aaa]{};
  \node[emptystate] (aaaab) [above = \dyyyy of aaaa]{};
   
  \node[tinystate] (aB) [below = \dyy of a]{};
  \node[tinystate] (aaB) [below = \dyyy of aa]{};
  \node[emptystate] (aaBB) [below = \dyyyy of aaB]{};
  \node[emptystate] (aaaB) [below = \dyyyy of aaa]{};
  \node[emptystate] (aaaaB) [below = \dyyyy of aaaa]{};
   
  \node[tinystate] (aba) [right = \dxxx of ab]{};
  \node[tinystate] (abA) [left = \dxxx of ab]{};
  \node[emptystate] (aaba) [right = \dxxxx of aab]{};
  \node[emptystate] (aabA) [left = \dxxxx of aab]{};

  \node[tinystate] (aBa) [right = \dxxx of aB]{};
  \node[tinystate] (aBA) [left = \dxxx of aB]{};
  \node[emptystate] (aaBa) [right = \dxxxx of aaB]{};
  \node[emptystate] (aaBA) [left = \dxxxx of aaB]{};

  \node[tinystate] (abb) [above = \dyyy of ab]{};
  \node[tinystate] (aBB) [below = \dyyy of aB]{};
   
  \node[emptystate] (abbA) [left = \dxxxx of abb]{};
  \node[emptystate] (abbb) [above = \dyyyy of abb]{};
  \node[emptystate] (abba) [right = \dxxxx of abb]{};
   
  \node[emptystate] (aBBA) [left = \dxxxx of aBB]{};
  \node[emptystate] (aBBB) [below = \dyyyy of aBB]{};
  \node[emptystate] (aBBa) [right = \dxxxx of aBB]{};
   
  \node[emptystate] (abAA) [left = \dxxxx of abA]{};
  \node[emptystate] (abAb) [above = \dyyyy of abA]{};
  \node[emptystate] (abAB) [below = \dxxxx of abA]{};
   
  \node[emptystate] (aBAA) [left = \dxxxx of aBA]{};
  \node[emptystate] (aBAb) [above = \dyyyy of aBA]{};
  \node[emptystate] (aBAB) [below = \dxxxx of aBA]{};
   
  \node[emptystate] (abaa) [right = \dxxxx of aba]{};
  \node[emptystate] (abab) [above = \dyyyy of aba]{};
  \node[emptystate] (abaB) [below = \dxxxx of aba]{};
   
  \node[emptystate] (aBaa) [right = \dxxxx of aBa]{};
  \node[emptystate] (aBab) [above = \dyyyy of aBa]{};
  \node[emptystate] (aBaB) [below = \dxxxx of aBa]{};
 
  \path[->] (ab) edge[red] (a);
  \path[->] (aab) edge[red] (aa);
  \path[->,densely dotted] (aabb) edge[red] (aab);
  \path[->,densely dotted] (aaab) edge[red] (aaa);
   
  \path[->] (a) edge[red] (aB);
  \path[->] (aa) edge[red] (aaB);
  \path[->,densely dotted] (aaB) edge[red] (aaBB);
  \path[->,densely dotted] (aaa) edge[red] (aaaB);
   
  \path[->] (abb) edge[red] (ab);
  \path[->] (aB) edge[red] (aBB);
   
  \path[->] (a) edge[blue] (1);
  \path[->] (aa) edge[blue] (a);
  \path[->] (aaa) edge[blue] (aa);
  \path[->,densely dotted] (aaaa) edge[blue] (aaa);
   
  \path[->] (ab) edge[blue] (abA);
  \path[->] (aba) edge[blue] (ab);
  \path[->] (aba) edge[blue] (ab);
   
  \path[->,densely dotted] (aab) edge[blue] (aabA);
  \path[->,densely dotted] (aaba) edge[blue] (aab);
   
  \path[->] (aB) edge[blue] (aBA);
  \path[->] (aBa) edge[blue] (aB);
  \path[->] (aBa) edge[blue] (aB);
   
  \path[->,densely dotted] (aaB) edge[blue] (aaBA);
  \path[->,densely dotted] (aaBa) edge[blue] (aaB);
   
  \path[->,densely dotted] (abbb) edge[red] (abb);
  \path[->,densely dotted] (abb) edge[blue] (abbA);
  \path[->,densely dotted] (abba) edge[blue] (abb);
   
  \path[->,densely dotted] (aBB) edge[red] (aBBB);
  \path[->,densely dotted] (aBB) edge[blue] (aBBA);
  \path[->,densely dotted] (aBBa) edge[blue] (aBB);
   
  \path[->,densely dotted] (abAb) edge[red] (abA);
  \path[->,densely dotted] (abA) edge[blue] (abAA);
  \path[->,densely dotted] (abA) edge[red] (abAB);
   
  \path[->,densely dotted] (aBAb) edge[red] (aBA);
  \path[->,densely dotted] (aBA) edge[blue] (aBAA);
  \path[->,densely dotted] (aBA) edge[red] (aBAB);
   
  \path[->,densely dotted] (abab) edge[red] (aba);
  \path[->,densely dotted] (abaa) edge[blue] (aba);
  \path[->,densely dotted] (aba) edge[red] (abaB);
   
  \path[->,densely dotted] (aBab) edge[red] (aBa);
  \path[->,densely dotted] (aBaa) edge[blue] (aBa);
  \path[->,densely dotted] (aBa) edge[red] (aBaB);
  \end{scope}
   
\end{tikzpicture}
\caption{The $(a^{-1})$-Cayley branch of $\Free_{\set{a,b}}$}\label{fig: branca}
\end{figure}

Given that our ambient group is always a free group $G=\Free[A]$, we will typically  write $\schreier_{\Free[A]}(H,S)=\schreier(H,S)$, where $S$ is usually taken to be the ambient basis~$A$. We summarize the main properties of Schreier automata, which will allow us to relate them to $\stallings(H,A)$. 

\begin{prop}\label{prop: propietats Sch}
Let $H$ be a subgroup of $\Free[A]$. Then, 
\begin{enumerate} [ind]
\item \label{item: Sch inv,det,sat} $\schreier(H,A)$ is a pointed involutive $A$-automaton which is deterministic and saturated;
\item \label{item: Sch connected} $\schreier(H,A)$ is connected;
\item \label{item: <Sch>=H} $\gen{\schreier(H,A)}=H$;
\item \label{item: Sch no core} $\schreier(H,A)$ is not necessarily core.
\end{enumerate}
\end{prop}

\begin{proof} Assertion \ref{item: Sch inv,det,sat} follows immediately from \Cref{def: Sch} and the observation that, for every $a\in A^{\pm}$, $Hga=Hg'a$ implies $Hg=Hg'$. 

In $\schreier(H,A)$, given any vertex $Hw$, let $w=a_{i_1}^{\varepsilon_1} a_{i_2}^{\varepsilon_2} \cdots a_{i_l}^{\varepsilon_l}$ and the walk  
\begin{equation} \label{eq: Sch connex}
\walki \colon \bp \! \xarc{\raisebox{0.6ex}{$\scriptstyle{a_{i_1}^{\varepsilon_1}}$}} Ha_{i_1}^{\varepsilon_1} \xarc{\raisebox{0.6ex}{$\scriptstyle{a_{i_2}^{\varepsilon_2}}$}} Ha_{i_1}^{\varepsilon_1}a_{i_2}^{\varepsilon_2}  \xarc{\raisebox{0.6ex}{$\scriptstyle{a_{i_3}^{\varepsilon_3}}$}} \ \cdots \ \xarc{\raisebox{0.6ex}{$\scriptstyle{a_{i_l}^{\varepsilon_l}}$}} Hw
 \end{equation}
connects the basepoint to $Hw$. This shows point~\ref{item: Sch connected}.

Assertion \ref{item: <Sch>=H} follows from the fact that the walk $\walki$ in~\eqref{eq: Sch connex} is closed if and only if $Hw=\bp=H$, that is, if and only if $\walki\rlab =w\in H$.


Finally, to see~\ref{item: Sch no core}, just observe the $\{a,b\}$-automaton $\schreier(\gen{b}, \{a,b\})$ depicted in \Cref{fig: Sch(<b>)}: a single $b$-loop at $\bp$ (coinciding with $\core (\schreier(\gen{b},\{a,b\}))$), and an $a$-Cayley branch and an $a^{-1}$-Cayley branch growing from the corresponding two $a$-arcs available at the basepoint. 
\end{proof}
\vspace{5pt}
\begin{figure}[H]\label{fig: Stallings <<b>>}
\centering
  \begin{tikzpicture}[shorten >=1pt, node distance=1.2 and 1.2, on grid,auto,>=stealth',transform shape]
  
  \node[state,accepting] (bp) {};
  
  \path[->] (bp)
            edge[red,loop right,min distance=10mm,in=90+35,out=90-35]
            node[above] {\scriptsize{$b$}}
            (bp);
  
  \begin{scope}
  
  \newcommand{\dx}{1.6}
  \newcommand{\dy}{1.6}
  \newcommand{\dxx}{\dx*0.55}
  \newcommand{\dxxx}{\dx*0.3}
  \newcommand{\dxxxx}{\dx*0.18}
  \newcommand{\dyy}{\dy*0.55}
  \newcommand{\dyyy}{\dy*0.3}
  \newcommand{\dyyyy}{\dy*0.18}

   \node[smallstate] (1) {};
   \node[smallstate] (a) [right = \dx of 1]{};
   \node[smallstate] (aa) [right = \dxx of a]{};
   \node[tinystate] (aaa) [right = \dxxx of aa]{};
   \node[emptystate] (aaaa) [right = \dxxxx of aaa]{};
   
   \node[tinystate] (ab) [above = \dyy of a]{};
   \node[tinystate] (aab) [above = \dyyy of aa]{};
   \node[emptystate] (aabb) [above = \dyyyy of aab]{};
   \node[emptystate] (aaab) [above = \dyyyy of aaa]{};
   \node[emptystate] (aaaab) [above = \dyyyy of aaaa]{};
   
   \node[tinystate] (aB) [below = \dyy of a]{};
   \node[tinystate] (aaB) [below = \dyyy of aa]{};
   \node[emptystate] (aaBB) [below = \dyyyy of aaB]{};
   \node[emptystate] (aaaB) [below = \dyyyy of aaa]{};
   \node[emptystate] (aaaaB) [below = \dyyyy of aaaa]{};
   
   \node[tinystate] (aba) [right = \dxxx of ab]{};
   \node[tinystate] (abA) [left = \dxxx of ab]{};
   \node[emptystate] (aaba) [right = \dxxxx of aab]{};
   \node[emptystate] (aabA) [left = \dxxxx of aab]{};

   \node[tinystate] (aBa) [right = \dxxx of aB]{};
   \node[tinystate] (aBA) [left = \dxxx of aB]{};
   \node[emptystate] (aaBa) [right = \dxxxx of aaB]{};
   \node[emptystate] (aaBA) [left = \dxxxx of aaB]{};

   \node[tinystate] (abb) [above = \dyyy of ab]{};
   \node[tinystate] (aBB) [below = \dyyy of aB]{};
   
   \node[emptystate] (abbA) [left = \dxxxx of abb]{};
   \node[emptystate] (abbb) [above = \dyyyy of abb]{};
   \node[emptystate] (abba) [right = \dxxxx of abb]{};
   
   \node[emptystate] (aBBA) [left = \dxxxx of aBB]{};
   \node[emptystate] (aBBB) [below = \dyyyy of aBB]{};
   \node[emptystate] (aBBa) [right = \dxxxx of aBB]{};
   
   \node[emptystate] (abAA) [left = \dxxxx of abA]{};
   \node[emptystate] (abAb) [above = \dyyyy of abA]{};
   \node[emptystate] (abAB) [below = \dxxxx of abA]{};
   
   \node[emptystate] (aBAA) [left = \dxxxx of aBA]{};
   \node[emptystate] (aBAb) [above = \dyyyy of aBA]{};
   \node[emptystate] (aBAB) [below = \dxxxx of aBA]{};
   
   \node[emptystate] (abaa) [right = \dxxxx of aba]{};
   \node[emptystate] (abab) [above = \dyyyy of aba]{};
   \node[emptystate] (abaB) [below = \dxxxx of aba]{};
   
   \node[emptystate] (aBaa) [right = \dxxxx of aBa]{};
   \node[emptystate] (aBab) [above = \dyyyy of aBa]{};
   \node[emptystate] (aBaB) [below = \dxxxx of aBa]{};
     
   \path[->] (a) edge[red] node[pos = 0.4,left] {\scriptsize{$b$}} (ab);
   \path[->] (aa) edge[red] (aab);
   \path[->,densely dotted] (aab) edge[red] (aabb);
   \path[->,densely dotted] (aaa) edge[red] (aaab);
   
   \path[->] (aB) edge[red] (a);
   \path[->] (aaB) edge[red] (aa);
   \path[->,densely dotted] (aaBB) edge[red] (aaB);
   \path[->,densely dotted] (aaaB) edge[red] (aaa);
   
   \path[->] (ab) edge[red] (abb);
   \path[->] (aBB) edge[red] (aB);
   
   \path[->] (1) edge[blue] node[pos = 0.4,above] {\scriptsize{$a$}}(a);
   \path[->] (a) edge[blue] (aa);
   \path[->] (aa) edge[blue] (aaa);
   \path[->,densely dotted] (aaa) edge[blue] (aaaa);
   
   \path[->] (abA) edge[blue] (ab);
   \path[->] (ab) edge[blue] (aba);
   \path[->] (ab) edge[blue] (aba);
   
   \path[->,densely dotted] (aabA) edge[blue] (aab);
   \path[->,densely dotted] (aab) edge[blue] (aaba);
   
   \path[->] (aBA) edge[blue] (aB);
   \path[->] (aB) edge[blue] (aBa);
   \path[->] (aB) edge[blue] (aBa);
   
   \path[->,densely dotted] (aaBA) edge[blue] (aaB);
   \path[->,densely dotted] (aaB) edge[blue] (aaBa);
   
   \path[->,densely dotted] (abb) edge[red] (abbb);
   \path[->,densely dotted] (abbA) edge[blue] (abb);
   \path[->,densely dotted] (abb) edge[blue] (abba);
   
   \path[->,densely dotted] (aBBB) edge[red] (aBB);
   \path[->,densely dotted] (aBBA) edge[blue] (aBB);
   \path[->,densely dotted] (aBB) edge[blue] (aBBa);
   
   \path[->,densely dotted] (abA) edge[red] (abAb);
   \path[->,densely dotted] (abAA) edge[blue] (abA);
   \path[->,densely dotted] (abAB) edge[red] (abA);
   
   \path[->,densely dotted] (aBA) edge[red] (aBAb);
   \path[->,densely dotted] (aBAA) edge[blue] (aBA);
   \path[->,densely dotted] (aBAB) edge[red] (aBA);
   
   \path[->,densely dotted] (aba) edge[red] (abab);
   \path[->,densely dotted] (aba) edge[blue] (abaa);
   \path[->,densely dotted] (abaB) edge[red] (aba);
   
   \path[->,densely dotted] (aBa) edge[red] (aBab);
   \path[->,densely dotted] (aBa) edge[blue] (aBaa);
   \path[->,densely dotted] (aBaB) edge[red] (aBa);
  \end{scope}

  \begin{scope}[rotate=180]
  \newcommand{\dx}{1.6}
  \newcommand{\dy}{1.6}
  \newcommand{\dxx}{\dx*0.55}
  \newcommand{\dxxx}{\dx*0.3}
  \newcommand{\dxxxx}{\dx*0.18}
  \newcommand{\dyy}{\dy*0.55}
  \newcommand{\dyyy}{\dy*0.3}
  \newcommand{\dyyyy}{\dy*0.18}

  \node[smallstate] (1) {};
  \node[smallstate] (a) [right = \dx of 1]{};
  \node[smallstate] (aa) [right = \dxx of a]{};
  \node[tinystate] (aaa) [right = \dxxx of aa]{};
  \node[emptystate] (aaaa) [right = \dxxxx of aaa]{};
   
  \node[tinystate] (ab) [above = \dyy of a]{};
  \node[tinystate] (aab) [above = \dyyy of aa]{};
  \node[emptystate] (aabb) [above = \dyyyy of aab]{};
  \node[emptystate] (aaab) [above = \dyyyy of aaa]{};
  \node[emptystate] (aaaab) [above = \dyyyy of aaaa]{};
   
  \node[tinystate] (aB) [below = \dyy of a]{};
  \node[tinystate] (aaB) [below = \dyyy of aa]{};
  \node[emptystate] (aaBB) [below = \dyyyy of aaB]{};
  \node[emptystate] (aaaB) [below = \dyyyy of aaa]{};
  \node[emptystate] (aaaaB) [below = \dyyyy of aaaa]{};
   
  \node[tinystate] (aba) [right = \dxxx of ab]{};
  \node[tinystate] (abA) [left = \dxxx of ab]{};
  \node[emptystate] (aaba) [right = \dxxxx of aab]{};
  \node[emptystate] (aabA) [left = \dxxxx of aab]{};

  \node[tinystate] (aBa) [right = \dxxx of aB]{};
  \node[tinystate] (aBA) [left = \dxxx of aB]{};
  \node[emptystate] (aaBa) [right = \dxxxx of aaB]{};
  \node[emptystate] (aaBA) [left = \dxxxx of aaB]{};

  \node[tinystate] (abb) [above = \dyyy of ab]{};
  \node[tinystate] (aBB) [below = \dyyy of aB]{};
   
  \node[emptystate] (abbA) [left = \dxxxx of abb]{};
  \node[emptystate] (abbb) [above = \dyyyy of abb]{};
  \node[emptystate] (abba) [right = \dxxxx of abb]{};
   
  \node[emptystate] (aBBA) [left = \dxxxx of aBB]{};
  \node[emptystate] (aBBB) [below = \dyyyy of aBB]{};
  \node[emptystate] (aBBa) [right = \dxxxx of aBB]{};
   
  \node[emptystate] (abAA) [left = \dxxxx of abA]{};
  \node[emptystate] (abAb) [above = \dyyyy of abA]{};
  \node[emptystate] (abAB) [below = \dxxxx of abA]{};
   
  \node[emptystate] (aBAA) [left = \dxxxx of aBA]{};
  \node[emptystate] (aBAb) [above = \dyyyy of aBA]{};
  \node[emptystate] (aBAB) [below = \dxxxx of aBA]{};
   
  \node[emptystate] (abaa) [right = \dxxxx of aba]{};
  \node[emptystate] (abab) [above = \dyyyy of aba]{};
  \node[emptystate] (abaB) [below = \dxxxx of aba]{};
   
  \node[emptystate] (aBaa) [right = \dxxxx of aBa]{};
  \node[emptystate] (aBab) [above = \dyyyy of aBa]{};
  \node[emptystate] (aBaB) [below = \dxxxx of aBa]{};
 
  \path[->] (ab) edge[red] (a);
  \path[->] (aab) edge[red] (aa);
  \path[->,densely dotted] (aabb) edge[red] (aab);
  \path[->,densely dotted] (aaab) edge[red] (aaa);
   
  \path[->] (a) edge[red] (aB);
  \path[->] (aa) edge[red] (aaB);
  \path[->,densely dotted] (aaB) edge[red] (aaBB);
  \path[->,densely dotted] (aaa) edge[red] (aaaB);
   
  \path[->] (abb) edge[red] (ab);
  \path[->] (aB) edge[red] (aBB);
   
  \path[->] (a) edge[blue] (1);
  \path[->] (aa) edge[blue] (a);
  \path[->] (aaa) edge[blue] (aa);
  \path[->,densely dotted] (aaaa) edge[blue] (aaa);
   
  \path[->] (ab) edge[blue] (abA);
  \path[->] (aba) edge[blue] (ab);
  \path[->] (aba) edge[blue] (ab);
   
  \path[->,densely dotted] (aab) edge[blue] (aabA);
  \path[->,densely dotted] (aaba) edge[blue] (aab);
   
  \path[->] (aB) edge[blue] (aBA);
  \path[->] (aBa) edge[blue] (aB);
  \path[->] (aBa) edge[blue] (aB);
   
  \path[->,densely dotted] (aaB) edge[blue] (aaBA);
  \path[->,densely dotted] (aaBa) edge[blue] (aaB);
   
  \path[->,densely dotted] (abbb) edge[red] (abb);
  \path[->,densely dotted] (abb) edge[blue] (abbA);
  \path[->,densely dotted] (abba) edge[blue] (abb);
   
  \path[->,densely dotted] (aBB) edge[red] (aBBB);
  \path[->,densely dotted] (aBB) edge[blue] (aBBA);
  \path[->,densely dotted] (aBBa) edge[blue] (aBB);
   
  \path[->,densely dotted] (abAb) edge[red] (abA);
  \path[->,densely dotted] (abA) edge[blue] (abAA);
  \path[->,densely dotted] (abA) edge[red] (abAB);
   
  \path[->,densely dotted] (aBAb) edge[red] (aBA);
  \path[->,densely dotted] (aBA) edge[blue] (aBAA);
  \path[->,densely dotted] (aBA) edge[red] (aBAB);
   
  \path[->,densely dotted] (abab) edge[red] (aba);
  \path[->,densely dotted] (abaa) edge[blue] (aba);
  \path[->,densely dotted] (aba) edge[red] (abaB);
   
  \path[->,densely dotted] (aBab) edge[red] (aBa);
  \path[->,densely dotted] (aBaa) edge[blue] (aBa);
  \path[->,densely dotted] (aBa) edge[red] (aBaB);
  \end{scope}
   
\end{tikzpicture}
\caption{The Schreier automaton for $\gen{b} \leqslant \Free_{\set{a,b}}$}
\label{fig: Sch(<b>)}
\end{figure}

From \Cref{prop: propietats Sch} it is clear that the property of being core is the only one missing for $\schreier(H,A)$ to be a reduced automaton recognizing $H$; the result below follows.


\begin{prop} \label{prop: St = core Sch}
Let $H$ be a subgroup of $\Free[A]$. The Stallings' automaton for $H$ is the core of its Schreier automaton \ie $\stallings(H,A)=\core(\schreier(H,A))$. \qed
\end{prop}

In analogy with this result, we define the \defin{strict Stallings' digraph} of $H$ with respect to $A$ to be $\stallings^*(H,A)=\core^*(\stallings(H,A))$.

\medskip
Finally, for later use, we highlight additional connections between the Schreier and the Stallings' automata of a subgroup. 

\begin{lem}\label{lem: St vs Sch}
Let $H$ be a subgroup of $\Free[A]$. Then,  
\begin{enumerate}[ind]
\item \label{item: St sat iff Sch cor} $\stallings(H,A) \text{ is saturated} \ \Leftrightarrow\  \stallings(H,A)=\schreier(H,A) 
\ \Leftrightarrow\  \schreier(H,A) \text{ is core;}$
\item \label{item: Sch = St + Cay branches} $\schreier(H,A)$ equals $\stallings(H,A)$ with a Cayley $a$-branch attached to each $a$-deficient vertex from $\stallings(H,A)$, for all $a\in A^\pm$.
\end{enumerate}
\end{lem}

\begin{proof}
To see~\ref{item: St sat iff Sch cor} we only need to recall that $\schreier(H)$ is always saturated, $\stallings(H)$ is always core, and $\stallings(H)=\core(\schreier(H))$.

As for~\ref{item: Sch = St + Cay branches}, if $\stallings(H,A)$ is saturated there is nothing to prove. Otherwise, note that if $\verti$ is an $a$-deficient vertex in $\stallings(H,A)$, then the $a$-arc in $\schreier(H,A)$ starting at $\verti$ does not belong to $\stallings(H,A)$. Now, the claim \ref{item: Sch = St + Cay branches} follows from the fact that every vertex in $\schreier(H,A) \setmin \stallings(H,A)$ is saturated and does not appear in any reduced walk in $\schreier(H,A) \setmin \stallings(H,A)$ (otherwise it would also appear in some reduced $\bp$-walk of $\schreier(H,A)$, in contradiction with not belonging to the core $\stallings(H,A)$).
\end{proof}

We conclude this section with an example illustrating that a change in the ambient basis may drastically affect the appearance of the Stallings' automaton associated to a given subgroup. In fact, not much is known about the function $\stallings(H,A) \mapsto \stallings(H,B)$, where $A,B$ are different bases of a free group $\Free$. If the reference basis $A$ is clear from the context, we will often abbreviate $\schreier(H) =\schreier(H,A)$, $\stallings(H)= \stallings(H,A)$, $\stallings^*(H)=\stallings^*(H,A)$, etc. 

\begin{exm}[Example~2.2 in~\cite{miasnikov_algebraic_2007}]
Let $\Free[3]$ be the free group with basis $A=\set{a,b,c}$ and let $H=\gen{ab,acba}$. It is easy to see that $B=\set{a',b',c'}$ is also a basis for $\Free[3]$, where $a'=a$, $b'=ab$, and $c'=acba$. The Stallings' automaton for~$H$ with respect to $A$ and to $B$ are represented in \Cref{fig: St vs St}.
\vspace{5pt}
\begin{figure}[H] 
\centering
  \begin{tikzpicture}[shorten >=1pt, node distance=1 and 2, on grid,auto,>=stealth']
  \begin{scope}
   \node[state,accepting] (0) {};
   \node[state] (1) [right = of 0]{};
   \node[state] (2) [below = of 1]{};
   \node[state] (3) [below = of 0]{};

    \path[->]
        (0) edge[blue, bend left = 20]
            node[above] {\scriptsize{$a$}}
            (1);
            
    \path[->]
        (1) edge[red, bend left = 20]
            node[below] {\scriptsize{$b$}}
            (0);
            
    \path[->]
        (1) edge[black]
            node[right] {\scriptsize{$c$}}
            (2);
            
    \path[->] (2) edge[red] (3);
    \path[->] (3) edge[blue] (0);
    \end{scope}
    
    \begin{scope}[xshift=5 cm,yshift=-0.5 cm]
    \node[state,accepting] (0) {};
    
    \path[->]
        (0) edge[Sepia,loop left,min distance=15mm,in=-35,out=35]
            node[right] {\scriptsize{$b'$}}
            (0);
            
    \path[->]
        (0) edge[OliveGreen,loop left,min distance=15mm,in=180-35,out=180+35]
            node[left] {\scriptsize{$c'$}}
            (0);           
    \end{scope}
\end{tikzpicture}
\vspace{-10pt}
\caption{The automata $\stallings(H,A)$ and $\stallings(H,B)$}
\label{fig: St vs St}
\end{figure}
\end{exm}

\section{The membership problem}\label{sec: MP}

One of the first algorithmic consequences of the Stallings' theory is a nice and efficient solution to the subgroup membership problem for free groups. We recall the general problem for an arbitrary finitely presented group $G=\pres{X}{R}$.

\begin{named}[Subgroup Membership problem for $G=\pres{X\!}{\!R}$, \MP(G)]
Decide, given finitely many words $w_0,w_1,\ldots,w_p \in (X^{\pm})^*$, whether (the element in $G$ represented by) $w_0$ belongs to the subgroup generated by (the elements in $G$ represented by) $w_1,\ldots,w_p$.
\end{named}

It is well known that, by a standard brute force argument based on Tietze transformations, the solvability of $\MP(G)$ does not depend on the particular finite presentation for $G$ in use. In our situation, for free groups $\Free[A]$ we will always consider finite presentations with no relations, $\Free[A]=\pres{A}{-}$, \ie will always work \wrt an ambient finite free basis $A$.

\begin{thm}\label{thm: memb}
The subgroup membership problem is decidable for $\Free[A]$.
\end{thm}

\begin{proof}
Suppose we are given a finite set $S=\{w_1,\ldots ,w_{p}\}$ generating $H\leqslant \Free[A]$ and $w_0\in \Free[A]$ (all of them are assumed to be reduced words since they can be reduced beforehand). Start building $\stallings(H,A)$ as explained in the previous section (note that, even if $A$ is infinite, only finitely many letters are in use). Then, try to read the candidate word $w_0$ as the label of a walk in $\stallings(H,A)$ starting at~$\bp$. The determinism of $\stallings(H,A)$ guarantees that, at each step, there is at most one possible arc labeled by the following letter. Then, only two situations are possible:
 \begin{enumerate*}[cases]
\item we get stuck before completing the reading of $w_0$; in this case $w_0\not\in H$ and we answer \nop;
\item we can complete the reading of the full word $w_0$; in this case, $w_0\in H$ if and only if the (unique possible) walk spelling $w_0$ and starting at $\bp$ terminates back at~$\bp$.
\end{enumerate*}
Hence, we return \yep\ if we are able to read $w_0$ as the label of a $\bp$-walk in $\stallings(H,A)$, and \nop\ otherwise. 
\end{proof}

\begin{exm}\label{ex: MP}
Let $A=\{a,b\}$, and consider the subgroup of $\Free[A]$ generated by $S=\set{w_1,w_2,w_3}$, where $w_1=a^3$, $w_2=abab^{-1}$, and $w_3=a^{-1}bab^{-1}$, and the set of candidates $W_0 =\set{a, b, bab^{-1}a^2b^{-1}} \subseteq \Free[A]$. To decide whether the candidates belong to $H$ we first construct $\stallings(H,A)$, as done in \Cref{ex: Stallings}. Then, clearly, $a$ spells a $\bp$-walk (of length $1$) in $\stallings(H,A)$; so, $a\in H$. Similarly, $b$ spells a walk in $\stallings(H,A)$ starting at $\bp$ but not ending at $\bp$; so, $b\not\in H$. Finally, starting at $\bp$ one can read the initial segment $bab^{-1}a^2$ from $bab^{-1}a^2b^{-1}$ as the label of a walk starting at $\bp$, but it is not possible to complete the full reading of $bab^{-1}a^2b^{-1}$; hence, $bab^{-1}a^2b^{-1}\not\in H$. 
\end{exm}

A natural complement to the decision problem $\MP(G)$ is the extra requirement that the algorithm, in case $w_0\in H=\gen{w_1,\ldots ,w_p}$, also outputs a word expressing $w_0$ in terms of $S=\set{w_1,\ldots ,w_p}$. This is always theoretically doable since, knowing that $w_0\in H$, one can always start a brute force search among all possible words on $S^{\pm}$ and wait for the guaranteed match. However, we can take profit of the Stallings' techniques to design a faster algorithm. The idea is very simple: realize $w_0$ as the label of a reduced $\bp$-walk in $\stallings(H,A)$ and `elevate' it up the folding sequence to the flower automaton $\flower(S)$. Now, observe that such a reduced $\bp$-walk is precisely a word in $S^{\pm}$ spelling $w_0$ modulo cancellation. 

\begin{defn}
Let $\Phi\colon \Ati \xto{} \Ati'$ be a homomorphism of $A$-automata, and let~$\walki$ be a reduced walk in $\Ati'$. Then, an \defin[elevation of a walk]{elevation of $\walki$ through $\Phi$} (or a \defin{$\Phi$-elevation} of~$\walki$) is any reduced walk $\widetilde{\walki}$ in $\Ati$ such that $\red{(\widetilde{\walki}) \Phi} =\walki$ (and hence, $(\widetilde{\walki}) \red{\lab}=(\walki) \red{\lab} \in \Free[A]$). 
\end{defn}

We start by showing that any reduced $\bp'$-walk $\walki$ in $\Ati'$ admits an elevation through any elementary folding~\smash{$\Ati \xtr{\!\phi\!} \Ati'$}. Let us fix notation: denote the folded arcs by $\edgii_1,\, \edgii_2\in \Edgs\Ati$, $\vertiii =\edgii_1\init =\edgii_2\init$, $\vertii_1 =\edgii_1 \term$, and $\vertii_2 =\edgii_2 \term$; and let us denote the folded resulting arc and vertex by $\edgii \in \Edgs \Ati'$ and $\vertii \in \Verts \Ati'$, respectively. To construct a $\phi$-elevation for $\walki$, we need to pay attention to its visits to $\vertii$: such a visit is called \defin[essential visit]{essential} if it is of the form $\edgi_1 \vertii \edgi_2$, where $\edgi_1,\edgi_2 \neq \edgii^{\pm 1}$ and $\edgi_1 \phi^{-1} \tau \neq \edgi_2 \phi^{-1} \iota$ (note that only open foldings entail essential visits).
 
Suppose that $\phi$ is open (\ie $\vertii_1\neq \vertii_2$) and let $\walki =\verti_0 \edgi_1 \verti_1 \cdots \verti_{l-1} \edgi_l \verti_l$ be a reduced $\bp'$-walk in $\Ati'$ (where $\verti_0=\bp'=\verti_l$). Elevate the vertices (resp., arcs) in $\walki$ verbatim to the `same' vertex (resp., arc) in $\widetilde{\walki}$ except for the essential visits of $\walki$ to $\vertii$, namely those of the form $\edgi_i \vertii \edgi_{i+1}$, where $\edgi_i, \edgi_{i+1}\neq \edgii^{\pm 1}$ and $\{\edgi_i\term,\, \edgi_{i+1}\init\}=\{\vertii_1,\, \vertii_2\}$, in which case we elevate either to $\edgi_i \vertii_1 \edgii_1^{-1} \vertiii \edgii_2 \vertii_2 \edgi_{i+1}$ or to $\edgi_i \vertii_2 \edgii_2^{-1} \vertiii \edgii_1 \vertii_1 \edgi_{i+1}$ depending on whether $\edgi_i \term =\vertii_1$ or $\edgi_i \term =\vertii_2$ in $\Ati$ (including the degenerate case where $\vertii_1=\bp$ and $\widetilde{\walki}$ starts or ends with an edge incident to $\vertii_2$, or vice-versa). It is clear that the $\bp$-walk $\widetilde{\walki}$ constructed in this way is reduced and satisfies $\red{(\widetilde{\walki})\phi} =\walki$; therefore, it is an elevation of $\walki$ through the open folding $\phi$. (We do not care now about uniqueness, but we will see later that $\widetilde{\walki}$ is in fact the unique possible elevation for $\walki$ if $\phi$ is open.)   

Suppose now that $\phi$ is closed (\ie $\vertii_1 =\vertii_2 =\vertii$). In this case, there is an obvious elevation of $\walki$ consisting of elevating every vertex and arc in $\walki$ verbatim, except for each occurrence of the folded arc $\edgii$, which can be elevated indistinctly to either $\edgii_1$ or $\edgii_2$. Note the crucial fact that, in this case, there are infinitely many possible elevations: in fact, every visit to $\vertii$, say $\edgi_i \vertii \edgi_{i+1}$, can be elevated to $\edgi_i \vertii_1 \xi \vertii_1 \edgi_{i+1}$, where $\xi$ is a $\vertii_1$-walk in $\Ati$ satisfying $\xi\rlab =1$ (consider, for example, $\xi =(\edgii_1^{-1}\vertiii \edgii_2)^n$ for any integer $n\in \mathbb{Z}$). This non-uniqueness of elevations through closed foldings has a neat algebraic interpretation described later. 

\begin{defn}
Let $\phi\colon \Ati \xto{} \Ati'$ be a closed folding identifying the arcs $\edgii_1$ and~$\edgii_2$. Then, a (non-trivial) reduced \mbox{$\bp$-walk} in $\Ati$ of the form
\vspace{3pt}
 \[
\bp \xwalk{\gamma_0} \bullet \xarc{\edgii_2} \bullet  \xarc{\edgii_1^{-1}} \bullet \xwalk{\gamma_0^{-1}} \bp
 \]
is called an \defin[elementary walk]{elementary walk associated to} $\phi$. Moreover, if $\Phi\colon \Ati_{\!0} \xto{} \Ati$ is a homomorphism of automata, then any $\Phi$-elevation of an elementary walk associated to $\phi$ is called an \defin{elementary $\Phi$-elevation} of $\phi$.
\end{defn}
\vspace{-5pt}
\begin{figure}[H]
\centering
\begin{tikzpicture}[shorten >=3pt, node distance=.3cm and 1cm, on grid,auto,>=stealth',decoration={snake, segment length=2mm, amplitude=0.3mm,post length=1.5mm}]
   \node[state,accepting] (0) {};
   \node[state] (1) [right = 2.15 of 0] {};
   \node[state] (2) [right = 1.5 of 1] {};
   \path[->]
        ([yshift=0.5ex]0.east) edge[snake it]
            node[pos=0.5,above] {$\gamma_0$}
            ([yshift=0.3ex]1)
        ([yshift=-0.5ex]1.west) edge[snake it]
            node[pos=0.5,below] {$\gamma_0^{-1}$}
            ([yshift=-0.3ex]0);
    \path[->]
        (1) edge[bend left] 
        node[pos=0.5,above] {$\edgii_2$}
        (2);
    \path[->]
        (2) edge[bend left] 
        node[pos=0.5,below] {$\edgii_1^{-1}$}
        (1);
\end{tikzpicture}
\vspace{-5pt}
\caption{An elementary walk associated to a closed folding}\label{fig: involutive a-arc}
\end{figure}
\vspace{-5pt}

We illustrate this process of elevation of reduced $\bp$-walks in the context of the example above. 

\begin{exm}
Let us recover \Cref{ex: Stallings} and find an expression for $w_0=a$ as a word on $w_1=a^3$, $w_2=abab^{-1}$, $w_3=a^{-1}bab^{-1}$ (and their inverses). 

Consider the sequence of foldings in \Cref{fig: Stallings sequence}. 
Note that the elementary foldings on it are all open except for $\Ati_{\hspace{-1pt} 4} \xtr{} \Ati_{\hspace{-1pt} 5}$, which is closed because it identifies two $a$-loops into a single one. Clearly $a\in H$, since it can be realized as the label of the $\bp$-walk $\walki_6 =a_1$ in $\Ati_{\hspace{-1pt} 6}=\stallings(H,A)$; this walk is depicted using a
dashed gray line:
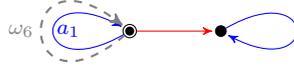
\vspace{-15pt}
\begin{figure}[H]
\centering
\begin{tikzpicture}[shorten >=1pt, node distance=1cm and 1cm, on grid,auto,auto,>=stealth']       
\newcommand{\rad}{1.2}
\node[state,accepting] (0)  {};
\node[state] (1) [right = \rad of 0] {};

\draw[->,blue,loop right,min distance=12 * \rad mm,in=150,out=210] (0) edge 
node[pos=0.5,right=-.3mm] {$a_1$} (0);
\draw[->,gray,loop right,min distance= 16 * \rad mm,in=140,out=220] (0) edge[dashed,thick]
node[left] {\scriptsize{$\walki_6$}} (0);
\draw[->,blue,loop right,min distance= 12 * \rad mm,in=-30,out=30] (1) edge (1);

\draw[->,red] (0) edge (1);
\end{tikzpicture}
\end{figure}
\vspace{-20pt}
Note that the $a$-arc incident to $\bp$ (\ie the only one traversed by $\walki_6$) has been denoted by $a_1$; we will denote some arcs by their corresponding labeling letter together with certain numerical subscripts keeping track of the elevation process (the concrete notational convention will be clear below); the unused arcs will receive no special name. 

As a first step, we elevate the walk $\walki_6$ up trough the last folding $\Ati_{\hspace{-1pt} 5} \xtr{\,} \Ati_{\hspace{-1pt} 6}$; since $\walki_6$ does not visit the identified vertex, it elevates to 
$\walki_5 =a_1$ in $\Ati_{\hspace{-1pt} 5}$ (we abuse language and denote the visited edge again by $a_1$):
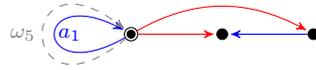
\vspace{-17pt}
\begin{figure}[H]
\centering
\begin{tikzpicture}[shorten >=1pt, node distance=1.2cm and 1.2cm, on grid,auto,auto,>=stealth']       
\newcommand{\rad}{1.2}
\node[state,accepting] (0)  {};
\node[state] (1) [right = \rad of 0] {};
\node[state] (2) [right = \rad of 1] {};

\draw[->,blue,loop right,min distance=12 * \rad mm,in=150,out=210] (0) edge node[pos=0.5,right=-.3mm] {$a_1$} (0);
\draw[->,gray,loop right,min distance= 16 * \rad mm,in=140,out=220] (0) edge[dashed]
node[left] {\scriptsize{$\walki_5$}} (0);

\draw[->,red] (0) edge (1);
\draw[->,blue] (2) edge (1);
\draw[->,red,bend left] (0) edge (2);

\end{tikzpicture}
\end{figure}
\vspace{-17pt}
Observe now the closed folding $\Ati_{\hspace{-1pt} 4} \xtr{\,} \Ati_{\hspace{-1pt} 5}$. We will use a second subscript and denote by $a_{11}$ and $a_{12}$ the two $a$-loops in $\Ati_{\hspace{-1pt} 4}$ which get folded into $a_1$. Note that $\walki_5$ elevates through $\Ati_{\hspace{-1pt} 4} \xtr{\,} \Ati_{\hspace{-1pt} 5}$ in a non-unique way; the simplest elevations are those using one of the folded $a$-loops, for example, $\walki_4=a_{11}$:
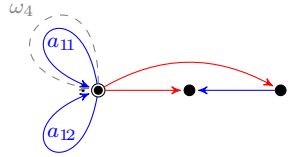
\vspace{-15pt}
\begin{figure}[H]
\centering
\begin{tikzpicture}[shorten >=1pt, node distance=1cm and 1cm, on grid,auto,auto,>=stealth']       
\newcommand{\rad}{1.2}
\node[state,accepting] (0)  {};
\node[state] (1) [right = \rad of 0] {};
\node[state] (2) [right = \rad of 1] {};
\node[] (00) [left = 0.2 of 0] {.};

\draw[->,gray,loop left,min distance= 16 * \rad mm,in=170,out=90] (0) edge[dashed]
node[above left] {\scriptsize{$\walki_4$}} (0);

\draw[->,blue,loop right,min distance=12 * \rad mm,in=160,out=100] (0) edge node[pos=0.5, below right = 0pt and -3pt] {$a_{1\!1}$} (0);
\draw[->,blue,loop right,min distance=12 * \rad mm,in=-160,out=-100] (0) edge node[pos=0.5,above right = 1pt and -3pt] {$a_{1\!2}$} (0);

\draw[->,red] (0) edge (1);
\draw[->,blue] (2) edge (1);
\draw[->,red,bend left] (0) edge (2);

\end{tikzpicture}
\end{figure}
\vspace{-15pt}
\noindent (we could have used any other possibility, for example, $\walki_4=a_{12}$, $\walki_4=a_{11}a_{12}a_{11}^{-1}$, $\walki_4=a_{11}a_{12}^{-1}a_{11}$, etc.). The next step is the elevation to $\Ati_{\hspace{-1pt} 3}$: denote by $a_{121}$ and $a_{122}$ the two $a$-arcs from $\Ati_{\hspace{-1pt} 3}$ which get folded into $a_{12}$, and let us elevate $\walki_4$ up to $\walki_3=a_{11}a_{122}^{-1}a_{121}$:
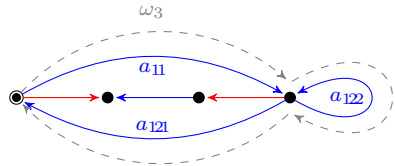
\begin{figure}[H]
\centering
\begin{tikzpicture}[shorten >=1pt, node distance=1cm and 1cm, on grid,auto,auto,>=stealth']       
\newcommand{\rad}{1.2}
\node[state,accepting] (0)  {};
\node[state] (1) [right = \rad of 0] {};
\node[state] (2) [right = \rad of 1] {};
\node[state] (3) [right = \rad of 2] {};
\node[inner sep=0pt] (33) [above = 0.2 of 3]{};
\node[inner sep=0pt] (333) [below = 0.2 of 3]{};

\draw[->,red] (0) edge (1);
\draw[->,blue] (2) edge (1);
\draw[->,red] (3) edge (2);
\draw[->,blue,bend left] (0) edge node[pos=0.5, below] {$a_{1\!1}$} (3);
\draw[->,gray,bend left, in =135,out=45] (0) edge[dashed] node[above = 0.1] {\scriptsize{$\walki_3$}} (33);

\draw[->,blue,bend left] (3) edge node[pos=0.5, above] {$a_{1\!2\!1}$} (0);
\draw[->,gray,bend left, in =135,out=45] (333) edge[dashed] (0);

\draw[->,blue,loop right,min distance=13 * \rad mm,in=30,out=-30] (3) edge node[pos=0.5, left] {$a_{1\!2\!2}$} (3);

\draw[->,gray,loop left,min distance= 17 * \rad mm,in=-30,out=30] (33) edge[dashed] (333);

\end{tikzpicture}
\end{figure}
\vspace{-5pt}
Now, denoting by $a_{1211}$ and $a_{1212}$ the two $a$-arcs in $\Ati_{\hspace{-1pt} 2}$ which get folded into $a_{121}$, the $\bp$-walk $\walki_3$ elevates to $\Ati_{\hspace{-1pt} 2}$ as $\walki_2=a_{11}a_{1211}a_{1212}^{-1} a_{122}^{-1}a_{1211}$:
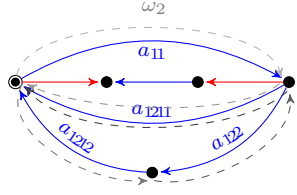
\begin{figure}[H]
\centering
\begin{tikzpicture}[shorten >=1pt, node distance=1cm and 1cm, on grid,auto,auto,>=stealth']       
\newcommand{\rad}{1.2}
\node[state,accepting] (0)  {};
\node[state] (1) [right = \rad of 0] {};
\node[state] (2) [right = \rad of 1] {};
\node[state] (3) [right = \rad of 2] {};
\node[state] (4) [below right = \rad and \rad/2 of 1] {};
\node[inner sep=0pt] (00) [right = 0.1 of 0]{};
\node[inner sep=0pt] (33) [left = 0.1 of 3]{};
\node[inner sep=0pt] (44) [below = 0.1 of 4]{};
\node[inner sep=0pt] (333) [below = 0.1 of 3]{};

\draw[->,red] (0) edge (1);
\draw[->,blue] (2) edge (1);
\draw[->,red] (3) edge (2);
\draw[->,blue,bend left] (0) edge node[pos=0.5, below] {$a_{1\!1}$} (3);
\draw[->,blue,bend left] (3) edge node[pos=0.5, above] {$a_{1\!2\!1\!1}$} (0);
\draw[->,blue,bend left] (3) edge node[sloped, anchor=center, above] {$a_{1\!2\!2}$} (4);
\draw[->,blue,bend left] (4) edge node[sloped, anchor=center, above] {$a_{1\!2\!1\!2}$} (0);

\draw[->,darkgray!40,bend left, in =90,out=80,max distance= 9 mm] (0) edge[dashed] node[above = 0.1] {\color{gray}\scriptsize{$\walki_2$}} (33);
\draw[->,darkgray!55,bend left, in =150,out=40,max distance= 7.5 mm] (33) edge[dashed] (00);
\draw[->,darkgray!70,bend left, in =225,out=-80,max distance= 7.5 mm] (00) edge[dashed] (44);

\draw[->,darkgray!85,bend left, in =250,out=-45,max distance= 7 mm] (44) edge[dashed] (333);
\draw[->,darkgray,bend left, in =120,out=45,max distance= 9 mm] (333) edge[dashed] (00);

\end{tikzpicture}
\end{figure}
\noindent The next step is to elevate $\walki_2$ up in $\Ati_{\hspace{-1pt} 1}$ as $\walki_1=a_{111}a_{1211} a_{1212}^{-1} a_{122}^{-1} a_{112}^{-1}a_{111}a_{1211}$:
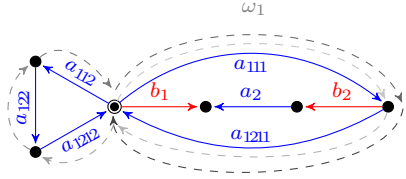
\begin{figure}[H]
\centering
\begin{tikzpicture}[shorten >=1pt, node distance=1cm and 1cm, on grid,auto,auto,>=stealth']       
\newcommand{\rad}{1.2}
\node[state,accepting] (0)  {};
\node[state] (1) [right = \rad of 0] {};
\node[state] (2) [right = \rad of 1] {};
\node[state] (3) [right = \rad of 2] {};
\node[state] (5) at ($(180+30:\rad cm)$) {};
\node[state] (6) at ($(180-30:\rad cm)$) {};
\node[inner sep=0pt] (00) [below = 0.2 of 0] {};
\node[inner sep=0pt] (000) [above = 0.2 of 0] {};
\node[inner sep=0pt] (33) [right = 0.2 of 3] {};

\draw[->,blue] (0) edge node[sloped, anchor=center, above] {$a_{1\!1\!2}$} (6);
\draw[->,blue] (6) edge node[sloped, anchor=center, below] {\rotatebox{180}{$a_{1\!2\!2}$}} (5);
\draw[->,blue] (5) edge node[sloped, anchor=center, below] {$a_{1\!2\!1\!2}$} (0);
\draw[->,red] (0) edge node[above] {$b_{1}$}  (1);
\draw[->,blue] (2) edge node[sloped, anchor=center, above] {$a_{2}$} (1);
\draw[->,red] (3) edge node[above] {$b_{2}$} (2);
\draw[->,blue,bend left, out=40,in=140] (0) edge node[sloped, anchor=center, below] {$a_{1\!1\!1}$} (3);
\draw[->,blue,bend left] (3) edge node[sloped, anchor=center, above] {$a_{1\!2\!1\!1}$} (0);

\draw[->,darkgray!30,bend left, in =110,out=65,max distance= 11 mm] (0) edge[dashed] node[above = 0.3] {\color{gray}\scriptsize{$\walki_1$}} (3);

\draw[->,darkgray!40,bend left, in = 130,out=75,max distance= 8 mm] (3) edge[dashed]  (00);

\draw[->,darkgray!50,bend left, in = 120,out=50,max distance= 9 mm] (00) edge[dashed] (5);

\draw[->,darkgray!60,bend left, in = 120,out=65,max distance= 9 mm] (5) edge[dashed] (6);

\draw[->,darkgray!70,bend left, in = 150,out=65,max distance= 9 mm] (6) edge[dashed] (000);

\draw[->,darkgray!80,bend left, in =110,out=65,max distance= 12 mm] (000) edge[dashed] (33);

\draw[->,darkgray,bend left, in = 80,out=80,max distance= 11 mm] (33) edge[dashed]  (0);
\end{tikzpicture}
\end{figure}
Finally, denoting $a_2$, $b_1$, and $b_2$ the arcs in $\Ati_{\hspace{-1pt} 1}$ yet not used, and keeping the notational convention for the corresponding arcs in $\Ati_{\hspace{-1pt} 0}$, we elevate $\walki_1$ to
 \begin{equation*}
\walki_0=a_{111}b_{21}a_{21}b_{11}^{-1}b_{12}a_{22}^{-1}b_{22}^{-1}  a_{1211}a_{1212}^{-1}a_{122}^{-1}a_{112}^{-1}a_{111}b_{21}a_{21}b_{11}^{-1}b_{12}a_{22}^{-1}b_{22}^{-1}a_{1211}.
 \end{equation*}
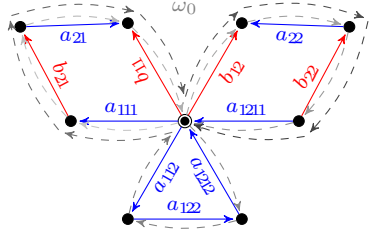
\vspace{-5pt}
\begin{figure}[H]
\centering
\begin{tikzpicture}[shorten >=1pt, node distance=1cm and 1cm, on grid,auto,auto,>=stealth',rotate=90,transform shape]       
\newcommand{\rad}{1.5}
\newcommand{\Rad}{2.5}
\node[state,accepting] (bp)  {};

\foreach \x in {0,...,5}
{
\node[state] (\x) at ($(90+\x*360/6:\rad cm)$) {};
}

\node[state] (34) at ($(180+120:\Rad cm)$) {};
\node[state] (50) at ($(180+2400:\Rad cm)$) {};
\node[inner sep=0pt] (bp0) [above left = 0.2 of bp] {};
\node[inner sep=0pt] (00) [above left = 0.2 of 0] {};
\node[inner sep=0pt] (500) [above right = 0.2 of 50] {};
\node[inner sep=0pt] (55) [below right = 0.2 of 5] {};
\node[inner sep=0pt] (bp1) [right = 0.3 of bp] {};
\node[inner sep=0pt] (44) [right = 0.2 of 4] {};
\node[inner sep=0pt] (344) [below right= 0.25 and 0.1 of 34] {};
\node[inner sep=0pt] (33) [below left = 0.1 and 0.2 of 3] {};
\node[] (g0) [right = 1.5 of bp] {\rotatebox{-90}{\color{gray}{\scriptsize{$\walki_0$}}}};

\draw[->,blue] (bp) edge node[sloped, below, pos=0.5] {$a_{1\!1\!2}$} (1);
\draw[->,blue] (1) edge node[sloped, above, pos=0.5] {$a_{1\!2\!2}$} (2);
\draw[->,blue] (2) edge node[sloped, above, pos=0.5] {\rotatebox{180}{$a_{1\!2\!1\!2}$}} (bp);

\draw[->,blue] (bp) edge node[sloped, anchor=center, below,pos=0.55] {\rotatebox{180}{$a_{1\!1\!1}$}} (0);
\draw[->,red] (0) edge node[sloped, below, pos=0.4] {\rotatebox{180}{$b_{2\!1}$}} (50);
\draw[->,blue] (50) edge[] node[sloped, anchor=center, below] {$a_{2\!1}$} (5);
\draw[->,red] (bp) edge  node[sloped, above, pos=0.6] {$b_{1\!1}$} (5);

\draw[->,red] (bp) edge node[sloped, below, pos=0.6] {$b_{1\!2}$} (4);
\draw[->,blue] (34) edge node[sloped, above, pos=0.55] {\rotatebox{180}{$a_{2\!2}$}} (4);
\draw[->,red] (3) edge node[sloped, above, pos=0.4] {$b_{2\!2}$} (34);
\draw[->,blue] (3) edge node[sloped, below, pos=0.45] {\rotatebox{180}{$a_{1\!2\!1\!1}$}} (bp);

\draw[->,darkgray!30,bend left=15,max distance= 11 mm] (bp) edge[dashed] (0);
\draw[->,darkgray!34,bend left=15,max distance= 11 mm] (0) edge[dashed] (50);
\draw[->,darkgray!38,bend left=15,max distance= 11 mm] (50) edge[dashed] (5);
\draw[->,darkgray!42,bend left=15,max distance= 11 mm] (5) edge[dashed] (bp);
\draw[->,darkgray!46,bend left=15,max distance= 11 mm] (bp) edge[dashed] (4);
\draw[->,darkgray!50,bend left=15,max distance= 11 mm] (4) edge[dashed] (34);
\draw[->,darkgray!54,bend left=15,max distance= 11 mm] (34) edge[dashed] (3);
\draw[->,darkgray!58,bend left=15,max distance= 11 mm] (3) edge[dashed] (bp);
\draw[->,darkgray!62,bend left=15,max distance= 11 mm] (bp) edge[dashed] (2);
\draw[->,darkgray!66,bend left=15,max distance= 11 mm] (2) edge[dashed] (1);
\draw[->,darkgray!70,bend left=15,max distance= 11 mm] (1) edge[dashed] (bp0);
\draw[->,darkgray!74,bend left=15,max distance= 11 mm] (bp0) edge[dashed] (00);
\draw[->,darkgray!78,bend left=15,max distance= 11 mm] (00) edge[dashed] (500);
\draw[->,darkgray!82,bend left=15,max distance= 11 mm] (500) edge[dashed] (55);
\draw[->,darkgray!86,bend left=15,max distance= 11 mm] (55) edge[dashed] (bp1);
\draw[->,darkgray!90,bend left=15,max distance= 11 mm] (bp1) edge[dashed] (44);
\draw[->,darkgray!94,bend left=15,max distance= 11 mm] (44) edge[dashed] (344);
\draw[->,darkgray!98,bend left=15,max distance= 11 mm] (344) edge[dashed] (33);
\draw[->,darkgray,bend left=15,max distance= 11 mm] (33) edge[dashed] (bp);
\end{tikzpicture}
\end{figure}
 \vspace{5pt}
Bracketing this expression according to the complete visits to each petal (or petal inverse) of $\Ati_0=\flower(S)$, we obtain the desired expression of $w_0=a$ as a word in $w_1, w_2, w_3$:
 \begin{align*}
a & \,=\, \big( abab^{-1}\big) \big( ba^{-1}b^{-1}a\big) \big( a^{-1}a^{-1}a^{-1} \big) \big( abab^{-1} \big) \big( ba^{-1}b^{-1}a\big) \\ & \,=\, w_2 w_3^{-1} w_1^{-1} w_2 w_3^{-1}.
 \end{align*}

As will be detailed below, the elevation of a $\bp$-walk through an open folding is unique, but the elevation of $\walki_5$ trough the closed folding $\Ati_{\hspace{-1pt} 4} \xtr{\,} \Ati_{\hspace{-1pt} 5}$ had several possibilities. Different choices would have ended in possibly different final results; namely, alternative expressions of $w_0=a$ as a word on $w_1^{\pm 1}, w_2^{\pm 1}, w_3^{\pm 1}$. We invite the reader to choose $\walki_4=a_{12}$ instead of $\walki_4=a_{11}$, and check that its corresponding elevation through $\Ati_{\hspace{-1pt}0}$ leads to the new expression $a=(a^{-1}bab^{-1}) (ba^{-1}b^{-1}a^{-1})(aaa) =w_3w_2^{-1}w_1$. 
\end{exm}

In the previous example, the fact that we got \emph{two different expressions} for~${a\in H}$ as words on $w_1, w_2, w_3$, namely $w_2 w_3^{-1} w_1^{-1} w_2 w_3^{-1} =\, a \,=\, w_3w_2^{-1}w_1$, confirms the fact that the generating set $\{ w_1, w_2, w_3\}$ \emph{it is not a free set}. For instance, $w_2 w_3^{-1} w_1^{-1} w_2 w_3^{-1} w_1^{-1} w_2 w_3^{-1}=1$ is a non-trivial relation between the three given generators for $H$.

This discussion relates to another natural algorithmic problem for groups: computing presentations for subgroups. Below, we analyze this problem in the context of free groups: given finitely many elements $\{w_1,\ldots ,w_n\} \subseteq \Free[A]$ 
compute a full set of relations $r_1,\ldots ,r_m$ between them; \ie a presentation of the form $\pres{w_1,\ldots ,w_n}{r_1,\ldots ,r_m}$ for the (free) subgroup $\gen{w_1,\ldots ,w_n} \leqslant \Free[A]$. 

To answer this question we need to understand the kernel of the group homomorphism $\phi^{*}\colon \pi(\Ati) \xto{} \pi(\Ati')$ determined by an elementary folding \smash{$\Ati \xtr{\!\! \phi \!}\Ati'$}. 

\begin{prop}\label{prop: unsolfolding}
Let~\smash{$\Ati \xtr{\!\phi\!} \Ati'$} be an elementary folding. Then, 
 \begin{enumerate}[ind]
\item both $\phi\colon \Ati\xto{} \Ati'$ and $\phi^{*}\colon \pi (\Ati)\onto \pi (\Ati')$ are surjective;
\item if $\phi$ is open then $\phi^{*}$ is an isomorphism;
\item if $\phi$ is closed then $\ker\phi^{*}=\ncl{\gamma}$, where $\gamma$ is any elementary walk associated to $\phi$. 
 \end{enumerate}
\end{prop}

\begin{proof}
Let us fix notation: denote by $\edgii_1,\, \edgii_2\in \Edgs\Ati$ the arcs folded by $\phi$, $\vertiii =\edgii_1\init =\edgii_2\init$, $\vertii_1 =\edgii_1 \term$, and $\vertii_2 =\edgii_2 \term$; and denote by $\edgii \in \Edgs \Ati'$ and $\vertii \in \Verts \Ati'$, respectively, the folded arc and vertex. Formally, $\Ati'$ equals $\Ati$ with $\edgii_1$ and $\edgii_2$ identified (into $\edgii$); alternatively, in the case of a closed folding, we can think $\Ati'$ as the result of removing $\edgii_2$ from $\Ati$ (and renaming $\edgii_1$ into $\edgii$).

The homomorphism $\phi\colon \Ati \xto{} \Ati'$ is bijective except for the vertices $\vertii_1,\, \vertii_2\in \Verts\Ati$ both mapped to the same vertex $\vertii\in \Verts\Ati'$, and for the arcs $\edgii_1,\, \edgii_2\in \Edgs\Ati$ both mapped to the same arc $\edgii \in \Edgs\Ati'$. Hence, the surjectivity of $\phi$ is clear. Moreover, any reduced $\bp'$-walk $\walki$ in $\Ati'$ elevates to a reduced $\bp$-walk in $\Ati$ as explained above. Therefore, $\phi^{*}$ is surjective as well. This shows assertion~(i). 

To see (ii), assume that the folding $\phi$ is open and let $\walki\in \pi (\Ati)$ be a reduced \mbox{$\bp$-walk}. In general, $\walki \phi$ is a reduced $\bp'$-walk except for the following possible backtracking situations: for any visit of $\walki$ to the folding, \ie one of the form $\edgi_1 \vertii_1\edgii_1^{-1}\vertiii \edgii_2\vertii_2\edgi_2$ or $\edgi_1 \vertii_2\edgii_2^{-1} \vertiii\edgii_1 \vertii_1 \edgi_2$, the image walk $\walki\phi$ presents a backtracking over the folded arc, $\edgi_1 \vertii\edgii^{-1} \vertiii \edgii \vertii \edgi_2$. We claim that by just removing from $\walki \phi$ these backtrackings $\edgii^{-1}\edgii$ and \emph{nothing else}, we already get the reduced path $\overline{\walki\phi}=\walki \phi^{*}$: this is because, in $\Ati$, $\edgi_1 \term =\vertii_1\neq \vertii_2 =\edgi_2 \init$ (resp., $\edgi_1 \term =\vertii_2\neq \vertii_1 =\edgi_2 \init$) and so, $\edgi_1\neq \edgi_2^{-1}$. Therefore, after canceling the backtracking $\edgii^{-1} \edgii$ from $\walki\phi$, the two new edges in contact $\edgi_1\edgi_2$ do not present backtracking anymore. With this claim in mind, the injectivity of $\phi^{*}$ follows easily: indeed, if $\walki\in \pi(\Ati)$ is non-trivial then it is either $\walki=\edgii_i^n$, for $i=1,2$ and $0\neq n\in \mathbb{Z}$ (assuming $\edgii_i$ is a loop), or it must contain an edge different from $\edgii_1^{\pm 1}, \edgii_2^{\pm 1}$; in the first case $\edgii$ is also a loop and $\walki\phi^{*}=\edgii^n \neq 1$ in $\pi(\Ati')$; in the second case, $\overline{\walki\phi}$ contains an edge different from $\edgii^{\pm 1}$ and so, $\overline{\walki \phi}\neq 1$ in $\pi(\Ati')$ as well. This shows~(ii). 

Finally, to see~(iii), assume that $\vertii_1=\vertii_2$, choose an arbitrary elementary nontrivial walk $\gamma =\gamma_0 \vertiii \edgii_2 \vertii_1 \edgii_1^{-1}\vertiii \gamma_0^{-1}$ associated to $\phi$, and let us see that $\ker\phi^{*} =\ncl{\gamma} \normaleq \pi (\Ati)$. Since $\gamma\phi^{*} =\red{(\gamma_0\phi) \vertiii \edgii \vertii \edgii^{-1} \vertiii (\gamma_0\phi)^{-1}}=1$, the inclusion to the left is clear. Conversely, fix an arbitrary $\walki\in \ker\phi^{*}$, and we will prove that $\walki\in \ncl{\gamma}$ by induction on the number of occurrences (say $s\geq 0$) of $\edgii_2^{\pm 1}$ in $\walki$. If $s=0$ then $\walki$ is, in fact, a $\bp$-walk in $\Ati\setmin \{\edgii_2\}=\Ati'$ and so $\walki =\walki\phi=1\in \ncl{\gamma}$. Now assume that $s\geq 1$, and the result true for all walks with $s-1$ occurrences of $\edgii_2^{\pm 1}$ in them. Highlighting anyone of such occurrences, $\walki$ can be written as $\walki =\alpha \edgii_2^{\varepsilon} \beta$, for some $\varepsilon=\pm 1$, and some reduced paths $\alpha, \beta$. Rewriting $\walki$ (and, for the sake of clarity, omitting the vertices), we get 
 $$
\walki =\left\{ \begin{array}{ll} \alpha \edgii_2 \beta =(\alpha \gamma_0^{-1})(\gamma_0 \edgii_2 \edgii_1^{-1} \gamma_0^{-1})(\gamma_0 \alpha^{-1})(\alpha \edgii_1 \beta), & \text{ if } \varepsilon=1, \\  \alpha \edgii_2^{-1} \beta =(\alpha \edgii_1^{-1}\beta) (\beta^{-1}\gamma_0^{-1})(\gamma_0 \edgii_1 \edgii_2^{-1}\gamma_0^{-1}) (\gamma_0 \beta), & \text{ if } \varepsilon=-1. \end{array} \right.
 $$
Hence, $\walki$ equals (in $\pi(\Ati)$) the product of a conjugate of $\gamma^{\pm 1}$, times $\alpha \edgii_1^{\pm 1}\beta$ which belongs to $\ncl{\gamma}$ by induction. Therefore $\walki\in \ncl{\gamma}$, and the proof is complete.
\end{proof} 

From \Cref{prop: unsolfolding} we will infer that, in any folding sequence $\Ati \xtr{\!\! \Phi \!}\Ati'$, the open elementary foldings do not contribute to $\ker\Phi$, whereas each closed folding does contribute with a new essential generator for $\ker\Phi$ as normal subgroup of $\pi(\Ati)$. In this sense, the following theorem quantifies the non-injectivity of the group homomorphism $\Phi^{*} \colon \pi(\Ati)\xto{} \pi(\Ati')$, both from the algebraic and from the geometric points of view. This is a reformulation of a recent result by Rosenmann and Ventura in \cite{rosenmann_dependence_2024} (see also \cite{ascari_ideals_2022}).

\begin{defn}
The \defin{loss of a folding sequence} $\Phi =(\phi_i)_{i\geq 1}$, denoted by $\loss(\Phi)$, is the (possibly infinite) number of closed foldings in $\Phi$.
\end{defn}





\begin{thm}[Rosenmann--Ventura, \cite{rosenmann_dependence_2024}]\label{thm: loss seq}
Let $\Ati$ be a finite involutive pointed \mbox{$A$-automaton}, and let $\Phi =\phi_1\phi_2\cdots \phi_p$ be a folding sequence on $\Ati$,
 \begin{equation} \label{eq: loss seq}
\Ati=\Ati_{\!0} \xtr{\!\!\! \phi_1 \!\!}\Ati_{\!1} \xtr{\!\!\! \phi_2 \!\!} \cdots \xtr{\!\!\! \phi_p \!\!} \Ati_{\!p}=\Ati'. 
 \end{equation}
For each closed folding $\phi_{j}$ in $\Phi$, $j\in[1,p]$, let $\widetilde{\phi}_j \in \pi(\Ati_{\!0})$ be an elementary elevation of~$\phi_{j}$ through $\phi_1\cdots \phi_{j-1}$. Then, the set\, $\Xi =\set{\widetilde{\phi}_{j} \st \phi_{j}  \text{ is a closed folding in } \Phi}$ is a computable minimum set of generators for $\ker(\Phi^{*})$ as a normal subgroup. That is, $\ker(\Phi^{*}) =\ncl{\Xi}$ and $\loss(\Phi) \,=\, \nrk(\ker(\Phi^{*})) \,=\, \rk(\Ati)-\rk(\Ati')$.
\end{thm}

\begin{proof}
Let $\rk(\Ati)=n$ and $\rk(\Ati')=m\leq n$. With an easy algebraic argument, let us first show that $\ker(\Phi^{*})$ cannot be generated by less than $n-m$ elements, even as normal subgroup of $\pi (\Ati)$. Suppose that $\walki_1,\ldots ,\walki_r\in \pi(\Ati)$ normally generate $\ker\Phi^{*} =\ncl{\walki_1,\ldots ,\walki_r} \normaleq \pi (\Ati)$, and let us see that $r\geq n-m$. Since $\Phi^{*}\colon \pi(\Ati)\onto \pi(\Ati')$ is onto, we have that $\pi(\Ati')\isom \pi(\Ati)/\ker\Phi^{*}$. Now abelianizing (and since $\pi(\Ati)$ is a free group of rank $n$, and $\pi(\Ati')$ is a free group of rank $m$), we get $\mathbb{Z}^m \isom \mathbb{Z}^n/\langle \walki_1^{\rm ab}, \ldots ,\walki_r^{\rm ab}\rangle$, which implies $m\geq n-r$.

Using \Cref{rem: rank folding}, it is clear that $\rk(\Ati)-\rk(\Ati')$ equals the number of closed foldings in~\eqref{eq: loss seq}, \ie $\rk(\Ati)-\rk(\Ati')=\loss(\Phi)$. And in \Cref{prop: unsolfolding} we understood the kernel of each elementary folding $\Ati_{i-1} \xtr{\!\!\! \phi_i \!\!}\Ati_{\hspace{-1pt}i}$ in~\eqref{eq: loss seq}, $i\in [1,p]$: if $\phi_i$ is open then $\ker\phi_{i}^{*}$ is trivial; and if $\phi_{j}$ is closed then $\ker\phi_{j}^{*}= \ncl{\gamma_j} \normaleq \pi(\Ati_{j-1})$, where $\gamma_j$ denotes an elementary walk associated to $\phi_j$ which elevates to $\widetilde{\phi}_j$ through $\Phi_{j-1}=\phi_1 \cdots \phi_{j-1}$. 
For notational simplicity, let $\widetilde{\phi}_i=1$ for the indices $i\in [1,p]$ for which $\phi_i$ is open. 

We have to prove that $\ker\Phi^{*} =\ncl{\widetilde{\phi}_1,\ldots ,\widetilde{\phi}_p}\normaleq \pi(\Ati)$. The inclusion to the left is clear from construction. For the other inclusion, let $\walki\in \ker \Phi^{*}$ be an arbitrary reduced $\bp$-walk in $\pi(\Ati)$ with $\red{\walki\Phi^{*}} =\red{ \walki\phi_{1}^{*}\cdots \phi_p^*}=1$; we see that $\walki\in \ncl{ \widetilde{\phi}_1,\ldots ,\widetilde{\phi}_p}$ by induction on the first index $i=1,\ldots ,p$ such that $\walki\in \ker \Phi_i^*$. Indeed, if $i=1$ then $\walki\in \ker\Phi_{1}^{*} =\ker\phi_{1}^{*}=\ncl{\gamma_1}=\ncl{ \widetilde{\phi}_1} \normaleq \ncl{\widetilde{\phi}_1,\ldots ,\widetilde{\phi}_p}$. And, for arbitrary $2\leq i\leq p$, we have 
 \begin{align*}
\walki\Phi_{i-1}^*\in \ker \phi_i^* & \,=\,\ncl{\gamma_i}_{\pi(\Ati_{i-1})} \,=\, \ncl{ \widetilde{\phi}_i \Phi_{i-1}^*}_{(\pi(\Ati))\Phi_{i-1}^*} \\ & \,=\, \big( \ncl{\widetilde{\phi}_i}_{\pi(\Ati)} \big) \Phi_{i-1}^*.
 \end{align*}
This means that, for some $\xi\in\ncl{\widetilde{\phi}_i}_{\pi(\Ati)}$, $(\walki \xi^{-1})\Phi_{i-1}^*=1$ and, by induction, $\walki \xi^{-1}\in \ncl{\widetilde{\phi}_1,\ldots ,\widetilde{\phi}_p}$; therefore, $\walki\in \ncl{\widetilde{\phi}_1,\ldots ,\widetilde{\phi}_p}$, completing the proof.
\end{proof}

\begin{cor} \label{cor: comp pres}
There is an algorithm that, given finitely many elements $w_1,\ldots ,w_n$ in~$\Free[A]$, 
computes a complete set of relations $r_1,\ldots ,r_m$ between them; that is, a presentation of the form $\pres{w_1,\ldots ,w_n}{r_1,\ldots ,r_m}$ for the (free) subgroup~$\gen{w_1,\ldots ,w_n} \leqslant~\Free[A]$. \qed
\end{cor}

It is clear that the loss of all maximal folding sequences on a given finite automaton coincide.

\begin{defn}
The \defin{loss of an automaton} $\Ati$, denoted by $\loss(\Ati)$, is 
the supremum of the losses of all the folding sequences on $\Ati$.
Note that if $\Ati$ is finite then there exists an obvious maximal sequence of foldings from $\Ati$ (to a deterministic automaton whose core is $\red{\Ati}$). Hence, $\loss(\Ati) = \rk(\Ati) - \rk(\red{\Ati})$.\footnote{It is possible to define the direct limit of an infinite folding sequence, and the notion of maximality (which corresponds to the limit being deterministic).}
\end{defn}


Combining \Cref{thm: loss seq} with \Cref{lem: deterministic gamma} we reach the promised description of the kernel of $\red{\lab}\colon \pi(\Ati) \onto \gen{\Ati}$.

\begin{cor} \label{cor: ker rlab}
Let $\Ati$ be a finite automaton, and let $\rlab\colon \pi(\Ati) \xto{} \gen{\Ati}$ be the group homomorphism $\walki \mapto (\walki)\rlab$. With the above notation, 
 \begin{equation}
\ker(\red{\lab}) \,=\, \ncl{\widetilde{\phi}_j \st \phi_j \text{ is a closed folding in } \Phi},
 \end{equation}
where $\Phi=\phi_1\cdots \phi_p$ is any maximal folding sequence $\Ati \xto{} \red{\Ati}$. \qed
\end{cor}

\begin{cor} \label{cor: loss=0}
    Let $S \subseteq \Free[A]$ and let $H=\gen{S}\leqslant \Free[A]$. Then, the following statements are equivalent:
    \begin{enumerate}[dep]
        \item $S$ is freely independent in $\Free[A]$ (\ie $S$ is a basis of $H$);
        \item $\loss(\flower(S))=0$;
        \item $\rk(\flower(S)) = \rk(\red{\flower(S)})$.
    \end{enumerate}
\end{cor}

\begin{proof}
    If $\loss(\flower(S))>0$ then there exists a (finite) folding sequence on $\flower(S)$ containing a closed folding $\phi_j$. The elevation $\widetilde{\phi}_j$ provides a nontrivial relation between the elements in $S$ (see~\Cref{thm: loss seq}). This proves that (a)\,$\Rightarrow$\,(b).

Any non-trivial relation between the elements in $S$ will involve only finitely many of them, say $S'\subseteq S$.
By \Cref{prop: unsolfolding}.(ii), any maximal folding sequence 
on $\flower(S')$ must contain a closed folding. Considering the same folding sequence on $\flower(S)$ we obtain that $\loss(\flower(S))>0$. This proves that (b)\,$\Rightarrow$\,(a).

If $S$ is finite, the equivalence (b)\,$\Leftrightarrow$\,(c)
follows immediately from \Cref{thm: loss seq}. The same claim is true for any infinite subset $S \subseteq \Free[A]$ but its proof requires a technical construction of the limit object of an infinite sequence of foldings, which we omit, and will not be used in this paper. 
\end{proof}



As an immediate consequence, we deduce several well-known fundamental facts about free groups in a very transparent way. Recall that a group $G$ is said to be \defin[Hopfian group]{Hopfian} if any surjective endomorphism $\varphi\colon G\xto{} G$ is also injective. 

\begin{thm}
Finitely generated free groups are Hopfian.
\end{thm}

\begin{proof}
Let $A=\{a_1,\ldots ,a_n\}$ and let $\varphi \colon \Free[A]\xto{} \Free[A]$ be a surjective endomorphism. Then, $\gen{a_1\varphi, \ldots ,a_n\varphi}=\Free[A]$ and so, $\loss(\flower(\{a_1\varphi, \ldots ,a_n\varphi\}))=n-n=0$. By \Cref{cor: loss=0}, $\{a_1\varphi, \ldots ,a_n\varphi\}$ is a basis for $\Free[A]$ and $\varphi$ is an automorphism. 
\end{proof}

\begin{cor}\label{cor: mateix cardinal}
Any two bases of a free group have the same cardinal.
\end{cor}

\begin{proof}
Let $B$ be a basis of a free group $\Free[A]$. From \Cref{cor: loss=0} (since $B$ is freely independent) $\card B=\rank (\flower(B)) = \rank (\red{\flower(B)}) = \rk(\stallings(\Free[A])) =  \rank(\Free[A])= |A|$.

Note that the previous argument uses the general implication (a)\,$\Rightarrow$\,(c) from \Cref{cor: loss=0}, unproved when $S$ is infinite. An alternative proof for the infinite case follows. When the basis $B\subseteq \Free[A]$ is infinite, $\card B = \card \Free[B] =\card \Free[A]$. On the other hand, if $A$ were finite, writing each $a\in A$ in terms of $B$ would 
provide a finite subset of $B$ generating $\Free[A]$, in contradiction with $B$ being a basis of $\Free[A]$; therefore, $A$ is also infinite and so $\card A=\card \Free[A]$, finishing the proof. 
\end{proof}

\begin{cor}\label{cor: mateix cardinal}
Two free groups are isomorphic if and only if they have the same rank. \qed
\end{cor}


\medskip

To finish this section we will mention a recent development done by Rosenmann and Ventura in \cite{rosenmann_dependence_2024} (see also \cite{ascari_ideals_2022}). The main idea is already expressed in \Cref{thm: loss seq} above, although in these papers the authors work in a particular case of their interest. 

\begin{defn}
Let $G$ be a group. An \defin[equation over a group]{equation over $G$} (a \defin{$G$-equation} for short) is an element from the free product $G*\mathbb{Z}=G*\gen{X}$, where the generator of the infinite cyclic group is denoted by $X$, and considered as an unknown. Writing elements from this free product in normal form, a typical equation looks like $w(X)=g_0X^{\varepsilon_1}g_1 \cdots g_{n-1}X^{\varepsilon_n}g_n$, where $g_i\in G$, $\varepsilon_i=\pm 1$, where $g_i=1$ implies $\varepsilon_{i-1}=\varepsilon_i$, and where $n=\sum_i |\varepsilon_i|\geq 0$ is called the \defin[degree of an equation]{degree} of $w(X)$. We say that $w(X)$ is \defin[balanced equation]{balanced} if $\sum_i \varepsilon_i=0$. Whenever all the $g_i$ belong to a certain subgroup $H\leqslant G$, we say that $w(X)$ \defin[equation with coefficients in a subgroup]{has coefficients in $H$}. An element $g\in G$ is a \defin[solution of an equation]{solution} for $w(X)$ whenever $w(g)=1$ in~$G$. 

Dually, an element $g\in G$ is said to be \defin[dependent (element) on a subgroup]{dependent on $H$} (\emph{$H$-dependent}) if it is the solution to some non-trivial $H$-equation; otherwise it is said to be \defin[independent (element) on a subgroup]{$H$-independent}. Finally, given $g\in H$, the \defin{ideal of $H$-equations for $g$} is $\mathcal{I}_H(g)=\{w(X)\in H*\gen{X} \mid w(g)=1\}$, a normal subgroup of $H*\gen{X}$, in fact, the kernel of the \defin{evaluation map}, $\varphi_g\colon H*\gen{X}\onto \gen{H,g}\leqslant G$, $h\mapsto h$ ($h\in H$), $X\mapsto g$.  
\end{defn}

Note the obvious parallelism with the classical situation in field theory. These definitions are clearly inspired in the classical notions of (one variable) polynomials with coefficients on a subfield, roots of polynomials, annihilator of an element, and algebraic and transcendental elements (over a subfield). 

The so-called Diophantine problem for groups (given an equation over a group~$G$, decide whether it has some solution in $G$) has been studied for several families of groups. For free groups, this is a deep problem solved some decades ago by Makanin~\cite{makanin_equations_1982} (see also  \cite{razborov_systems_1984}). The dual problem was recently solved by Rosenmann and Ventura in \cite{rosenmann_dependence_2024} using the above mentioned idea of elevation of $\bp$-walks up through folding sequences. 

\begin{thm}[Rosenmann--Ventura, \cite{rosenmann_dependence_2024}]\label{thm: RV}
Let $\Free[A]$ be a free group, let $H\leqfg \Free[A]$ be a finitely generated subgroup, and let $g\in \Free[A]$. Then, one can decide whether $g$ is $H$-dependent,
and compute finitely many equations $w_1(X),\ldots ,w_p(X)$ generating $\mathcal{I}_H(g)$ as normal subgroup of $H * \gen{X}$. 
\end{thm}

\begin{proof}
Compute $\stallings(H,A)$ and a basis $\{h_1,\ldots ,h_n\}$ for $H$. Let $\Ati$ be this Stallings' $A$-automaton with a petal spelling (the reduction of) $g$ attached to $\bp$. Note that $\Ati$ is reduced except for the possible violations of determinism around the basepoint. Compute now a reduction process for $\Ati$, 
 \begin{equation*} 
\Ati=\Ati_{\!0} \xtr{\!\!\! \phi_1 \!\!}\Ati_{\!1} \xtr{\!\!\! \phi_2 \!\!} \cdots \xtr{\!\!\! \phi_p \!\!} \Ati_{\!p}=\stallings(\gen{H,g},A). 
 \end{equation*}
We claim that $g$ is dependent on $H$ if and only if $\loss(\Ati)\geq 1$. Indeed if $\loss(\Ati)=0$ then, by \Cref{cor: hopfian grafic}, $\{g,h_1, \ldots ,h_n\}$ is a basis for $\gen{H,g}$ so, $\gen{H,g}=H*\gen{g}$ and we conclude that $g$ is $H$-independent. Conversely, if $\loss(\Ati)\geq 1$, then $\{g,h_1,\ldots ,h_n\}$ is a \emph{non-free} set of generators for $\gen{H,g}$ and so, there must be nontrivial relations between them, which will mandatorily involve $g$ since $\{h_1,\ldots ,h_n\}$ are freely independent. Such relations are nontrivial $H$-equations $w(X)$ such that $w(g)=1$. This proves the decision part. 

Moreover, by \Cref{cor: comp pres}, computing an elevation to $\Ati$ of each closed folding in the reduction process above, we get a presentation for the (free) group $\gen{H,g}$ of the form $\pres{g,h_1,\ldots ,h_n}{r_1,\ldots ,r_q}$, namely, a collection of $H$-equations $w_1(X),\ldots ,w_q(X)$, having $g$ as a solution (\ie $w_i(g)=r_i=1$), and generating the kernel of the evaluation map $\varphi\colon H*\gen{X}\onto \gen{H,g}$ as normal subgroup. 
\end{proof}

\begin{rem}
We note that \Cref{thm: RV} also follows indirectly from a more general result in \cite{kapovich_foldings_2005}. 
\end{rem}

\section{Conjugacy, cosets, and index}\label{sec: conj & cosets}

As already stated in \Cref{lem: recognized}, a change of basepoint in a Stallings' automaton has a clear algebraic meaning; it corresponds to conjugating the recognized subgroup. The general assertion merely requires allowing to move the basepoint outside the core.




\begin{named}[Subgroup Conjugacy Problem for $G=\pres{A}{R}$, $\SCP(G)$]
Decide, given finite subsets $R,S \subseteq (A^{\pm})^*$, whether the subgroups $\gen{S}$ and $\gen{R}$ are conjugate in $G$ and, if so, compute a conjugator.
\end{named}

\begin{thm} \label{thm: SCP Fn}
Let $H \leqslant \Free[A]$ and let $w \in \Free[A]$. Then, $\schreier(H^w)= [\schreier(H)]_{\scriptscriptstyle{H\!w}}$, $\stallings(H^{w}) =\core_{H\!w}(\schreier(H))$, and $\stallings^*(H^{w}) =\stallings^*(H)$. Hence, for any $H,K \leqslant \Free[A]$, the following statements are equivalent:
\begin{enumerate}[dep]
\item the subgroups $H$ and $K$ are conjugate in $\Free[A]$;
\item $\schreier(H)$ equals $\schreier(K)$ modulo a change in the basepoint;
\item $\stallings^*(H)=\stallings^*(K)$.
\end{enumerate}
\end{thm}

\begin{proof}
It is enough to realize that, for any $w \in \Free[A]$, $[\schreier(H)]_{\scriptscriptstyle{H\!w}}$ is a deterministic and saturated automaton recognizing $H^w$ (and further take the core and strict core of $\schreier(H^w)=[\schreier(H)]_{\scriptscriptstyle{H\!w}}$). The equivalence follows immediately.
\end{proof}




\begin{cor}
The Subgroup Conjugacy Problem is solvable for $\Free[A]$.  
\end{cor}

\begin{proof}
If both $H$ and $K$ are given by finite families of generators then their strict Stallings' automata are finite and computable and condition (c) in \Cref{thm: SCP Fn} is  algorithmically checkable. Moreover, if $\stallings^*(H) =\stallings^*(K)$  the label $w$ of any path $\bpH \xrwalk{\ } \bpK$ in $\schreier(H)$ is a conjugator of $K$ into $H$; \ie $K = H^w$.
\end{proof}

A graphical characterization of normality within free groups also follows from the discussion above.


\begin{prop} \label{prop: normal iff}
Let $H$ be a nontrivial subgroup of $\Free[A]$. Then, $H$ is normal in $\Free[A]$ if and only if $\stallings(H)$ is vertex-transitive and saturated.\footnote{An $A$-automaton $\Ati$ is said to be \defin[vertex-transitive automaton]{vertex-transitive} if for any pair of vertices $\verti,\vertii$ in $\Ati$, there exists an automorphism of $A$-digraphs sending $\verti$ to $\vertii$.} 
\end{prop}

\begin{proof}
The claim is obvious through the following sequence of equivalences:
 \begin{align*}
H \text{ is normal in } \Free[A] & \,\Leftrightarrow\, \schreier(H) \text{ is vertex-transitive} \\ & \,\Leftrightarrow\, \schreier(H) \text{ is vertex-transitive and core} \\ & \,\Leftrightarrow\, \stallings(H) \text{ is vertex-transitive and saturated.}
 \end{align*}
All the implications are immediate from the definitions and previous results except for the necessary coreness of a vertex-transitive Schreier automaton of a nontrivial subgroup. To see this, note that no automorphism of $A$-digraphs can send a vertex in the \emph{nontrivial} core of $\bm{\Sigma}=\schreier(H)$ (belonging to some reduced $\bp$-walk in $\bm{\Sigma}$) to a vertex outside the core (not belonging to any reduced $\bp$-walk in $\bm{\Sigma}$). This completes the proof.
\end{proof}

\begin{rem}
Although in order to decide whether a finitely generated subgroup $H \leqslant \Free[A]$ is normal in $\Free[A]$ it is enough to check whether the conjugates of the generators by the elements in $A^{\pm}$ belong to $H$ (\ie finitely many instances of $\MP(\Fn)$), the proposition above provides a more synthetic characterization also suitable for algorithmically checking normality when the input subgroup is finitely generated.
\end{rem}

The connection between Stallings' theory and cosets and index of subgroups of the free group is encapsulated by \Cref{prop: St = core Sch}: $\stallings(H)$ is precisely the core of $\schreier(H)$ (having as vertices the right cosets of $H$ in $\Free[A]$). In this section, we leverage this close relationship to explore several problems and  derive some classical results in this regard.

\begin{named}[Finite index problem a $G=\pres{A}{R}$, $\FIP(G)$]
Decide, given a finite subset $S \subseteq (A^{\pm})^*$, whether the subgroup $\gen{S}$ has finite index in $G$ and, if so, compute a family of (right) coset representatives.
\end{named}

\begin{prop} \label{prop: fi}
Let $H$ be a subgroup of a free group~$\Free[A]$. Then, the index $\ind{H}{\Free[A]}$ is finite if and only if $\stallings(H)$ is saturated and $\card\Verts(\stallings(H))<\infty$; in this case, $\ind{H}{\Free[A]} = \card\Verts(\stallings(H))$.
\end{prop}

\begin{proof}
For the implication to the right, note that:
 \begin{enumerate*}[ind]
\item since $\stallings(H) \subseteq \schreier(H)$, then $\card\Verts(\stallings(H)) \leq \card\Verts(\schreier(H)) = \ind{H}{\Free[A]}$, and 
\item if $\stallings(H)$ contains an $a$-deficient vertex $\verti$, then $\schreier(H)$ contains an (infinite) $a$-Cayley branch attached to $\verti$, and hence $\ind{H}{\Free[A]} = \card\Verts(\schreier(H,A))$ is infinite.
 \end{enumerate*}
    
Conversely, by \Cref{lem: St vs Sch}.\ref{item: St sat iff Sch cor}, if $\stallings(H)$ is saturated, then $\stallings(H) = \schreier(H)$. Hence, in this case $\ind{H}{\Free[A]}=\card\Verts(\schreier(H))= \card\Verts(\stallings(H))$, and the claimed result follows.
\end{proof}

As a consequence we obtain two well-known results (\Cref{cor: fi subgroups} and \Cref{cor: Schreier lemma}) on finitely generated free groups which easily extend to any finitely generated group.

\begin{cor} \label{cor: fi subgroups}
Every finitely generated group $G$ contains finitely many subgroups of any given finite index.
\end{cor}

\begin{proof}
Since every finitely generated group $G$ admits an epimorphism $\Fn \onto G$ for some $n\in \NN$, it is enough to prove the result for finitely generated free groups~$\Fn$. By \Cref{prop: fi}, for any $k\in \NN$, the bijection \eqref{eq: Stallings bijection2} in \Cref{thm: Stallings bijection2} restricts to a bijection between subgroups $H \leqslant \Fn$ of index $k$ and saturated (hence $n$-regular) automata with $k$ vertices; since the number of such automata is clearly finite, the claimed result follows. 
\end{proof}

A recursive formula computing the number of subgroups of $\Fn$ of a given finite index $k$ was obtained in \cite{hall_jr_subgroups_1949} essentially counting the corresponding graphical configurations.

\begin{thm}[\citenr{hall_jr_subgroups_1949}] \label{formula_number_fi_freegroup}
The number $N_{k}(\mathbb{F}_n)$ of subgroups of index $k$ in $\mathbb{F}_n$ is given recursively by $N_{1}(\mathbb{F}_n)=1$, and
 \begin{equation*} \label{eq: index k}
N_{k}(\mathbb{F}_n) \,=\, k(k!)^{n-1}-\sum_{i=1}^{k-1}[(k-i)!]^{n-1}N_{i}(\mathbb{F}_n). \tag*{\qed}
 \end{equation*}
\end{thm}
 \vspace{5pt}

\begin{figure}[h]
    \centering
\begin{tikzpicture}[shorten >=1pt, node distance=1.2cm and 1.75cm, on grid,auto,auto,>=stealth']


\begin{scope}[shift={(1,0)}]
    \node[state,accepting] (1) {};
    \node[state] (2) at (1,0) {};
    \node[state] (3) at (.5,{-sin(60)}) {};
  
   \path[->] (1) edge[blue,bend left=15] (2);
   \path[->] (2) edge[blue,bend left=15] (3);
   \path[->] (3) edge[blue,bend left=15] (1);
   \path[->] (1) edge[red,loop above,min distance=6mm,in=110,out=160] (1);
   \path[->] (2) edge[red,loop above,min distance=6mm,in=20,out=70] (2);
   \path[->] (3) edge[red,loop below,min distance=6mm,in=245,out=295] (3);

   \draw [ultra thick, rounded corners, draw=black, opacity=0.1]
    (-0.5,0.5) -- (-0.5,-1.5) -- (1.5,-1.5) -- (1.5,0.5) -- cycle;
\end{scope}
   
  \begin{scope}[shift={(3.2,0)}]
   \node[state,accepting] (1) {};
   \node[state] (2) at (1,0) {};
   \node[state] (3) at (.5,{-sin(60)}) {};

   \path[->] (1) edge[red,bend left=15] (2);
   \path[->] (2) edge[red,bend left=15] (3);
   \path[->] (3) edge[red,bend left=15] (1);
   \path[->] (1) edge[blue,loop above,min distance=6mm,in=110,out=160] (1);
   \path[->] (2) edge[blue,loop above,min distance=6mm,in=20,out=70] (2);
   \path[->] (3) edge[blue,loop below,min distance=6mm,in=245,out=295] (3);

   \draw [ultra thick, rounded corners, draw=black, opacity=0.1]
     (-0.5,0.5) -- (-0.5,-1.5) -- (1.5,-1.5) -- (1.5,0.5) -- cycle;

  \end{scope}
  
  \begin{scope}[shift={(5.4,0)}]
   \node[state,accepting] (1) {};
    \node[state] (2) at (1,0) {};
    \node[state] (3) at (.5,{-sin(60)}) {};
  
   \path[->] (1) edge[red,bend right=15] (2);
   \path[->] (2) edge[red,bend right=15] (3);
   \path[->] (3) edge[red,bend right=15] (1);
   \path[->] (1) edge[blue,bend left=15] (2);
   \path[->] (2) edge[blue,bend left=15] (3);
   \path[->] (3) edge[blue,bend left=15] (1);

   \draw [ultra thick, rounded corners, draw=black, opacity=0.1]
    (-0.5,0.5) -- (-0.5,-1.5) -- (1.5,-1.5) -- (1.5,0.5) -- cycle;

  \end{scope}
  
  \begin{scope}[shift={(7.6,0)}]
   \node[state,accepting] (1) {};
    \node[state] (2) at (1,0) {};
    \node[state] (3) at (.5,{-sin(60)}) {};
  
   \path[->] (2) edge[red,bend left=15] (1);
   \path[->] (3) edge[red,bend left=15] (2);
   \path[->] (1) edge[red,bend left=15] (3);
   \path[->] (1) edge[blue,bend left=15] (2);
   \path[->] (2) edge[blue,bend left=15] (3);
   \path[->] (3) edge[blue,bend left=15] (1);

   \draw [ultra thick, rounded corners, draw=black, opacity=0.1]
    (-0.5,0.5) -- (-0.5,-1.5) -- (1.5,-1.5) -- (1.5,0.5) -- cycle;

  \end{scope}
  
  \begin{scope}[shift={(2.3,-2.5)}]
    \node[state,accepting] (1) {};
    \node[state] (2) at (1,0) {};
    \node[state] (3) at (.5,{-sin(60)}) {};
  
   \path[->] (1) edge[blue,bend left=15] (2);
   \path[->] (2) edge[blue,bend left=15] (1);
   \path[->] (2) edge[red,bend left=15] (3);
   \path[->] (3) edge[red,bend left=15] (2);
   \path[->] (1) edge[red,loop above,min distance=6mm,in=110,out=160] (1);
   \path[->] (3) edge[blue,loop below,min distance=6mm,in=245,out=295] (3);

   \draw [ultra thick, rounded corners, draw=black, opacity=0.1]
       (-0.5,0.5) -- (-0.5,-1.5) -- (5.25,-1.5) -- (5.25,0.5) -- cycle;
  \end{scope}
  
  \begin{scope}[shift={(4.3,-2.5)}]
    \node[state] (1) {};
    \node[state,accepting] (2) at (1,0) {};
    \node[state] (3) at (.5,{-sin(60)}) {};
  
   \path[->] (1) edge[blue,bend left=15] (2);
   \path[->] (2) edge[blue,bend left=15] (1);
   \path[->] (2) edge[red,bend left=15] (3);
   \path[->] (3) edge[red,bend left=15] (2);
   \path[->] (1) edge[red,loop above,min distance=6mm,in=110,out=160] (1);
   \path[->] (3) edge[blue,loop below,min distance=6mm,in=245,out=295] (3);
  \end{scope}
  
  \begin{scope}[shift={(6.3,-2.5)}]
    \node[state] (1) {};
    \node[state] (2) at (1,0) {};
    \node[state,accepting] (3) at (.5,{-sin(60)}) {};
  
   \path[->] (1) edge[blue,bend left=15] (2);
   \path[->] (2) edge[blue,bend left=15] (1);
   \path[->] (2) edge[red,bend left=15] (3);
   \path[->] (3) edge[red,bend left=15] (2);
   \path[->] (1) edge[red,loop above,min distance=6mm,in=110,out=160] (1);
   \path[->] (3) edge[blue,loop below,min distance=6mm,in=245,out=295] (3);
  \end{scope}
  
  \begin{scope}[shift={(0,-5)}]
    \node[state,accepting] (1) {};
    \node[state] (2) at (1,0) {};
    \node[state] (3) at (.5,{-sin(60)}) {};
  
   \path[->] (1) edge[blue,bend left=15] (2);
   \path[->] (2) edge[blue,bend right=25] (3);
   \path[->] (2) edge[red] (3);
   \path[->] (3) edge[red,bend right=25] (2);
   \path[->] (3) edge[blue,bend left=15] (1);
   \path[->] (1) edge[red,loop above,min distance=6mm,in=110,out=160] (1);

   \draw [ultra thick, rounded corners, draw=black, opacity=0.1]
       (-0.5,0.5) -- (-0.5,-1.25) -- (4.75,-1.25) -- (4.75,0.5) -- cycle;
  \end{scope}
  
  \begin{scope}[shift={(1.75,-5)}]
    \node[state] (1) {};
    \node[state,accepting] (2) at (1,0) {};
    \node[state] (3) at (.5,{-sin(60)}) {};
  
   \path[->] (1) edge[blue,bend left=15] (2);
   \path[->] (2) edge[blue,bend right=25] (3);
   \path[->] (2) edge[red] (3);
   \path[->] (3) edge[red,bend right=25] (2);
   \path[->] (3) edge[blue,bend left=15] (1);
   \path[->] (1) edge[red,loop above,min distance=6mm,in=110,out=160] (1);
  \end{scope}
  
  \begin{scope}[shift={(3.5,-5)}]
    \node[state] (1) {};
    \node[state] (2) at (1,0) {};
    \node[state,accepting] (3) at (.5,{-sin(60)}) {};
  
   \path[->] (1) edge[blue,bend left=15] (2);
   \path[->] (2) edge[blue,bend right=25] (3);
   \path[->] (2) edge[red] (3);
   \path[->] (3) edge[red,bend right=25] (2);
   \path[->] (3) edge[blue,bend left=15] (1);
   \path[->] (1) edge[red,loop above,min distance=6mm,in=110,out=160] (1);
  \end{scope}
  
  \begin{scope}[shift={(5.5,-5)}]
    \node[state,accepting] (1) {};
    \node[state] (2) at (1,0) {};
    \node[state] (3) at (.5,{-sin(60)}) {};
  
   \path[->] (1) edge[red,bend left=15] (2);
   \path[->] (2) edge[red,bend right=25] (3);
   \path[->] (2) edge[blue] (3);
   \path[->] (3) edge[blue,bend right=25] (2);
   \path[->] (3) edge[red,bend left=15] (1);
   \path[->] (1) edge[blue,loop above,min distance=6mm,in=110,out=160] (1);

   \draw [ultra thick, rounded corners, draw=black, opacity=0.1]
       (-0.5,0.5) -- (-0.5,-1.25) -- (4.75,-1.25) -- (4.75,0.5) -- cycle;
  \end{scope}
  
  \begin{scope}[shift={(7.25,-5)}]
    \node[state] (1) {};
    \node[state,accepting] (2) at (1,0) {};
    \node[state] (3) at (.5,{-sin(60)}) {};
  
   \path[->] (1) edge[red,bend left=15] (2);
   \path[->] (2) edge[red,bend right=25] (3);
   \path[->] (2) edge[blue] (3);
   \path[->] (3) edge[blue,bend right=25] (2);
   \path[->] (3) edge[red,bend left=15] (1);
   \path[->] (1) edge[blue,loop above,min distance=6mm,in=110,out=160] (1);
  \end{scope}
  
  \begin{scope}[shift={(9,-5)}]
    \node[state] (1) {};
    \node[state] (2) at (1,0) {};
    \node[state,accepting] (3) at (.5,{-sin(60)}) {};
  
   \path[->] (1) edge[red,bend left=15] (2);
   \path[->] (2) edge[red,bend right=25] (3);
   \path[->] (2) edge[blue] (3);
   \path[->] (3) edge[blue,bend right=25] (2);
   \path[->] (3) edge[red,bend left=15] (1);
   \path[->] (1) edge[blue,loop above,min distance=6mm,in=110,out=160] (1);
  \end{scope}
  \vspace{10pt}
\end{tikzpicture}
\caption{The $13$ subgroups of index $3$ in $\Free[2] = \pres{a,b}{-}$ with the corresponding conjugacy classes enclosed in gray. Obviously, normal subgroups correspond to classes with a unique representative.}
\end{figure}

\vspace{10pt}
The characterization of finite index subgroups among the finitely generated ones, which follows from \Cref{prop: fi}, is even simpler and leads to the computability of the Finite Index Problem for free groups.

\begin{thm}\label{thm: FIP}
Let $H$ be a finitely generated subgroup of~$\Free[A]$. Then, the index $\ind{H}{\Free[A]}$ is finite if and only if $\stallings(H)$ is saturated. In particular, the Finite Index Problem is computable for $\Free[A]$.
\end{thm}

\begin{proof}
For the first claim, it is enough to realize that if $H$ is finitely generated, then $\stallings(H)$ is finite and computable, and to apply \Cref{prop: fi} taking into account that $\stallings(H)=\schreier(H)$ when saturated. The decision part of $\FIP(\Free[A])$ follows immediately (since, of course, the saturation condition is algorithmically checkable in a finite automaton).

Let us finally assume that $\stallings(H)$ is finite and saturated.
In order to compute a (finite) family  $R$ of right coset representatives  modulo $H$, recall that, since $\Ati =\stallings(H)=\schreier(H)$, it is enough to consider, for every vertex $\verti$ in $\Ati$, the label of some reduced walk in $\bp \xrwalk{\ } \verti$ in $\Ati$. A systematic way to obtain such a set is to compute a spanning tree $T$ for $\Ati$, and to take $R$ to be the labels of the $T$-geodesics from $\bp$ to every vertex in $\Ati$; \ie take $R=\set{ (\walki)\lab \st \walki \equiv \bp \xrwalk{_T} \verti , \verti \in \Verts(\Ati) }$.
\end{proof}
Hence,
\begin{equation}
 H\backslash \Free[A] \,=\, \bigsqcup_{\mathclap{\verti \in \Verts\stallings(H)}} \ H(T[\bp, p])\lab.
 \end{equation}
(Of course, a left transversal for $H$ is obtained inverting the elements in $R$, since $(Hw)^{-1} = w^{-1}H^{-1} = w^{-1}H$.)

\begin{exm}
Let $\Free[2]$ be the free group over $A=\{a,b\}$, and consider the subgroup $H=\langle v_1, v_2, v_3, v_4 \rangle\leqslant \Free[2]$, generated by the elements $v_1=a$, $v_2=b^2$, $v_3=ba^2b^{-1}$, and $v_4=baba^{-1}b^{-1}$. To determine whether $H$ has finite index in $\Free[2]$, we compute $\stallings(H)$ from the flower automaton $\flower({v_1, v_2, v_3, v_4})$, as explained in \Cref{thm: Stallings bijection}. The result is depicted below:
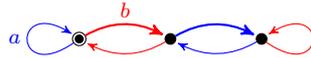
\vspace{-5pt}
\begin{figure}[H]
\centering
  \begin{tikzpicture}[shorten >=1pt, node distance=1.2 and 1.2, on grid,auto,>=stealth']
   \node[state,accepting] (0) {};
   \node[state] (1) [right = of 0]{};
   \node[state] (2) [right = of 1]{};

   \path[->]
        (0) edge[loop left,blue,min distance=10mm,in=145,out=215]
            node[left] {\scriptsize{$a$}}
            (0);
            
    \path[->]
        (0) edge[bend left,red,thick]
            node[above]{$b$} (1);
            
    \path[->]
        (1) edge[bend left,red]
            (0);
            
    \path[->]
        (1) edge[bend left,blue,thick] (2);
            
    \path[->]
        (2) edge[bend left,blue]
            (1);
            
    \path[->]
        (2) edge[loop left,red,min distance=10mm,in=35,out=325]
            (2);

\end{tikzpicture}
\end{figure}
\vspace{-5pt}

\noindent Since the three vertices of $\stallings(H)$ have both an incoming and an outgoing $a$-arc, as well as an incoming and an outgoing $b$-arc, the automaton $\stallings(H)$ is saturated, and therefore, $H$ has finite index in~$\Free[2]$; more concretely, it has index $\card \Verts\stallings(H)=3$. Choosing the spanning tree given by the two arcs drawn with bold lines, we obtain ${1, b, ba}$ as a set of representatives for the right cosets modulo $H$, meaning that $\Free[2]=H\sqcup Hb\sqcup H(ba)$.
\end{exm} 

The well-known result below, which relates the index of a subgroup of finite index to its rank, also transparently follows from Stallings' graphical description of subgroups.

\begin{thm}[Schreier index formula]\label{thm: Schreier index formula}
Let $\Free[\kappa]$ be a free group of rank $\kappa \in [1,\infty]$ and let  $H\leqslant \Free[\kappa]$ be a subgroup of finite index. Then,
 \begin{equation}\label{eq: Schreier index formula}
\rk(H)-1\,=\,\ind{H}{\Free[\kappa]}(\kappa-1) \,.
 \end{equation}
In particular, a finite index subgroup $H\leqslant \Free[\kappa]$ is finitely generated if and only if the ambient group $\Free[\kappa]$ is finitely generated.
\end{thm}


\begin{proof}
By \Cref{prop: fi}, if $H$ has finite index then $\stallings(H)$ is saturated (and hence $2 \kappa$-regular) and $\card \Verts(\stallings(H)) < \infty$. Therefore, $\rk(\stallings(H))<\infty$ (\ie there is a finite number of edges outside of any spanning tree) if and only if $\kappa<\infty$. In particular, if $\kappa =\infty$ then $\rk(H)=\infty$ and \eqref{eq: Schreier index formula} holds.

Let's now assume $\kappa <\infty$. By \Cref{cor: Schreier lemma} we know that $H$ is finitely generated; hence, $\stallings(H) =\schreier(H)$ is finite, saturated and, in particular, $2\kappa$-regular. Moreover, by \Cref{prop: B_T comp}, if $T$ is a spanning tree of $\stallings(H)$ then
 \begin{align*}
\rk (H)-1 &\,=\,\card(\Edgs^+\stallings(H)\setminus \Edgs T)-1 \,=\, \card \Edgs^+\stallings(H)-\card \Edgs T-1 \\ & \,=\, \card\Edgs^+\stallings(H)-\card \Verts T \,=\, \kappa \card\Verts\stallings(H)-\card\Verts\stallings(H) \\ & \,=\, \ind{H}{\Free[\kappa]}(\kappa-1),
 \end{align*}
where the penultimate equality comes from $2\kappa \,\card \Verts \stallings(H)=2\, \card \Edgs^+\stallings(H)$, obtained by summing the degrees of all the vertices.
\end{proof}

\begin{cor} \label{cor: Schreier lemma}
A finite index subgroup of a finitely generated group $G$ is always finitely generated. \qed
\end{cor}

\begin{proof}
Write $G$ as a quotient of a finitely generated free group, $\rho \colon \Fn \onto G$. Then, for any finite index subgroup $H\leqfi G$, the full preimage $H\rho\preim$ is a finite index subgroup of $\Fn$. By \Cref{thm: Schreier index formula}, it is finitely generated and so is its image~$H$. 
\end{proof}

Combining the graphical characterizations of finite index (\Cref{prop: fi}) and normality (\Cref{prop: normal iff}), we arrive at a kind of reciprocal of Schreier's Lemma for non-trivial normal subgroups of a finitely generated free group.

\begin{cor} \label{cor: fg normal iff}
A non-trivial normal subgroup $H$ of $\Fn$ is finitely generated if and only if it has finite index. \qed
\end{cor}

\begin{cor}
A non-trivial normal subgroup $H$ of $\Free[\aleph_0]$ always has infinite rank. \qed
\end{cor}


To finish this section we present a series of results proved my M.~Hall in the mid-20th century which follow smoothly from the graphical interpretation given by Stallings' automata. In this case the key geometric property is \Cref{lem: = defc}.

\begin{lem}[\citenr{hall_jr_coset_1949}] \label{lem: F rf}
For any nontrivial element $w \in \Free[A]$, there exists a subgroup $H\leqslant \Free[A]$ of (finite) index $|w|+1$ not containing $w$.
\end{lem}

\begin{proof}
It is enough to consider the reduced involutive  $A$-path $\Ati$ reading~$w$. Since~$\Ati$ is deterministic and finite, by \Cref{lem: = defc}, for each $a\in A$, we can pair up the $a$-deficient vertices with the $a^{-1}$-deficient vertices in $\Ati$ and join each pair with a new arc labeled by $a$. In this way, we obtain a saturated Stallings' automaton recognizing the claimed subgroup $H\leq \Free[A]$.
\end{proof}

\begin{figure}[H] 
  \centering
  \begin{tikzpicture}[shorten >=1pt, node distance=1.2 and 1.2, on grid,auto,>=stealth']
   \node[state, accepting] (0) {};
   \node[state] (1) [right = of 0]{};
   \node[state] (2) [right = of 1]{};
   \node[state] (3) [right = of 2]{};
   \node[state] (4) [right = of 3]{};
   \node[] (5) [right = 0.8 of 4]{};

   \path[->]
        (0) edge[loop,red,dashed,min distance=10mm,in=55,out=125]
            node[left = 0.1] {\scriptsize{$b$}}
            (0)
            edge[blue,dashed, bend right]
            (2);

    \path[->]
        (1) edge[blue]
        node[above] {\scriptsize{$a$}}
            (0)
            edge[red,dashed, bend left]
            (2);

    \path[->]
        (2) edge[red]
            (1)
            edge[blue]
            (3);

    \path[->]
        (3) edge[red]
            (4)
            edge[blue,dashed, bend right]
            (4);
            
    \path[->]
        (4) edge[blue,dashed, bend right = 40]
            (1)
            edge[red,dashed, bend right]
            (3);
\end{tikzpicture}
\caption{An example of an index 5 subgroup of $\Free_{\set{a,b}}$ not containing the commutator $a^{-1} b^{-1} a b$: $H=\langle b,\, aba,\, a^{-1}b^2a,\, a^{-1}b^{-1}a^2b^{-1}a^{-1}ba,$ $a^{-1}b^{-1}ab^2a^{-1}ba,\, a^{-1}b^{-1}aba^2\rangle$}
\label{fig: fi not containing w}
\end{figure}

Recall that a group $G$ is said to be \defin[residually finite group]{residually finite} if for every element $g \in G \setminus \set{\trivial}$, there exists a normal subgroup of $G$ of finite index not containing $g$.
Since finite index subgroups in finitely generated groups have finitely many conjugates, the intersection $c(H)=\bigcap_{x\in \Fn} H^x$ is finite and provides a finite index normal subgroup of $\Fn$ not containing $w$ as well. The result below follows immediately.

\begin{cor}
Free groups are residually finite. \qed
\end{cor}

The parallel result below is derived using the same idea (of saturating unsaturated automata). 

\begin{thm}[\citenr{hall_jr_topology_1950}] \label{thm: Hall vff}
If $H$ is a finitely generated subgroup of $\Free[A]$, then $H$ is a free factor of a finite index subgroup of $\Free[A]$.
\end{thm}

\begin{proof}
Consider the (finite) Stallings' automaton $\Ati =\stallings(H)$. If $\Ati$ is already saturated we take $K=H$. Otherwise, we can use the same approach as in the proof above to saturate $\Ati$ into the Stallings' automaton of the claimed subgroup $K$. It is clear by construction that $\stallings(H)$ is a subautomaton of (the saturated and finite automaton) $\stallings(K)$; hence, $H \leqslant\ff K\leqslant\fin \Free[A]$, as claimed.
\end{proof}

Finally, combining the proofs of \Cref{thm: Hall vff} and \Cref{lem: F rf}, we see that free groups are \defin[fully LERF group]{fully LERF}, \ie we can choose the subgroup $K$ in \Cref{thm: Hall vff} further avoiding any given finite subset~$S \subseteq \Free[A] \setmin H$. 

\begin{thm}[\citenr{hall_jr_topology_1950}]
Finitely generated free groups are fully LERF. In particular, every finitely generated subgroup of $\Free[A]$ is closed in the profinite topology.\qed
\end{thm}

\begin{figure}[H] 
  \centering
  \begin{tikzpicture}[shorten >=1pt, node distance=1.2 and 1.2, on grid,auto,>=stealth']
   \node[state, accepting] (0) {};
   \node[state] (1) [right = of 0]{};
   \node[state] (2) [right = of 1]{};
   \node[state] (3) [right = of 2]{};
   \node[state] (4) [right = of 3]{};
   \node[state] (5) [below = 1 of 2]{};

   \path[->]
        (0) edge[red,dashed, bend left = 50] (2)
            edge[blue, bend right = 25] (5);

    \path[->]
        (1) edge[red, bend left= 25]
        node[below] {\scriptsize{$b$}}
            (0)
            edge[blue, bend right= 25]
         node[above] {\scriptsize{$a$}}
            (0);

    \path[->]
        (2) edge[blue]
            (1)
            edge[red]
            (3);

    \path[->]
        (3) edge[blue] (4)
            edge[red, bend left] (5);
            
    \path[->]
        (4) edge[loop,red,dashed,min distance=10mm,in=35,out=-35](4)
        (4) edge[blue,dashed, bend right = 50]
            (2);
            
    \path[->]
        (5) edge[red]
            (1)
            edge[blue,dashed, bend left]
            (3);
\end{tikzpicture}
\caption{The Stallings' automaton of a finite index subgroup of $\Free_2$ containing $H = \gen{a b^2,a^{-2} b^{4}, a^{-1} b}$ (in solid lines) as free factor, and not containing the element~$a^{-2} ba$.}
\label{fig: fi containing H not containing w}
\end{figure}

\section{Intersections} \label{sec: intersections}

In this section we will use Stallings' automata to understand and effectively compute intersections of subgroups in the free group. Let us begin with a classical property and an algorithmic problem known to, respectively, hold and be solvable for free groups since the mid-20th century.

\begin{defn}
A group $G$ is said to satisfy the \defin{Howson property} (to be \emph{Howson}, for short) if the intersection of any two (and so,
 finitely many) finitely generated subgroups of $G$ is again finitely generated.
\end{defn}

\begin{named}[Subgroup Intersection Problem for $G=\pres{A}{R}$, $\SIP(G)$]
Given two finite subsets $S,R \subseteq (A^{\pm})^*$, decide whether the intersection $\gen{S}_G \cap \gen{R}_G$ is finitely generated and, in the affirmative case, compute a set of generators for it.  
\end{named}

\medskip

This property was named after A.G.~Howson, who first proved it to hold for free groups back in the 1950's. Below, we give a transparent argument for this result, using Stallings' automata. On the way, we will also obtain an efficient algorithm to solve the subgroup intersection problem for free groups, and will establish a general upper bound for the rank of the intersection, known as the \defin{Hanna Neumann inequality} (see \Cref{thm: Howson}). Of course, in the case of Howson groups, the decision part of the subgroup intersection problem is irrelevant and it reduces to the computability of a set of generators for the intersection. 

We recall the reader that, not far from free groups, one can easily find non-Howson groups; the simplest example being $\Free[2]\times \mathbb{Z}$ (see~\cite{moldavanskii_intersection_1968} and \cite{Delgado_stallings_FTA_2022}).


The key link between Stallings' automata and intersections is the notion of (categorical) product or pullback of automata, defined below.

\begin{defn} \label{def: product of automata}
The \defin[pullback of automata]{pullback} of two $A$-automata $\Ati_{\hspace{-1pt}1}$ and $\Ati_{\hspace{-1pt}2}$, denoted by ${\Ati_{\!1} \times \Ati_{\!2}}$, is the $A$-automaton with set of vertices $\Verts \Ati_{\hspace{-1pt}1} \times \Verts \Ati_{\hspace{-1pt}2}$, an arc $(\verti_1,\verti_2) \xarc{a\,} (\vertii_1,\vertii_2)$ for each pair of arcs $\verti_1 \xarc{a\,} \vertii_1$ in~$\Ati_{\hspace{-1pt}1}$, and $\verti_2 \xarc{a\,} \vertii_2$ in~$\Ati_{\hspace{-1pt}2}$ (with the same label $a\in A$), and basepoint $\bp=(\bp_1, \bp_2)$; see~\Cref{fig doubly enriched arc}.
\end{defn}

\begin{figure}[h]
\centering
\begin{tikzpicture}[shorten >=1pt, node distance=1.25cm and 1.25cm, on grid,auto,>=stealth']
\newcommand{\dx}{1}
\newcommand{\dy}{1}
   \node (00) {};
   \node[state,accepting] (0) [below = \dy*0.5 of 00]{};
   \node[] (bp1) [below right =0.1 and 0.15 of 0] {$\scriptscriptstyle{1}$};
   \node[state,accepting] (0') [right = \dx*0.5 of 00]{};
   \node[] (bp1) [below right =0.1 and 0.15 of 0'] {$\scriptscriptstyle{2}$};
   \node[state,accepting,blue] (00') [right = \dx*0.5 of 0]{};
   \node[state] (1) [below = \dy*0.75 of 0] {};
   \node[state] (2) [below = \dy of 1] {};
   \node[] (G') [left = 0.8 of 1] {$\scriptstyle{\Ati_{\hspace{-1pt}1}}$};
   \node[state] (1') [right = \dx*0.75 of 0'] {};
   \node[state] (2') [right = \dx of 1'] {};
   \node[] (G) [right = 0.7 of 2'] {$\scriptstyle{\Ati_{\hspace{-1pt}2}}$};
   \node[state,blue] (11') [below right = \dy*0.75 and \dx*0.75 of 00'] {};
   \node[state,blue] (22') [below right = \dy and \dx of 11'] {};
     \node[blue] (GG') [right = 0.9 of 22'] {$\scriptstyle{\Ati_{\hspace{-1pt}1} \times \Ati_{\hspace{-1pt}2}}$};

   \path[dashed]
       (0') edge[] (1');
   \path[dashed]
       (0) edge[] (1);
   \path[dashed,blue]
       (00') edge[] (11');
   \path[->]
       (1') edge[]
             node[pos=0.5,above=-.1mm] {$a$}
             (2');
   \path[->]
         (1) edge[]
             node[pos=0.5,left=-.1mm] {$a$}
             (2);
   \path[->,blue]
         (11') edge[]
             node[pos=0.5,above right] {$a$}
             (22');
\end{tikzpicture}
\caption{Product (in blue) of two $A$-automata (in black)} \label{fig doubly enriched arc}
\end{figure}

\begin{exm}\label{exe: producte}
Consider $H=\gen{b, a^3, a^{-1}bab^{-1}a}, K=\gen{ab,a^3, a^{-1}ba}\leqslant\fg\Free[\set{\!a,b\!}]$. In \Cref{fig: pullback} we can see $\stallings(H)$ depicted above (in horizontal format), $\stallings(K)$ depicted on the left (in vertical format), and its pullback $\stallings(H)\times \stallings(K)$ computed in the central part (and reorganized on the right). 
\vspace{5pt}
\begin{figure}[H]
\centering
\begin{tikzpicture}[shorten >=1pt, node distance=1.2cm and 2cm, on grid,auto,auto,>=stealth']
\newcommand{\dx}{1.3}
\newcommand{\dy}{1.2}
\node[] (0)  {};
\node[state,accepting] (a1) [right = \dy-1/3 of 0] {};
\node[state] (a2) [right = \dx of a1] {};
\node[state] (a3) [right = \dx of a2] {};
\node[state] (a4) [right = \dx of a3] {};

\node[state,accepting] (b1) [below = 2*\dy/3 of 0] {};
\node[state] (b2) [below = \dy of b1] {};
\node[state] (b3) [below = \dy of b2] {};

\foreach \y in {1,...,4}
\foreach \x in {1,...,3} 
\node[state] (\x\y) [below right = (\x-1/3)*\dy and (\y-1/3)*\dx of 0] {};
    
\node[state,accepting] () [below right = 2*\dy/3 and 2*\dx/3 of 0] {};

\path[->]
     (a1) edge[red,loop above,min distance=7mm,in=205,out=155]
     node[above = 0.1] {$b$}
     (a1)
     (a1) edge[blue] node[below] {$a$} (a2)
     (a2) edge[blue] (a3)
     (a3) edge[blue,bend right=30] (a1)
     (a3) edge[red] (a4)
     (a4) edge[blue,loop above,min distance=7mm,in=25,out=-25] (a4)

     (b1) edge[blue,bend right=25] (b2)
     (b2) edge[red,bend right=25] (b1)
     (b2) edge[blue] (b3)
     (b3) edge[blue,bend left=40] (b1) 
     (b3) edge[red,loop left,min distance=7mm,in=295,out=245] (b3);
             
 \path[->]   
     (11) edge[blue] (22)
     (22) edge[blue] (33) 
     (33) edge[blue, bend right=25] (11)
            
     (21) edge[blue] (32)
     (32) edge[blue] (13)
     (13) edge[blue] (21)
        
     (12) edge[blue] (23)
     (23) edge[blue] (31) 
     (31) edge[blue] (12)
            
     (14) edge[blue] (24)
     (24) edge[blue] (34)
     (34) edge[blue, bend right=25] (14)
            
     (21) edge[red] (11)
     (23) edge[red] (14)
     (33) edge[red] (34)            
     (31) edge[red,loop left,min distance=7mm,in=250,out=200] (31);
    
 \node[] (i) [right = 1 of 24] {$=$};  
 \node[state,accepting] (1) [above right = 0.7 and 2.25 of i] {};
 \node[state] (1a) [below left = \dy and \dx/3 of 1] {};
 \node[state] (1b) [below right = \dy and \dx/3 of 1] {};
 \node[state] (2) [right = \dx*0.7 of 1b] {};
 \node[state] (2a) [above left = \dy and \dx/3 of 2] {};
 \node[state] (2b) [above right = \dy and \dx/3 of 2] {};
 \node[state] (3) [right = \dx*0.7 of 2b] {};
 \node[state] (3a) [below left = \dy and \dx/3 of 3] {};
 \node[state] (3b) [below right = \dy and \dx/3 of 3] {};
 \node[state] (0) [left = \dx*0.7 of 1a] {};
 \node[state] (0a) [above left = \dy and \dx/3 of 0] {};
 \node[state] (0b) [above right = \dy and \dx/3 of 0] {};

 \path[->]
     (0b) edge[red,thick] (1)
     (1b) edge[red,thick] (2)
     (3) edge[red,thick] (2b)
    
     (0b) edge[blue,thick] (0)
     (0a) edge[blue,thick] (0b)
     (0) edge[blue] (0a)
    
     (1b) edge[blue,thick] (1)
     (1a) edge[blue,thick] (1b)
     (1) edge[blue] (1a)
    
     (2) edge[blue,thick] (2b)
     (2b) edge[blue,thick] (2a)
     (2a) edge[blue] (2)
    
     (3) edge[blue,thick] (3b)
     (3b) edge[blue,thick] (3a)
     (3a) edge[blue] (3)
     (3b) edge[red,loop,min distance=7mm,in=25,out=-25] (3b);
\end{tikzpicture}
\vspace{-5pt}
\caption{
The pullback of $\stallings(H)$ and $\stallings(K)$}
\vspace{-5pt}
\label{fig: pullback}
\end{figure}
\end{exm}
\vspace{5pt}
The following proposition summarizes the first properties of the pullback of automata, highlighting its close relationship with the intersection of subgroups. 

\begin{prop} \label{prop: prod props}
Let $\Ati_{\hspace{-1pt}1} $ and $\Ati_{\hspace{-1pt}2}$ be two $A$-automata and let $\Ati = \Ati_{\hspace{-1pt}1} \times \Ati_{\hspace{-1pt}2}$. Then,  
 \begin{enumerate}[ind]
\item\label{item: deg prod} 
for every vertex $(\verti,\vertii)$ in $ 
\Ati
$, $0\leqslant \deg_{\Ati}(\verti,\vertii)
\leqslant
\min \{\deg_{\Ati_{\hspace{-1pt}1}}(\verti),\, \deg_{\Ati_{\hspace{-1pt}2}}(\vertii)\}$;
\item\label{item: int det}
if $\Ati_{\hspace{-1pt}1}$ and $\Ati_{\hspace{-1pt}2}$ are deterministic then
$\Ati 
$ is deterministic and $\gen{\Ati}=\gen{\Ati_{\hspace{-1pt}1}}\cap \gen{\Ati_{\hspace{-1pt}2}}$;
\item \label{item: pullback computable} if $\Ati_{\hspace{-1pt}1}$ and $\Ati_{\hspace{-1pt}2}$ are finite then $\Ati$ is finite and computable;
\item\label{item: prod not connected} 
connectedness of $\Ati_{\hspace{-1pt}1}$ and $\Ati_{\hspace{-1pt}2}$ does not imply connectedness of $\Ati$;
\item\label{item: prod no core} 
coreness of $\Ati_{\hspace{-1pt}1}$ and $\Ati_{\hspace{-1pt}2}$ does not imply coreness of $\Ati$.
 \end{enumerate}
\end{prop}

\begin{proof}
\ref{item: deg prod} is clear from \Cref{def: product of automata}, since every $a$-arc leaving from (\resp arriving to) any vertex $(\verti,\vertii) \in \Verts \Ati$ corresponds to an $a$-arc leaving from (\resp arriving to) $\verti \in \Verts \Ati_{\hspace{-1pt}1}$ and to an $a$-arc leaving from (\resp arriving to) $\vertii \in \Verts\Ati_{\hspace{-1pt}2}$.

The first assertion in \ref{item: int det} follows easily from the definitions of determinism and pullback of $A$-automata. Now, if $\Ati_{\hspace{-1pt}1}$ and $\Ati_{\hspace{-1pt}2}$ are both deterministic then, for every $w\in \gen{\Ati }$, there is a (unique) reduced $\bp$-walk in $\Ati$ with label~$w$; projecting it to the first (\resp second) coordinate, we obtain a reduced $\bp_1$-walk (\resp $\bp_2$-walk) in $\Ati_{\hspace{-1pt}1}$ (\resp in $\Ati_{\hspace{-1pt}2}$) spelling $w$ as well; so, $w\in \gen{\Ati_{\hspace{-1pt}1}}\cap \gen{\Ati_{\hspace{-1pt}2}}$. 
And conversely, if $w\in \gen{\Ati_{\hspace{-1pt}1}}\cap \gen{\Ati_{\hspace{-1pt}2}}$ (with both $\Ati_{\hspace{-1pt}1}$ and $\Ati_{\hspace{-1pt}2}$ deterministic), then there is a (unique) reduced $\bp_1$-walk $\walki_1$ in $\Ati_{\hspace{-1pt}1}$ and a (unique) reduced $\bp_2$-walk $\walki_2$ in $\Ati_{\hspace{-1pt}2}$, both labeled by $w$. Now, \Cref{def: product of automata} allows us to stitch $\walki_1$ and $\walki_2$, arc by arc, into a (unique) reduced $\bp$-walk in $\Ati$ labeled by $w$; hence, $w\in \gen{\Ati}$.

Assertion \ref{item: pullback computable} is again obvious from  the definition of pullback of automata.

Finally, the subgroups $H=\gen{a^2, b}$ and $K=\gen{b, aba^{-1}}$ of $\Free[\{a,b\}]$ serve as a common example for assertions \ref{item: prod not connected} and \ref{item: prod no core}.  
\end{proof}

The description of the intersection subgroup in terms of Stallings' automata follows immediately. 

\begin{cor}\label{cor: x}
Let $H,K\leqslant \Free[A]$. Then, $\stallings(H\cap K)=\core (\stallings(H) \times \stallings(K))$. \qed
\end{cor}

From here, we can easily deduce the Howson property for free groups, the computability of intersections of finitely generated subgroups (and hence that of the \SIP), and the so-called Hanna Neumann bound for the rank of the intersection, within free groups. We define the \defin{reduced rank} of a free group $\Free[n]$ as $\rrk(\Free[n])=\max\{0, n-1\}$; that is, $\rrk(\Free[n])=n-1$ except for the trivial group, for which we define $\rrk(1)=0$ (instead of~$-1$). The reduced rank of a graph (or involutive automaton) is defined accordingly.

\begin{thm}[Howson, \cite{howson_intersection_1954}; H. Neumann, \cite{neumann_intersection_1956}]\label{thm: Howson}
Free groups $\Free[A]$ (of any rank) satisfy the Howson property and have solvable \SIP. Moreover, for any $H,K\leqslant\fg \Free[A]$, $\rrk(H\cap K)\leq 2\rrk(H)\rrk(K)$.
\end{thm}

\begin{proof}
Since the stated results only involve finitely generated subgroups, without loss of generality we can assume that the ambient $\Free[A]$ is of finite rank. Now, if $H,K\leqslant \Free[n]$ are finitely generated then the $A$-automata $\stallings(H)$ and $\stallings(K)$ are finite and so is $\stallings(H)\times \stallings(K)$. From \Cref{prop: prod props}.\ref{item: int det},
we deduce that $H\cap K$ is, again, finitely generated; therefore, $\Free[n]$ is Howson. Moreover, given finite sets of generators for $H$ and $K$, one can effectively (i) compute $\stallings(H)$ and $\stallings(K)$, (ii) construct the pullback $\stallings(H)\times \stallings(K)$, (iii) take the core and obtain $\stallings(H\cap K)$, and (iv) choose a spanning tree and compute a basis for $H\cap K$; this solves the subgroup intersection problem in $\Free[n]$. 
 
To see the Hanna Neumann bound on ranks, we can restrict ourselves to the case $H\cap K\neq \Trivial$ (otherwise, it is obvious); so, let us assume that $H\cap K$, and hence $H$ and $K$, are not trivial. Write $\Ati_{\hspace{-1pt}H} =\stallings(H)$, $\Ati_{\hspace{-1pt}K} =\stallings(K)$, $\Ati^*_{\!H}=\core^*(\Ati_{\!H})$, $\Ati^*_{\!K}=\core^*(\Ati_{\!K})$, and $\Ati =\Ati_{\!H} \times \Ati_{\!K}$. Let also $\Ati^* =\stallings^{*}(H \cap K) \!\subseteq \Ati^*_{\!H} \times \Ati^*_{\!K}\subseteq \Ati_{\!H} \times \Ati_{\!K}$ be the strict core of the connected component of $\Ati$ containing the basepoint. Note that all of them ($\Ati_{\!H}, \Ati_{\!K}, \Ati^*_{\!H}, \Ati^*_{\!K}, \Ati$ and~$\Ati^*$) are non-trivial by the assumption that $H\cap K\neq \Trivial$. 


Now, since $\Ati^* \!\subseteq \Ati^*_{\!H} \times \Ati^*_{\!K}$, every vertex $(\verti,\vertii) \in \Verts\Ati^*$ satisfies $\deg_{\Ati^*_{\hspace{-1pt}H}}(\verti) \geq 2$ and $\deg_{\Ati^*_{\hspace{-1pt}K}}(\vertii) \geq 2$ and, by \Cref{prop: prod props}\ref{item: deg prod}, 
 \begin{equation}\label{eq: deg-2}
\deg_{\Ati^*}(\verti,\vertii) - 2
\,\leq\,
\deg_{\Ati^*_{\!\!H} \times \Ati^*_{\!\!K}}(\verti,\vertii) - 2
\,\leq\,
(\deg_{\Ati^*_{\!\!H}}(\verti) - 2) (\deg_{\Ati^*_{\!\!K}}(\vertii) - 2)\,.
 \end{equation}
Finally, taking into account that the rank of a connected automaton equals the rank of its strict core, we have: 
 \begin{equation}\label{eq: hnproof}
\begin{aligned}
\rk(H\cap K)-1
&\,=\,
\rk (\Ati) -1 \\
&\,=\,
\rk (\Ati^*) -1 \\
&\,=\,
\textstyle{\frac{1}{2}\sum_{(\verti,\vertii)\in \Verts\Ati^*} 
( \deg_{\Ati^*}(\verti,\vertii)-2)}\\
&\,\leq\,
\textstyle{\frac{1}{2}\sum_{(\verti,\vertii)\in \Verts\Ati^*} (\deg_{\Ati^*_{\!\!H}}(\verti) - 2) (\deg_{\Ati^*_{\!\!K}}(\vertii) - 2)}\\
&\,\leq\,
\textstyle{\frac{1}{2}\sum_{(\verti,\vertii)\in \Verts(\Ati^*_{\!H}\times \Ati^*_{\!K})} (\deg_{\Ati^*_{\!\!H}}(\verti) - 2) (\deg_{\Ati^*_{\!\!K}}(\vertii) - 2)}\\
&\,=\,
\textstyle{\frac{1}{2}\big(\sum_{\verti\in \Verts\Ati^*_{\!\!H}} (\deg_{\Ati^*_{\!\!H}}(\verti)-2) \big)
 \big( \sum_{\vertii\in \Verts\Ati^*_{\!\!K}} ( \deg_{\Ati^*_{\!\!K}}(\vertii)-2) \big)}\\
&\, =\, \frac{1}{2}\cdot 2(\rk(\Ati^*_{\!\!H})-1)\cdot 2(\rk(\Ati^*_{\!\!K})-1) \\
&\, =\, 2(\rk(\Ati_{\!\!H})-1)\cdot (\rk(\Ati_{\!\!K})-1) ,\\
\end{aligned}
 \end{equation}
and the proof is complete.
\end{proof}

\begin{exm}
As an application, let us finish the computation in \Cref{exe: producte} and find a basis for the intersection of the subgroups of $\Free[\{a,b\}]$ given by $H=\langle u_1, u_2, u_3\rangle$ and $K=\langle v_1, v_2, v_3\rangle$, where $u_1=b$, $u_2=a^3$, $u_3=a^{-1}bab^{-1}a$ and $v_1=ab$, $v_2=a^3$, $v_3=a^{-1}ba$. In \Cref{exe: producte} we have already computed $\stallings(H)$, $\stallings(K)$ and $\stallings(H)\times \stallings(K)$. Since this last one is connected and core, \Cref{cor: x} tells us that $\stallings(H\cap K)=\stallings(H)\times \stallings(K)$. Choosing as spanning tree $T$ the one indicated by the boldfaced arcs in \Cref{fig: pullback}, we get the basis 
 \begin{equation*}
S_T =\set{ b^{-1}a^3 b, a^3,a^{-1}ba^3b^{-1}a, a^{-1}bab^{-1}a^3ba^{-1}b^{-1}a, a^{-1}bab^{-1}aba^{-1}ba^{-1}b^{-1}a}
 \end{equation*}
for the intersection $H\cap K$. Moreover, since each of the five generators in $S_T$ is the label of a known reduced $\bp$-walk in $\stallings(H)\times \stallings(K)$, we can project it to the first and second coordinate, respectively, and obtain expressions for the corresponding generators as words on $\{u_1, u_2, u_3\}$, and on $\{ v_1, v_2, v_3\}$:
 \begin{equation*}
\begin{array}{rcl}
H\ni u_1^{-1} u_2 u_1 =\ & b^{-1}a^3 b &\ = v_1^{-1}v_2 v_1 \in K, \\
H\ni u_2 =\ & a^3 &\ = v_2 \in K, \\ 
H\ni u_3^3 =\ & a^{-1}ba^3b^{-1}a &\ = v_3v_2 v_3^{-1} \in K, \\
H\ni u_3u_2u_3^{-1} =\ & a^{-1}bab^{-1}a^3ba^{-1}b^{-1}a &\ = v_3v_1^{-1}v_2v_1v_3^{-1}\in K, \\
H\ni u_3u_1u_3^{-1} =\ & a^{-1}bab^{-1}aba^{-1}ba^{-1}b^{-1}a &\ = v_3v_1^{-1}v_2v_3v_2^{-1}v_1v_3^{-1}\in K.
\end{array}
 \end{equation*}
\end{exm}

Observe that the proof of \Cref{thm: Howson} only uses the (strict core of the) connected component of the pullback containing the basepoint, and ignores completely the other ones. A more detailed argument taking into account all the connected components, will improve the above Hanna Neumann inequality into the so-called \defin{strengthened Hanna Neumann inequality} (proved several years later by Walter Neumann in~\cite{kovacs_intersections_1990}). To reach this result (\Cref{thm: SHN} below), we need to previously discuss some technical details. 

In accordance with the algebraic definition of rank and reduced rank, we define the reduced rank of a connected graph (or involutive automaton) $\Ati$ as $\rrk(\Ati)=\max\{0, \rk(\Ati)-1\}$, and the rank and reduced rank of an arbitrary graph as $\rk(\Ati)=\sum_{C \in \CC(\Ati)} \rk(C)$ and $\rrk(\Ati)=\sum_{C \in \CC(\Ati)} \rrk(C)$, respectively, where the sums run over the set $\CC(\Ati)$ of connected components of $\Ati$. Note also that the standard formula $\sum_{\verti\in \Verts C} (\deg(\verti)-2) =2(\# \Edgs C-\#\Verts C)= 2(\rk(C)-1)$, valid for any connected and finite $C$ (no matter if being a tree or not), can alternatively be written in the form 
 \begin{equation}\label{eq: rrk}
\sum_{\verti\in \Verts C} (\deg(\verti)-2) = \left\{ \begin{array}{ll} 2\rrk(C) & \text{if } C \text{ is not a tree}, \\ -2 & \text{if } C \text{ is a tree}. \end{array} \right. 
 \end{equation}
If $\CC =\CC(\Ati)$, we denote by $\TT_\CC \subseteq \CC$ the set of connected components in $\Ati$ which are trees. 

\begin{lem}\label{lem: d-2}
Let $\Ati_{\hspace{-1pt}1}$ and $\Ati_{\hspace{-1pt}2}$ be two $A$-automata. For every $(\verti,\vertii) \in \Verts(\Ati_1 \times \Ati_2)$ such that $\deg_{\Ati_{\!1}}(\verti)\geq 2$ and $\deg_{\Ati_{\!2}}(\vertii)\geq 2$, we can improve the second inequality in \eqref{eq: deg-2} to
\[\eta +\big( \deg_{\Ati_{\!1} \times \Ati_{\!2}}(\verti,\vertii)-2 \big) \,\leq\, \left( \deg_{\Ati_{\!1}}(\verti)-2\big)\big( \deg_{\Ati_{\!2}}(\vertii)-2 \right),
\]
where $\eta=2$ if $(\verti, \vertii)$ is an isolated vertex, $\eta=1$ if it has degree $1$, and $\eta=0$ otherwise. \qed 
\end{lem}

\begin{prop}\label{prop: untros}
Let $\Trivial \neq H,K\leqslant\fg \Free[A]$, let $\Ati_{\!H}^* =\core^*(\stallings(H))$, $\Ati_{\hspace{-1pt}K}^* =\core^*(\stallings(K))$, and $\Ati =\Ati_{\!H}^* \times \Ati_{\!K}^*$.
Then, $\rrk( \Ati )\leq 2\rrk(H)\rrk(K)$.
\end{prop}

\begin{proof}
Let us abbreviate $\CC =\CC(\Ati)$, $\TT =\TT_\CC$, and $\deg =\deg_{\Ati}$. By mimicking the arguments and calculations in the proof of Theorem 1.7.6, while now considering all the connected components in \Cref{thm: Howson}, we have:
 \begin{align*}
& 2\rrk( \Ati ) \,=\, 2\sum_{C \in \CC} \rrk (C) \,=\, 
\sum_{\mathclap{C \in \CC \setminus \TT}} 2\rrk (C) \,=\, \sum_{\mathclap{C \in \CC \setminus \TT}} 2 (\rk (C) - 1) \\
&\,=\, 
\hspace{-5pt}
\sum_{C\in \CC \setminus \TT} \sum_{(\verti,\vertii)\in \Verts C} \hspace{-8pt} \big( \deg(\verti,\vertii)-2\big) +\sum_{C\in \TT} \Big( 2+\sum_{\mathclap{(\verti,\vertii)\in \Verts C}} \big( \deg(\verti,\vertii)-2\big) \Big)\\
& \,\leq\,
\hspace{-5pt}
\sum_{C\in \CC \setminus \TT} \sum_{(\verti,\vertii)\in \Verts C} \hspace{-8pt} \big( \deg_{\Ati_{\!\!H}^*}(\verti)-2\big)\big(\deg_{\Ati_{\!\!K}^*}(\vertii)-2 \big) + \sum_{C\in \TT} \sum_{(\verti,\vertii)\in \Verts C} \hspace{-8pt}\big( \deg_{\Ati_{\!\!H}^*}(\verti)-2\big)\big(\deg_{\Ati_{\!\!K}^*}(\vertii)-2 \big)\\
&\,=\,
\hspace{-8pt}
\sum_{(\verti,\vertii)\in \Verts \Ati} \hspace{-8pt} (\deg_{\Ati_{\!\!H}^*}(\verti)-2)(\deg_{\Ati_{\!\!K}^*}(\vertii)-2) \,=\, \Big(\sum_{\mathclap{\verti\in \Verts\Ati_{\!\!H}^*}} \big(\deg_{\Ati_{\!\!H}^*}(\verti)-2\big) \Big) \Big( \sum_{\mathclap{\vertii\in \Verts\Ati_{\!\!K}^*}} \big( \deg_{\Ati_{\!\!K}^*}(\vertii)-2\big) \Big)\\
&\,=\,
2\rrk(\Ati_{\!\!H}^*) \cdot 2\rrk(\Ati_{\!\!K}^*)
\,=\, 4\, \rrk(H) \, \rrk(K), 
 \end{align*}
where, by~\eqref{eq: rrk}, 
the extra summand
$\sum_{C\in \TT} \big( 2+\sum_{(\verti,\vertii)\in \Verts C} \big( \deg(\verti,\vertii)-2\big) \big)$
in the second line equals zero;
and the inequality step uses \Cref{lem: d-2} and the fact that every tree component $C\in \TT$ is either an isolated vertex, or has at least two vertices of degree~$1$. Simplifying a factor of $2$, we obtain the desired inequality.
\end{proof}

Finally, in order to understand the algebraic meaning of the connected components in the pullback, we need to consider the set $H\backslash \Free[A]/K = \set{HwK \st w\in \Free[A]}$ of double cosets in $\Free[A]$ modulo $H$ and $K$.

\begin{lem}\label{lem: double cosets}
Let $H$ and $K$ be subgroups of $\Free[A]$ and let $u,v\in \Free[A]$. 
Then, $(H u, K)$ and $(H v, K)$ belong to the same connected component of the pullback $\schreier(H)\times \schreier(K)$ if and only if $HuK=HvK$.
\end{lem}

\begin{proof}
By definition, $(Hu,K)$ and $(Hv, K)$ belong to the same connected component of $\schreier(H)\times \schreier(K)$ if and only if there exists some $w\in \Free[A]$ labeling simultaneously a reduced walk $Hu \xrwalk{w} Hv$ in $\schreier(H)$, and a reduced walk $K \xrwalk{w} K$ in $\schreier(K)$; and this will happen if and only if simultaneously $uw v^{-1}\in H$ and $w\in K$. Such a $w$ exists if and only if $HuK=HvK$.
\end{proof}

With this lemma at hand, we can consider the map $\varsigma$ mapping $HuK$ to the connected component $D_u$ containing $(Hu, K)$:
 \begin{equation}
\begin{array}{rcl}
\varsigma \colon  H\backslash \Free[A]/K &  \xto{} & \CC(\schreier(H)\times \schreier(K))\\
HuK & \mapsto & D_u, 
\end{array} 
 \end{equation}
which is well defined by \Cref{lem: double cosets}. 

\begin{lem}\label{lem: bijection}
The map $\varsigma$ above is bijective. Moreover, for every $u \in \Free[A]$, we have that $\rrk(D_u)=\rrk(H^u\cap K)$, where $H^u =u^{-1}Hu$.
\end{lem}

\begin{proof}
To see the bijectivity of $\varsigma$ let us construct its inverse map. Given a connected component $D \in \CC(\schreier(H)\times \schreier(K))$, choose an arbitrary vertex in $D$ of the form $(Hu, K)$ (there is always such a vertex thanks to the saturation of $\schreier(H)$) 
and consider the double coset $HuK$. We claim that $HuK\mapsfrom D$ defines a map $\partial$ from the set of connected components of $\schreier(H)\times \schreier(K)$ back to~${H\backslash \Free[A]/K}$. In fact, choosing another such vertex $(Hu',K)\in D$ 
there exists $k\in K$ such that $Hu \xrwalk{k} Hu'$ and so, $H\xrwalk{u} Hu \xrwalk{k} Hu'$; this means that $uk(u')^{-1}\in H$ and so, $HuK=Hu'K$, proving the claim. Now it is straightforward to see that both $\varsigma\partial$ and $\partial\varsigma$ are the identity map and so, $\partial=\varsigma^{-1}$ and $\varsigma$ is bijective.

Finally, given $u\in \Free[A]$, consider the double coset $HuK$ and its corresponding connected component $D_u= (HuK)\varsigma$, which contains the vertex $(Hu, K)$. Clearly,
\begin{align*}
\rrk(D_u)
&\,=\,
\rrk(\gen{D_u}_{(Hu,K)})
\,=\,\rrk(\gen{\schreier(H)\times \schreier(K)}_{(Hu, K)})\\
&\,=\, \rrk(\gen{\schreier(H)}_{Hu}\cap \gen{\schreier(K)}_{K}\!)
\,=\, \rrk(H^u\cap K),
\end{align*}
as claimed.
\end{proof}

As above, let $\Ati_{\hspace{-1pt}H}^* =\core^*(\stallings(H))$ and $\Ati_{\hspace{-1pt}K}^* =\core^*(\stallings(K))$. In order to connect \Cref{prop: untros} with Lemmas~\ref{lem: double cosets} and \ref{lem: bijection}, we need to relate the set of connected components of $\Ati =\Ati_{\!H}^* \times \Ati_{\!K}^*$, denoted by $\CC =\CC(\Ati)$, with that of $\schreier(H)\times \schreier(K)$, denoted by $\DD =\CC(\schreier(H)\times \schreier(K))$. Since $\Ati\subseteq \schreier(H)\times \schreier(K)$, every $C\in \CC$ in contained in a unique connected component in $\DD$, which we denote by $D_C$. This determines a well-defined map:
\begin{equation}
\begin{array}{rcl}
\varepsilon \colon  \CC(\Ati_{\!H}^* \times \Ati_{\!K}^*) & \xto{} & \CC(\schreier(H)\times \schreier(K)) \\ C & \mapsto & D_C 
\end{array}
\end{equation}

\begin{prop} \label{prop: beta injective}
The map $\varepsilon$ above is injective, preserves the rank (and reduced rank), and restricts to a bijection $\varepsilon\colon \CC \setminus \TT_\CC \xto{} \DD \setminus \TT_\DD$.
\end{prop}


\begin{proof}
To see injectivity, suppose that $C_1,C_2 \in \CC(\Ati_{\!H}^* \times \Ati_{\!K}^*)$ are contained in the same $D_{C_1} = D_{C_2} \in \DD$ and consider vertices $(\verti_1, \vertii_1)\in \Verts C_1$ and $(\verti_2, \vertii_2)\in \Verts C_2$. Since $(\verti_1, \vertii_1)$ and $(\verti_2, \vertii_2)$ belong to the same connected component of $\schreier(H) \times \schreier(K)$ there must exist two reduced walks $\walki_H \equiv \verti_1 \xrwalk{\ } \verti_2$ in $\schreier(H)$ and $\walki_K \equiv \vertii_1 \xrwalk{\ } \vertii_2$ in $\schreier(K)$ reading the same label, say $w$. Since $\verti_1,\verti_2 \in \Ati_{\!H}^*$ and $\vertii_1,\vertii_2 \in \Ati_{\!K}^*$ this means that $\walki_H$ and $\walki_K$ are reduced walks reading $w$ in $\Ati_{\!H}^*$ and $\Ati_{\!K}^*$ respectively. Hence $(\verti_1,\verti_2) \xrwalk{w} (\vertii_1,\vertii_2)$ in $\Ati_{\!H}^* \times \Ati_{\!K}^*$, and $C_1 = C_2$, proving the injectivity of $\varepsilon$. 

For the second claim it is enough to see that for any $C\in \CC$, $\core^*(D_C ) \,=\, \core^*(C) \,\subseteq\, C\subseteq D_C$, and hence $\rk(C)=\rk(D_C )$ and $\rrk(C)=\rrk(D_C )$. All inclusions are immediate except $\core^*(D_C)\subseteq \core^*(C)$; to see this one, note that any non-trivial, cyclically reduced, closed walk $\walki$ in $D_C$ projects to a non-trivial, cyclically reduced, closed walk in $\schreier(H)$ and $\schreier(K)$, respectively; so, these projections are, respectively, in $\core^*(\schreier(H))=\Ati_{\hspace{-1pt}H}^*$ and $\core^*(\schreier(K))=\Ati_{\hspace{-1pt}K}^*$ and therefore $\walki$ was, in fact, in $C$. 

For the last claim, it remains to see that every $D\in \DD \setminus \TT_\DD$ belongs to the image of $\varepsilon$. In fact, take a non-trivial, cyclically reduced, closed walk in $D$ and, by the same argument as before, it must be contained in $\Ati_{\hspace{-1pt}H}^* \times \Ati_{\hspace{-1pt}K}^*$.
\end{proof}

\begin{cor} \label{cor: rrk =}
$\rrk(\Ati_H^* \times \Ati_K^*) = \rrk(\schreier(H)\times \schreier(K))$.    \qed 
\end{cor}


This is the final component required to complete the proof of the Strengthened Hanna-Neumann inequality, stated below.

\begin{thm}[W.~Neumann, \cite{kovacs_intersections_1990}]\label{thm: SHN}
For any $H,K\leqslant \Free[A]$, we have 
 \begin{equation}\label{eq: SHN}
\sum_{HwK\in H\backslash \Free[A]/K} \hspace{-10pt} \rrk(H^w\cap K)\,\leq\, 2\rrk(H)\rrk(K),
 \end{equation}
where $H^w=w^{-1}Hw$, and the sum runs over all the $w$'s in any set of double coset representatives for $H\backslash \Free[A]/K$. 
\end{thm}

\begin{proof}
For any $w\in \Free[A]$, $h\in H$ and $k\in K$, we have $H^{hwk}\cap K=((H^h)^w \cap K)^k=(H^w\cap K)^k$ and so, $\rrk(H^w\cap K)=\rrk(H^{hwk}\cap K)$. Hence, the formula \eqref{eq: SHN} is independent from the chosen set of representatives for $H\backslash \Free[A]/K$. Moreover, if one of $H,K$ is trivial or cyclic then the left hand side is zero and the inequality is obvious. Also, if one of $H,K$ is not finitely generated then the right hand side is infinite and the inequality is again obvious. So, let us assume $\Trivial \neq H,K\leqslant\fg \Free[A]$. 

Now, applying successively \Cref{lem: bijection}, \Cref{cor: rrk =} and \Cref{prop: untros} we have:
\begin{align*}
\sum_{w\in H\backslash \Free[A]/K} \hspace{-12pt} \rk(H^w\cap K) 
&\,=\,
\rrk(\schreier(H)\times \schreier(K))\\
&\,=\,
\rrk(\Ati_H^* \times \Ati_K^*) \,\leq\,
2\rrk(H)\rrk(K) \,,
\end{align*}
as claimed.
\end{proof}

\begin{rem}
In particular, all but finitely many summands on the left hand side of \eqref{eq: SHN} are zero; \ie all the intersections of the form $H^w \cap K$ are trivial or cyclic, except for finitely many. Note, however, that the argument proving \Cref{cor: rrk =} also works replacing $\rrk$ to $\rk$. Since $\rk(\Ati_H^* \times \Ati_K^*)$ is finite, $\schreier(H)\times \schreier(K)$ can only have finitely many non-tree components; this means that all intersections $H^w \cap K$ are trivial except for finitely many, including the cyclic ones. 
\end{rem}

After proving this classical result, it is worth mentioning that Hanna Neumann first (on the inequality from \Cref{thm: Howson}) and Walter Neumann later (on the inequality from \Cref{thm: SHN}) already conjectured that the factor `2' could be eliminated from the right hand side of both inequalities; these are the famous `Hanna Neumann Conjecture' and `Strengthened Hanna Neumann Conjecture', proved much later by J. Friedman and I. Mineyev independently and almost simultaneously. See \cite{friedman_sheaves_2015}, \cite{mineyev_submultiplicativity_2012}, and the remarkable simplifications given by W. Dicks in \cite[Appendix B]{friedman_sheaves_2015} and \cite{dicks_simplified_2012}, respectively. This is a much deeper result making use of stronger techniques; we are not aware of any proof for these inequalities using only Stallings' machinery.

\begin{thm}[J. Friedman \cite{friedman_sheaves_2015}, I. Mineyev \cite{mineyev_submultiplicativity_2012}]\label{thm: Mineyev}
For any $H,K\leqslant \Free[A]$, we have 
 \begin{equation*}\label{eq: Mineyev}
\sum_{HwK\in H\backslash \Free[A]/K} \hspace{-15pt} \rrk(H^w\cap K)\,\leq\, \rrk(H)\rrk(K). \tag*{\qed}
 \end{equation*}
\end{thm}

\medskip

Not surprisingly, the pullback construction is also useful to solve other problems related to intersections, like the coset intersection problem and the study of malnormality. Recall that, in an arbitrary group $G$, the intersection of two cosets $Hu$ and $Kv$ is either empty ($Hu\cap Kv=\varnothing$)
or exactly a coset of the intersection, \ie $Hu\cap Kv=(H\cap K)w$, for some $w\in G$. 

\begin{named}[Coset Intersection Problem in $G=\pres{A}{R}$, $\CIP(G)$]

Given two finite subsets $R,S \subseteq (A^{\pm})^*$ and two words $u,v\in (A^{\pm})^*$, decide whether the intersection $\gen{R} u\cap \gen{S} v$ is empty and, in the negative case, compute a representative.  
\end{named}

Recall that, by \Cref{lem: recognized}.\ref{item: coset recognized}, given a subgroup $H \leqslant \Free[A]$ and $u,w \in \Free[A]$, $w \in H u$ if and only if $w$ is the label of some reduced walk $H\xrwalk{\ } Hu$ in $\schreier(H)$.

\begin{thm}
Let $H,K$ be subgroups of a free group $\Free[A]$, and let $u,v \in \Free[A]$. Then, $Hu \cap Kv \neq \varnothing$ if and only if $(H, K)$ and $(Hu, Kv)$ belong to the same connected component in $\schreier(H) \times \schreier(K)$. In particular, $\CIP(\Free[A])$ is computable.
\end{thm}

\begin{proof}
If the vertices $(H, K)$ and $(Hu, Kv)$ belong to the same connected component of the pullback $\schreier(H) \times \schreier(K)$ then there exists $w\in \Free[A]$ such that $Hw=Hu$ and $Kw=Kv$, which means that $w\in Hu\cap Kv\neq \varnothing$. The very same argument read backwards proves the converse implication.

In order to prove the computability of $\CIP(\Free[A])$ it is enough to realize that if $H$ and $K$ are finitely generated, then $\stallings(H)$ and $\stallings(K)$ are finite and computable, and hence (since $u$ an $v$ are finite words) the parts of the corresponding Schreier automata involved in checking the coset intersection emptiness property above are finite and computable as well. Concretely, it is enough to consider the product $\Ati$ of the automata obtained by eventually enlarging $\stallings(H)$ (\resp $\stallings(K)$) through $\schreier(H)$ (\resp $\schreier(K)$) with a thread reading $u$ (resp $v$) from $H$ (\resp $K$). Finally, as proved in the first paragraph, a representative for the coset intersection is obtained by reading a walk $(H, K) \xrwalk{\ } (Hu, Kv)$ in $\Ati$, in case it exists.
\end{proof}

A subgroup $H$ of an arbitrary group $G$ is said to be \defin[malnormal subgroup]{malnormal} (in~$G$) if, for every $g\in G \setmin H$, the intersection $H^g \cap H=\Trivial$. We can use our geometric description of intersections within free groups to characterize malnormality.

\begin{prop}\label{prop: malnormality}
Let $H$ be a subgroup of a free group $\Free$. Then $H$ is malnormal in $\Free$ if and only if every connected component of $\stallings(H) \times \stallings(H)$ not containing the basepoint is a tree. \qed
\end{prop}

\begin{rem}
Note that, by a counting argument, apart from the diagonal component of $\stallings(H) \times \stallings(H)$ which is obviously $\stallings(H)$ itself (just meaning that $H\cap H=H$), there must be other connected components in the pullback. These are the relevant ones in order to analyze the malnormality of $H$. 
\end{rem}

\begin{cor}[\citenr{baumslag_malnormality_1999}]
In a free group it is algorithmically decidable whether a finitely generated subgroup is malnormal. \qed
\end{cor}


\begin{exm}
The existence of malnormal subgroups of infinite rank in $\Free_2$ was needed at some point in the literature, and proved with quite technical algebraic arguments (see, \eg \parencite{das_controlled_2015}). Using Stallings' automata, this fact can be proved by just inspecting a labeled automaton like the one in \Cref{fig: nonfg malnormal}. 
\vspace{5pt}
\begin{figure}[H]
\centering
\begin{tikzpicture}[shorten >=1pt, node distance=1.2cm and 2cm, on grid,auto,auto,>=stealth']

\newcommand{\dx}{0.85}
\newcommand{\dy}{0.8}
\node[] (0)  {};

\foreach \x in {0,...,4}
\foreach \y in {0,...,4} 
{
\node[state] (\x\y) [above right = \y*\dy and \x*\dx of 0] {};
\node[state,accepting] (bp)  {};
}
\node[state] (50) [right = \dx*1.2 of 40] {};
\node[] (60) [right = \dx of 50] {};
\node[state] (05) [above = \dy*1.2 of 04] {};
\node[] (06) [above = \dy of 05] {};
\node[] (444) [above right = \dy*0.6 and \dx*0.6 of 44] {\raisebox{5pt}{\reflectbox{$\ddots$}}};
\node[state] (55) [above right = \dy*1.2 and \dx*1.2 of 44] {};
\node[] (66) [above right = \dy/2 and \dx/2 of 55] {\raisebox{5pt}{\reflectbox{$\ddots$}}};
\draw[->,blue,dotted,thick] (40) edge (50);
\draw[->,blue,dotted,thick] (50) edge (60);
\draw[->,red,dotted,thick] (04) edge (05);
\draw[->,red,dotted,thick] (05) edge (06);
\draw[->,red,dotted,thick] (50) edge (55);
\draw[->,blue,dotted,thick] (05) edge (55);

\path[draw=none,red] (00) edge node[left] {$b$} (01);
\path[draw=none,blue] (00) edge node[above] {$a$} (10);

\foreach \x [count=\xi] in {0,...,3}
{\draw[->,blue] (\x0) edge (\xi0);
\draw[->,red] (0\x) edge (0\xi);}

\foreach \y [count=\yi] in {1,...,4}{
\pgfmathsetmacro{\yy}{\y-1}
\foreach \x [count=\xi] in {0,...,\yy}
{
\draw[->,blue] (\x\y) edge (\xi\y);
\draw[->,red] (\y\x) edge (\y\xi);
}}
\end{tikzpicture}
\caption{A malnormal subgroup of \protect{$\Free[\{a,b\}]$} of infinite rank: ${H = \gen{a^kb^k a^{-k}b^{-k} \st k \geq 1}}$}
\label{fig: nonfg malnormal}
\end{figure}
\end{exm}
To conclude this section, we present a final application of the pullback technique, offering a concise and elegant proof for free groups of a well-known result that holds in general (the proof for the general case requires more sophisticated algebraic tools, such as the Kurosh Subgroup Theorem).

\begin{prop}\label{prop: int ff}
Let $G$ be a group and let $H_i \leqff K_i\leqslant G$, for $i\in [1,n]$. Then, $H_1\cap \cdots \cap H_n\leqff K_1\cap \cdots \cap K_n$. 
\end{prop}

\begin{proof}[Proof for the case $G=\Free$ free]
Let us first see that, for any $L\leqslant \Free$, $H\leqslant\ff K\leqslant \Free$ implies $H\cap L\leqslant\ff K\cap L$. Consider a basis $A$ for $K$ which extends a basis for $H$ and observe that $\stallings(H,A)$ is, simply, a graph with a single vertex and one loop for each element in $A \cap H$. Consider also $\stallings(K\cap L, A)$ 
and let us compute ${H\cap L}=H\cap (K\cap L)$
by performing the pullback of the corresponding Stallings' automata: clearly, $\stallings(H,A)\times \stallings(K\cap L,A)$ is the $A$-subautomaton of $\stallings(K\cap L,A)$ determined by the arcs with label in $A \cap H$. By \Cref{cor: incl => ff}, $H\cap L\leqff K\cap L$, as we wanted to see. 

Now let us do induction on $n$. Applying the previous fact twice, we obtain $H_1\cap H_2\leqff K_1\cap H_2\leqff K_1 \cap K_2$ and so, $H_1\cap H_2\leqff K_1\cap K_2$, which is the case $n=2$. For arbitrary $n\geq 3$, applying again the previous fact and the inductive hypothesis, we have $(H_1\cap \cdots \cap H_{n-1})\cap H_n\leqff (K_1\cap \cdots \cap K_{n-1})\cap H_n\leqff (K_1\cap \cdots \cap K_{n-1})\cap K_n$. This completes the proof. 
\end{proof}

\section{Extensions of free groups} \label{sec: extensions}

In this final section we will study \defin[extension of free groups]{extensions of free groups}, \ie the relative position of pairs of subgroups $H\leqslant K\leqslant \Free[A]$. An elementary form of group extension occurs when $H$ is a free factor of $K$, denoted $H\leqff K$. Of course, many extensions of subgroups $H\leqslant K\leqslant \Free[A]$ are not of this type; a necessary condition being $\rk(H)\leq \rk(K)$. Note, however, that this inequality in rank is not sufficient: for example, the commutator
of two ambient free generators $a,b\in A$ is not a primitive element in the ambient group and so, the cyclic subgroup it generates, $H=\gen{[a,b]}\leq \gen{a,b}=K$, is not a free factor. Another straightforward example is between cyclic subgroups: for $\trivial \neq w\in \Free[A]$ and $r,s \neq 0$, $\gen{w^r}\leqslant \gen{w^s}$ if and only if $s\,\vert\, r$ while $\gen{w^r}\leqff \gen{w^s}$ if and only if $|r|=|s|$. 

In this regard, the situation seems very different from that in more classical algebraic contexts. In the theory of vector spaces, it is well known that any subspace is always a direct summand: if $E\leqslant F\leqslant V$ then $E\leqslant_{\oplus} F$ (meaning that $F=E\oplus L$, for some complementary subspace $L\leqslant V$). In the free abelian group $\ZZ^n$ the analogous exact statement is not true: $H=\gen{(2,2)}\leqslant \ZZ^2=K$ is not a direct summand because the coordinates of this vector are not coprime to each other. However, the following variation of the vector space situation holds: \emph{``for any $H\leqslant K\leqslant \ZZ^n$ there exists a unique $L\leqslant \ZZ^n$ such that $H\leqfi L\leqslant_{\oplus} K$''}. In this sense, we can think that the statement \emph{``$H\leqslant K$ implies $H\leqslant_{\oplus} K$''} is close to be true in $\ZZ^n$, namely, true up to an appropriate finite index step. 

In the 1950's, M. Takahasi discovered a similar behavior in the context of free groups: one just needs to relax a bit more the relation between $H$ and $L$ (as we make precise below, we have to give up the finite index condition, replace unicity of $L$ by a finite number of possibilities for $L$, and, of course, exclude from the picture the non-(finitely generated) subgroups). We present a modern re-statement of Takahasi's Theorem below, along with a proof using Stallings' automata. This approach was discovered more recently and independently by several authors (see~\cite{miasnikov_algebraic_2007} for a joint exposition of this and related results).

Let us start with the notion of algebraic extension, introduced by \citeauthor{kapovich_stallings_2002} in \parencite{kapovich_stallings_2002}, as a sort of dual to the notion of free factor. 

\begin{defn}
A subgroup extension $H\leqslant K \leqslant \Free[A]$ is called \defin[algebraic extension]{algebraic}, denoted by $H\leqalg K$, if $H$ is not contained in any proper free factor of $K$, \ie if $K=L*M$ and $H\leqslant L$ imply $L=K$ (and $M=1$). We denote by $\AAEE(H)$ the set of algebraic extensions of $H$ (within $\Free[A]$). 
\end{defn}


As first examples of algebraic extensions we have the ones given above: for every $w\in \Free[A]$ and $s, r\in \NN$, $\gen{w^{s r}}\leqalg \gen{w^r}$; also, $\gen{[a,b]}\leqalg \gen{a,b}$. Below, we summarize the first properties of algebraic extensions (together with the corresponding properties for the dual notion of free factor, in the primed assertions). This result is borrowed from \cite{miasnikov_algebraic_2007}. 

\begin{prop}\label{prop: composition}
Let $H\leqslant L \leqslant K\leqslant \Free[A]$ be a sequence of extensions of free groups. Then,
 \begin{enumerate}
\item[\rm{(i)}] If $H\leqalg L\leqalg K$ then $H\leqalg K$.
\item[\rm{(i')}] If $H\leqff L\leqff K$ then $H\leqff K$.
\item[\rm{(ii)}] If $H\leqalg K$ then $L\leqalg K$, while $H\leqslant L$ need not be algebraic.
\item[\rm{(ii')}] If $H\leqff K$ then $H\leqff L$, while $L$ need not be a free factor of $K$.
 \end{enumerate}
\medskip Moreover, for $H\leqslant L_1, L_2 \leqslant K$, we have \medskip
\begin{enumerate}
\item[\rm{(iii)}] If $H\leqalg L_1$ and $H\leqalg L_2$ then $H\leqalg \langle L_1 \cup L_2\rangle$, while $H\leqslant L_1\cap L_2$ need not be algebraic.
\item[\rm{(iii')}] If $H\leqff L_1$ and $H\leqff L_2$ then $H\leqff L_1\cap L_2$, while $H$ need not be a free factor of $\gen{L_1 \cup L_2}$. 
 \end{enumerate}
\medskip And, for $H_i\leqslant K_i \leqslant \Free[A]$, $i\in [1,n]$, we have \medskip
\begin{enumerate}
\item[\rm{(iv)}] If $H_i \leqalg K_i$ for $i\in [1,n]$, then $\gen{H_1\cup \cdots \cup H_n} \leqalg \gen{K_1\cup\cdots \cup K_n}$.
\item[\rm{(iv')}] If $H_i \leqff K_i$ for $i\in [1,n]$, then $H_1\cap \cdots \cap H_n \leqff K_1\cap \cdots \cap K_n$. 
 \end{enumerate}
\end{prop}

\begin{proof} Statement (i') is direct from the definition of free  factor, and the positive parts of statements (ii') and (iii') are direct applications of \Cref{prop: int ff}. The free group $\Free[A]$ on $A=\{a,b\}$ already contains counterexamples for the negative parts: for that in (ii'), we have $\gen{a} \leqff \Free[A]$ whereas $\gen{a}\leqff \gen{a, b^2} \leqalg \Free[A]$ (since any proper free factor of $\Free[A]$ is cyclic and cannot contain $\gen{a, b^2}$). And for that in (iii') we have $\gen{[a,b]} \leqff \gen{a, [a,b]}$ and $\gen{[a,b]}\leqff \gen{b, [a,b]}$, whereas $\gen{[a,b]}\leqalg \gen{a, [a,b], b}=\Free[A]$. Finally, (iv') was already proved in \Cref{prop: int ff}.  

Now assume $H\leqalg L\leqalg K$, and let $D$ be a free factor of
$K$ containing $H$. Then, by \Cref{prop: int ff}, $D\cap L$ is a free factor of $L$ containing $H$. Since $H\leqalg L$, we deduce that $D\cap L=L$, and hence $L\leqslant D$. But $L\leqalg K$, so $D=K$. Thus, the extension $H\leqslant K$ is algebraic, proving~(i).

The positive part of (ii) is clear from the definition of algebraic extension. A counterexample for the negative part in $\Free[A]$ is as follows: $\gen{[a,b]}\leqff \gen{a, [a,b]}\leqslant \Free[A]$, while $\gen{[a,b]}\leqalg \Free[A]$.

Suppose now that $H\leqalg L_1$ and $H\leqalg L_2$, and let $D$ be a free
factor of $\gen{L_1 \cup L_2}$ containing $H$. Then, by \Cref{prop: int ff}, $H\leqslant D\cap L_i \leqff L_i$, for $i=1,2$. Since $H\leqalg L_i$, we deduce that $D\cap L_i =L_i$ and hence, $L_i \leqslant D$. Hence, $D=\gen{L_1 \cup L_2}$, and the extension $H\leqslant \gen{L_1 \cup L_2}$ is algebraic, thus proving the positive part of (iii). To conclude the proof of (iii), we will exhibit subgroups $H$, $L_1$, $L_2$ such that $H\leqalg L_i$ ($i=1,2$) but $H\leqff L_1\cap L_2$. Again in $\Free[A]$ take, for example, $L_1=\gen{a^2, b}$ and $L_2 =\gen{a^3, b}$, whose intersection is $L_1 \cap L_2 =\gen{a^6, b}$. Letting $H=\gen{a^6b^6}$, we have $H\leqff L_1\cap L_2$ but $H\leqalg L_1$ and $H\leqalg L_2$. 

Finally, suppose that, for $i\in [1,n]$, $H_i\leqalg K_i$ and let $D$ be a free factor of $\gen{K_1\cup\cdots \cup K_n}$ containing $\gen{H_1\cup \cdots \cup H_n}$. Applying \Cref{prop: int ff} to each fixed $K_j$, we have $H_j \leqslant D\cap K_j \leqff \gen{K_1\cup\cdots \cup K_n}\cap K_j =K_j$. Since $H_j\leqalg K_j$, we deduce $D\cap K_j=K_j$ and $K_j\leqslant D$. This holds for each $j\in [1,n]$, so $D=\gen{K_1\cup\cdots \cup K_n}$ and we have shown that $\gen{H_1\cup \cdots \cup H_n}\leqalg \gen{K_1\cup\cdots \cup K_n}$. This shows~(iv) and completes the proof. 
\end{proof}

The notion of algebraic extension allows for the modern restatement of Takahasi's Theorem below. While the original proof in \cite{takahasi_note_1951} relied on purely combinatorial and algebraic techniques, it was independently rediscovered by Ventura~\parencite{ventura_fixed_1997} in 1997, by Margolis--Sapir--Weil~\parencite{margolis_closed_2001} in 2001, and by Kapovich--Miasnikov~\parencite{kapovich_stallings_2002} in 2002, each in slightly different contexts.
The approach used in these three new proofs happened to be essentially the same (based in the geometric intuition given by Stallings' automata); see also the subsequent paper \parencite{miasnikov_algebraic_2007} by~\citeauthor{miasnikov_algebraic_2007} unifying the three points of view. In order to state and prove it, we need the following definition as a kind of graphical counterpart to the notion of algebraic extension. 


\begin{defn} \label{def: fringe}
Let $H\leqslant \Free[A]$,
and let $\mathcal{P}$ be a partition of
$\Verts(\stallings(H))$. We denote by $\stallings(H)/\mathcal{P}$
the 
(not necessarily reduced)
$A$-automaton obtained by identifying in $\stallings(H)$ all the vertices in every $P \in \mathcal{P}$.    
The \defin[fringe of a subgroup]{$A$-fringe} of $H$ is
\[
\OO_A(H)
\,=\,
\Set{
\gen{\stallings(H)/\mathcal{P}} 
\st \mathcal{P} \text{ is a partition of } \Verts(\stallings(H))
}.
\]
\end{defn}
Note that $\OO_A(H)$ is a family of finitely generated extensions of $H$ which may depend on $A$ and always includes $H$ itself (corresponding to the trivial partition). 

\begin{lem}
    If $H$ is finitely generated, then $\OO_A(H)$ is finite and computable.
\end{lem}
\begin{proof}
This is clear since, if $H$ is finitely generated, then $\stallings(H)$ and hence the number of partitions of $\Verts(\stallings(H))$ is finite. On the other side, for each such partition, the quotient $\stallings(H)/\mathcal{P}$ is a computable finite automaton which, using Stallings' foldings, can be reduced  to the Stallings' automaton of the corresponding element in $\OO_A(H)$. 
\end{proof}

The following property of the fringe is crucial for our purposes.

\begin{prop}\label{prop: fringe}
Let $H\leqfg \Free[A]$ and consider its fringe $\OO_A(H)$. For every (not necessarily finitely generated) subgroup $K$ satisfying $H\leqslant K\leqslant \Free[A]$, there is some $H_i \in \OO_A(H)$ such that $H\leqslant H_i\leqff K$.  
\end{prop}

\begin{proof}
Given subgroups $H\leqslant K\leqslant \Free[A]$, consider the corresponding $A$-homomorphism $\theta_{H,K}\colon \stallings(H)\xto{} \stallings(K)$ (see \Cref{prop: main prop}). Then, the image $\operatorname{Im}(\theta_{H,K})$ is a (finite and deterministic) subautomaton of $\stallings(K)$ satisfying that: (1) the restriction $\theta_{H,K}\colon \stallings(H)\onto \operatorname{Im}(\theta_{H,K})$ is onto and so, $\gen{\operatorname{Im}(\theta_{H,K})}\in \OO_A(H)$; and (2) $\operatorname{Im}(\theta_{H,K})$ is a subautomaton of $\stallings(K)$ and so, by \Cref{cor: incl => ff}, ${\gen{\operatorname{Im}(\theta_{H,K})}\leqff K}$.     
\end{proof}

\begin{thm}[\citenr{takahasi_note_1951}]\label{thm: Tak}
Let $H\leqfg \Free[A]$. Then $\AAEE(H)$ is a finite and computable collection of finitely generated subgroups of $\Free[A]$; furthermore, every subgroup $K\leqslant \Free[\Alfi]$ containing $H$ algorithmically determines a unique $H'\in \AAEE(H)$ such that $H\leqalg H'\leqff K\leqslant \Free[\Alfi]$ (called the \defin[algebraic closure]{$K$-algebraic closure of~$H$}).
\end{thm}

\begin{proof}
\Cref{prop: fringe} immediately implies that $\AAEE(H)\subseteq \OO_A(H)$ and so, $\AAEE(H)$ is a finite collection of finitely generated subgroups of $\Free[A]$. To see its computability, it only remains to be able to algorithmically distinguish between the elements from $\OO_A(H)$ that are algebraic extensions of $H$ from the ones that are not. 

This is done in~\parencite{kapovich_stallings_2002, miasnikov_algebraic_2007} by cleaning up $\OO_A(H)$ in the following way: for each pair of distinct subgroups $H_i, H_j\in \OO_A(H)$, check whether $H_i\leqff H_j$ and, in the affirmative case, remove $H_j$ from the list (see~\Cref{thm: decide ff}). We claim that, after this cleaning process, the remaining subgroups in $\OO_A(H)$ constitute precisely $\AAEE(H)$. In fact, any $K\in \AAEE(H)\subseteq \OO_A(H)$ has no proper free factor containing $H$ and so, is never cleaned up along the cleaning process. Conversely, suppose that $H_j\in \OO_A(H)$ survives the cleaning process, and let $H\leqslant D\leqff H_j$; by \Cref{prop: fringe} there is some $H_i \in \OO_A(H)$ such that $H\leqslant H_i\leqff D\leqff H_j$ and, by construction, $H_i=H_j$ and so, $D=H_j$ proving that $H_j \in \AAEE(H)$. Therefore, $\AAEE(H)$ is computable. 

Now, given $H\leqslant K\leqslant \Free[A]$, consider the smallest free factor $D$ of $K$ containing $H$ (it exists because, by \Cref{prop: int ff}, if $H\leqslant D_1\leqff K$ and $H\leqslant D_2\leqff K$, then $H\leqslant D_1\cap D_2 \leqff K$). By construction, $H\leqalg D\leqff K$. To see the uniqueness, suppose that $H\leqalg E\leqff K$. Then, $H\leqslant D\cap E\leqslant \gen{D,E}\leqslant K$. But, by \Cref{prop: composition}(iii) $H\leqalg \gen{D,E}$ and, by \Cref{prop: int ff} $D\cap E\leqff K$. Again by \Cref{prop: int ff}, $D\cap E\leqalg \gen{D,E}$, $D\cap E\leqff \gen{D,E}$ and so, $D\cap E=\gen{D,E}$ and $D=E$. Finally, to compute $D$, observe that it also coincides with the greatest algebraic extension of $H$ contained in $K$: for any bigger one, $H\leqslant D<D'\leqslant K$ with $H\leqalg D'$, we would have $D\leqff D'$, a contradiction. Thus, having computed the set $\AAEE(H)$, we can easily determine $D$, and the proof is complete. 
\end{proof}

In order to obtain $\AAEE(H)$, we need to run the cleaning process described above on $\OO_A(H)$. This requires being able to detect free factors, which can be done using classical Whitehead techniques or, alternatively, using modern algorithms based, again, on Stallings' automata.

\begin{thm}[Whitehead, \parencite{lyndon_combinatorial_2001}; Silva--Weil, \parencite{silva_algorithm_2008}; Puder, \parencite{puder_primitive_2014}] \label{thm: decide ff}
There is an algorithm which, on input an extension of finitely generated subgroups $H\leqfg K\leqfg \Free[\Alfi]$, decides whether $H$ is a free factor of $K$. \qed
\end{thm}

\section{Further reading}\label{sec: further}
Over the last decades, numerous papers have been published showing an increasing number of applications of Stallings' automata. In these works, which often possess an algorithmic flavor, one can find both new results and alternative proofs of classical ones, frequently using clearer and more transparent arguments. Without claiming to be exhaustive (or even representative), in this final section we highlight the versatility of this theory with a brief selection of these results 
and the corresponding references.

\subsection{Closures in the pro-${\mathcal V}$ topologies}

This is a nice application of Takahasi's Theorem (\Cref{thm: Tak}) to the computability of a basis for the closures of finitely generated subgroups with respect to certain topologies in the free group $\Free[n]$.

A \defin{pseudo-variety} ${\mathcal V}$ of finite groups is a family of (isomorphism classes of) finite groups that is closed under taking subgroups, quotients, and finite direct products. This notion derives from that of a \defin{variety of groups}, where infinite direct products are also allowed. For instance, all finite groups, the family of $p$-groups for a given prime $p$, finite nilpotent groups, finite solvable groups or finite abelian groups, etc., are typical examples of pseudo-varieties of finite groups.

Let $\mathcal{V}$ be a pseudo-variety of finite groups, and $G$ an arbitrary group. One can consider in $G$ the initial topology with respect to all homomorphisms $\varphi \colon G\to H$, where $H \in \mathcal{V}$ is endowed with the discrete topology on $H$ (\ie the smallest topology making all such maps continuous). This is called the \defin{pro-$\mathcal{V}$ topology} in $G$. It is easy to see that this topology coincides with the one determined by the following pseudo-metric: given two elements $g,g'\in G$, the \defin{${\mathcal V}$-distance} between them is $d_{\mathcal V} (g,g')=2^{-v(g,g')}$, where $v(g,g')$ is the smallest cardinal of a group $H\in {\mathcal V}$ for which there is a homomorphism $\varphi \colon G\xto{} H$ \emph{separating} $g$ and $g'$ (i.e., such that $g\varphi \neq g'\varphi$), and  $d_{\mathcal V} (g,g')=0$ if  there is no such $H$. In case the group $G$ is \defin{residually-${\mathcal V}$} (for every two distinct elements $g,g'\in G$ there exists $H\in {\mathcal V}$ and a homomorphism $\varphi \colon G\xto{} H$ separating $g$ and $g'$), then $d_{\mathcal V}$ is a real metric, and the induced topology on $G$ becomes Hausdorff. 

The above examples of pseudo-varieties give rise to the so-called \defin[pro-finite topology]{pro-finite}, \defin[pro-$p$ topology]{pro-$p$}, \defin[pro-nilpotent topology]{pro-nilpotent}, \defin[pro-solvable topology]{pro-solvable}, and \defin[pro-abelian topology]{pro-abelian} topologies of $G$, respectively. It is then easy to see that, in general, the ${\mathcal V}$-closure of a subgroup $H\leqslant G$ is again a subgroup, $H\leqslant \Cl_{\mathcal V}(H)\leqslant G$; this observation opens the door to nice questions about algebraic or algorithmic properties of these closure operators.   

Particularizing to free groups, $G=\Free[n]$, and restricting the attention to extension-closed pseudo-varieties~${\mathcal V}$ (${\mathcal V}$ is called \defin[extension-closed pseudo-variety]{extension-closed} if, for any short exact sequence $1\xto{} A\xto{} B\xto{} C\xto{} 1$ of finite groups, $A,C\in {\mathcal V}$ imply $B\in {\mathcal V}$), in~\cite{margolis_closed_2001}, \citeauthor{margolis_closed_2001} proved that any free factor of a ${\mathcal V}$-closed subgroup of $\Free[n]$ is again ${\mathcal V}$-closed.
This automatically connects with Takahasi's Theorem because it implies that the pro-${\mathcal V}$ closure of any subgroup $H\leqfg \Free[n]$ is an algebraic extension of $H$ itself; that is, $H\leqalg \Cl_{\mathcal V}(H)$. Using this idea, the authors of \parencite{margolis_closed_2001} provided algorithms, based on Stallings' automata, to compute the pro-finite, the pro-$p$, and the pro-nilpotent closures of finitely generated subgroups $H\leqfg \Free[n]$. 


The computability of the pro-solvable closure is still an open problem. Being an extension-closed pseudo-variety of finite groups, the above argument works and so, $H\leqalg \Cl_{\operatorname{sol}}(H)$ for any $H\leqfg \Fn$. However, no criterion is known yet to distinguish, among the computed candidates in $\AAEE(H)$, which one is the solvable-closure of $H$ (it would suffice to find an algorithm to decide whether a given $H\leqfg \Fn$ is solvable-closed or not).

\subsection{Relative order and spectra}

In \parencite{delgado_relative_2024}, \citeauthor{delgado_relative_2024} consider natural generalizations of the concepts of root and order of an element in a group $G$, and study these notions in the context of free groups. 

Given $g\in G$ and a subset $S\subseteq G$, $g$ is said to be a \defin[$k$-root \wrt a set]{$k$-root} of $S$ if $g^k \in S$; and the \defin[relative order]{(relative) order} of $g$ in $S$, denoted by $\ord_S(g)$, is the minimum $k\geqslant 1$ such that~$g^k \in S$, and zero if there is no such $k$ (in particular, $\ord_S(g)=1$ if and only if~$g\in S$). The set of elements from $G$ of order $k$ in $S$ is called the \defin[$k$-preorder]{$k$-preorder of~$S$}, and the set of orders in $S$ of the elements in~$G$ is called the \defin[relative spectrum]{(relative) spectrum of~$G$ in~$S$}.


Again, Stallings' automata are a natural tool for studying these notions in the context of free groups. The starting idea is that an element $w\in \Free[n]$ has order $k\geq 1$ in $H\leqslant \Free[n]$ if and only if the sequence of vertices $\vvertii =(\bp w^i)_{i\geq 0}$ in $\schreier(H)$ has its first repetition at $i=k$ (\ie $\bp =\bp w^k$). A \defin{closed trail} is a sequence of vertices $\vvertii =(\vertii_0,\vertii_1,\ldots,\vertii_k =\vertii_0)$ in $\schreier(H)$ whose only repetition is $\vertii_k =\vertii_0$. Then, for each such trail $\vvertii$, one can use the pullback technique to compute the (possibly empty) set of elements $w$ realizing it. 
The problem is, of course, that there are, in general, infinitely many trails of length $k$ to consider in $\schreier(H)$. However, it is not difficult to see that if we restrict the preorbits to cyclically reduced words, then one must only consider orbits in the Stallings' digraph $\rstallings(H)$, which is finite and computable if $H$ is finitely generated. The final step is to see that one can overcome the gap between cyclically reduced and general preorbits by conjugating by $H$, which can still be done in an algorithmic way. 

As a consequence, the authors prove the computability of the relative order of an element \wrt any finitely generated subgroup (and its cosets); as well as the computability of the sets of relative preorders, relative $k$-roots, and the relative spectrum of any finitely generated subgroup of the free group, among other related results; see~\parencite{delgado_relative_2024}.  

\subsection{Coset enumeration}

Another classical problem in algorithmic group theory is \defin{coset enumeration}. Let~$G$ be a group given by a finite presentation, and let $H\leqfg G$. A seminal result by \citeauthor{todd_practical_1936} in \cite{todd_practical_1936} states that if the index ${\ind{H}{G}}$ is finite, then a list of coset representatives can always be obtained algorithmically (see also \cite{knuth_simple_1970} for the alternative ---and more modern--- Knuth--Bendix algorithm). As originally stated, Todd--Coxeter's algorithm is somewhat complicated, but it turns out that the ideas involved admit a very transparent interpretation in terms of Stallings' automata (see~\cite{stallings_todd-coxeter_1987}).

Let us reformulate this interpretation in our automata language. Suppose that
\smash{$\ncl{R} \xinto{\ } \Free[A] \xonto{\rho\, } G$} is a finite presentation for $G$.
Then, if $S$ is a subset of $\Free[A]$, and $H=\gen{S \rho} \leqslant G$, it is clear that the sets $H \backslash G$ and $G/H$ (of right and left cosets of $H$ in $G$) are in bijection with the set $H \rho \preim \backslash \Free[A]$ (of cosets of the full preimage $H\rho\preim$ in $\Free[A]$); hence the index of $H$ in $G$ is expressible in the background group $\Free[A]$ as $\Ind{H}{G}=\Ind{H\rho\preim}{\Free[A]} =\Ind{\gen{S}\ncl{R}}{\Free[A]}$. Therefore, it is enough to compute the Schreier automaton of the subgroup $H\rho\preim =\gen{S}\ncl{R}\leqslant \Free[A]$ to recover the cosets of $H$ in $G$. This approach entails the obstacle that the target subgroup $H\rho\preim =\gen{S}\ncl{R}$ may not be finitely generated. However, this obstacle disappears if we restrict ourselves to subgroups $H\leqslant\fin G$ of finite index (equivalently $H\rho\preim \leqslant\fin \Free[A]$).  

So, assume that $S$ is finite and $\ind{H}{G}<\infty$. Note that then the automaton~$\Atiii$  obtained after attaching the flower automaton $\flower(R)$ to every vertex in the (potentially infinite) automaton $\schreier(\gen{S})$ recognizes $\gen{S}$ at the basepoint, and $\gen{R}$ at every vertex; that is, $\gen{\Atiii} = \gen{S} \ncl{R} = H \rho\preim $. Now, the key idea is to attach these (theoretically infinitely many) flower automata $\flower(R)$ sequentially, starting from $\flower(S)$ and folding after each layer addition. Formally, we set $\Ati_{\!0} =\flower(S)$ and, for each $i \geq 0$, $\Ati_{\!i+1}$ is the automaton obtained after attaching $\flower(R)$ to every vertex in~$\Ati_{\!i}$ and folding.

By construction, every automaton $\Ati_{\!i}$ in this sequence is reduced, and recognizes the subgroup $\gen{S}$ at the basepoint, and $\gen{R}$ at every vertex \emph{except possibly the ones corresponding to the $\flower(R)$'s attached at the $i$-th step}. Since the subgroup $H \rho \preim$ is, by hypothesis, of finite index in $\Free[A]$, we know that $\stallings(\Atiii) = \stallings(H \rho \preim) = \schreier(H \rho \preim)$ will be a finite and saturated reduced automaton recognizing $\gen{S}$ at the basepoint and $\gen{R}$ at every vertex. This means that the sequence $(\Ati_{\!i})_{i \geq 0}$ must stabilize after a finite number of steps precisely at $\schreier(H \rho \preim)$, which is therefore computable.

\subsection{Whitehead’s cut vertex lemma}

A classical result from the 1930s by Whitehead~\parencite{whitehead_certain_1936}, known as the \defin{cut vertex lemma}, has been fundamental in studying certain properties of free groups. This lemma states that if an element $w$ in a free group $\Free[A]$ is \defin[primitive element]{primitive} (that is, belongs to some basis of $\Free[A]$), then its (finite and easily constructible from $w$) \defin{Whitehead graph} $\Wh_A(w)$ is either disconnected or has a cut vertex (which disconnects a connected component). Stallings later extended this result in~\parencite{stallings_whitehead_1999}, showing that if a set $S \subseteq \Free[A]$ is \defin[separable subset]{separable} (i.e., there exists a non-trivial decomposition $\Free[n]=H*K$ such that any $s\in S$ has a conjugate either in $H$ or in $K$), then $\Wh_A(S)$ also has a cut vertex or is disconnected.

Both the original proof by Whitehead and the extension  given by Stallings involve the use of \defin[handlebody]{handlebodies}, a family of 3-dimensional manifolds whose fundamental groups are free groups. In 2019, Heusener and Wiedmann provided a much simpler proof for both results using Stallings' automata, see \parencite{heusener_remark_2019}. They defined the \defin{Whitehead graph} of an $A$-automaton, and introduced the notion of \index{almost-rose}\defin{$A$-almost-rose}, to prove that almost-roses always have cut vertices and that for any separable set $S$, there exists an almost-rose $\bm{\Ati}$ such that $\Wh_A(S)$ is contained in $\Wh_A(\bm{\Ati})$, both having the same set of vertices. 

\subsection{The Herzog--Schönheim conjecture} 

Another interesting problem for which Stallings' automata have been used is the Herzog–Schönheim conjecture. This conjecture stems from the following arithmetic problem: a \defin[cover of $\mathbb{Z}$]{cover} of the set of integers $\mathbb{Z}$ is a finite collection of remainder classes $r_i+d_i\mathbb{Z}$, $i=1,\ldots ,s$, disjoint to each other, and such that $\mathbb{Z}=\bigsqcup_{i=1}^s (r_i+d_i\mathbb{Z})$. Typical examples of covers are $\mathbb{Z}=0+1\mathbb{Z}$ (the trivial one), $\mathbb{Z}= 2\mathbb{Z} \sqcup (1+2\mathbb{Z})$, $\mathbb{Z}=4\mathbb{Z}\sqcup (2+4\mathbb{Z})\sqcup (1+2\mathbb{Z})$, $\mathbb{Z}=4\mathbb{Z}\sqcup (2+12\mathbb{Z})\sqcup (6+12\mathbb{Z})\sqcup (10+12\mathbb{Z})\sqcup (1+2\mathbb{Z})$, etc. This concept was introduced by P. Erd\"os in~\parencite{erdos_integers_1950}, who conjectured that in any such non-trivial cover the largest index $d_s$ must appear at least twice. This conjecture was proved independently by Davenport, Rado, Mirsky, and Newman (see~\parencite{znam_exactly_1970}) using analysis of complex functions; furthermore, it was proved that this largest index $d_s$ appears at least $p$ times, where $p$ is the smallest prime dividing $d_s$ (among other related results). See~\parencite{ginosar_tile_2018} for an alternative modern proof using group representations. 

In 1974, M. Herzog and J. Sch\"onheim extended Erd\"os’ conjecture to arbitrary finitely generated groups and launched the following much more general conjecture: if $\{ H_1g_1,\ldots ,H_sg_s\}$, $s\geqslant 2$, is a nontrivial \defin{coset partition} of a finitely generated group $G$ (meaning that $H_i$ are finite index subgroups of $G$, and $G=\bigsqcup_{i=1}^s H_i g_i$) then the list of indices $d_i=\ind{H_i}{G}$ must contain at least a repetition; one usually refers to this fact by saying that the coset partition has \defin[multiplicity of a partition]{multiplicity}. The common approach to the Herzog--Sch\"onheim conjecture is through finite groups: it is easy to see that it reduces to the case where $G$ is finite. However, although during the following decades several papers appeared providing partial results, the conjecture in its general form remains still open today. 

In a series of papers by F. Chouraqui (see~\parencite{chouraqui_space_2018,chouraqui_herzog-schonheim_2019,chouraqui_about_2021,chouraqui_approach_2020,chouraqui_Herzog_2021}),
the author adopts a completely different approach to attack Herzog--Sch\"onheim conjecture: it can also be easily reduced to the case where $G$ is finitely generated free (by taking full preimages of coset partitions of $G$ through the canonical projection $\rho\colon \Free[n] \onto G$). Now finite index subgroups in $\Free[n]$ are saturated Stallings' automata and a coset partition, $\Free[r]=\bigsqcup_{i=1}^s H_i w_i$, can be codified with the finite list of Stallings' automata $\Ati_{\!1},\ldots ,\Ati_{\!s}$ for the cosets $H_1 w_1,\ldots, H_s w_s$, respectively; the goal is to show that the equality $\Free[n]=\bigsqcup_{i=1}^s H_i w_i$ implies that at least two of them have the same number of vertices, \ie $\card \Verts \Ati_{\!i}=\card \Verts \Ati_{\!j}$, for some $i \neq j$. In the above mentioned series of papers Chouraqui finds several conditions on the $\Ati_{\!i}$'s (in terms of their transition monoids, or their adjacency matrices) to ensure this conclusion, proving several special cases of the Herzog--Sch\"onheim conjecture.

 \subsection{Random groups}

For many years, the study of random groups has been a topic of interest in the literature, including efforts to prove the `genericity' of several algebraic properties. Of course, a preliminary issue is to properly define the meaning of `genericity' in this context; that is, to establish a probability model where such questions are considered. 

A classical way of doing this is via presentations: let us restrict ourselves to finitely presented groups, fix the number of generators (say, $n$) and the number of relations (say, $m$), and we can consider an $n$-generated $m$-related group just as an $m$-tuple of words (or cyclic words) $r_1,\ldots ,r_m\in \Free[n]$ (obviously referring to the group $G=\pres{a_1,\ldots ,a_n}{r_1, \ldots ,r_m}$). There are different ways to define a probability distribution on the set of $m$-tuples of words (maybe of a given bounded length and then let this length tend to infinity) and this provides a well-defined notion of ``random'' group. Using this model and some technical variations, the property of being nontrivial and hyperbolic happens to have probability one, as hinted by \citeauthor{gromov_hyperbolic_1987} in his seminal monograph \parencite{gromov_hyperbolic_1987}, and formally  proved later by \citenr{olshanskii_almost_1992} and \citeauthor{champetier_phd_1991} \cite{champetier_phd_1991,champetier_proprietes_1995}.


The same and several other results were proved later in a related (but different) model of genericity, now called the \defin{Arzhantseva--Ol'shanskii model}
\parencite{arzhantseva_generality_1996,arzhantseva_class_1996}. It turns out that in many of these works Stallings' automata play, again, an essential role; see \eg \parencite{arzhantseva_class_1996,arzhantseva_groups_1997,arzhantseva_phd_1998,arzhantseva_generic_1998,arzhantseva_property_2000,kapovich_genericity_2005,kapovich_random_2009}. 

We would like to point out, however, some nuances of these \defin{word-based models}: we are really counting $m$-tuples of words and so neglecting the fact that different $m$-tuples $R$ and $R'$: (1) could generate the same subgroup $\gen{R}=\gen{R'}\leqslant \Free[n]$; (2) different subgroups, 
$\gen{R}\neq \gen{R'}\leqslant \Free[n]$, but the same normal closure $\ncl{R}=\ncl{R'}\normaleq \Free[n]$; or finally (3) even different normal closures, $\ncl{R}\neq \ncl{R'}\normaleq \Free[n]$, but isomorphic quotients, $\Free[n]/\ncl{R}\isom \Free[n]/\ncl{R'}$. In all these cases the presented group is the same, $\pres{A}{R}=\pres{A}{R'}$, but counted twice, by the $m$-tuple $R$ and, again, by the $m$-tuple $R'$. And it is highly non clear whether this over counting affects uniformly to all (isomorphism classes of) groups or gives preference to some of them at the expense of some others. 

An appealing attempt to overcome the first one of the previous three objections was done by \citeauthor{bassino_random_2008} in \parencite{bassino_random_2008} (see also~\parencite{bassino_statistical_2013,bassino_generic_2016,bassino_genericity_2016}), where the authors used for the first time Stallings' automata to stratify the family of finitely generated subgroups of the free group, and study them asymptotically. With this new point of view, a finitely presented group is modeled as $\pres{A}{\Ati}$, where $\Ati$ is a finite Stallings' automaton (and referring to the group $G=\Free[A]/\ncl{\gen{\Ati}}$). Using the number of vertices of $\Ati$ as a natural measure of the size of the corresponding subgroup $\gen{\Ati}\leqslant \Free[n]$, we get the so-called \defin{graph-based model} introduced in~\parencite{bassino_random_2008}. Based on this nice idea, \parencite{bassino_random_2008} contain results like the one saying that the expected (\ie average) rank of a randomly chosen size $k$ subgroup of $\Free[n]$ is asymptotically equivalent~to $(n-1)k-n\sqrt{k}+1$. 


Once arrived at this point it is worth mentioning that \citenr{bassino_statistical_2013}, using the graph-based distribution to analyze finitely presented groups, got the surprisingly contrasting result that, generically, a finitely presented group $\pres{A}{\Ati}$ is trivial (\ie with probability one, the normal closure of a randomly chosen subgroup $\gen{\Ati}$ of $\Free[n]$, is $\Free[n]$ itself). 


Here, the idea is the following. It can be proved that, for any given letter $a\in A$, the probability that $\Ati$ contains a closed walk labeled $a^r$ for some $0\neq r\in \mathbb{N}$ tends to 1 as the size of $\Ati$ tends to infinity. This translates into saying that, with probability tending to 1, in a random presentation $\pres{A}{\Ati}$ any letter represents a torsion element. Strengthening this argument, it can also be proved that, in fact, the probability that $\Ati$ contains several closed walks labeled $a^{r_1},\ldots ,a^{r_d}$ with $\gcd(r_1,\ldots ,r_d)=1$, also tends to 1. Of course, this translates into saying that, with probability tending to 1, in a random presentation $\pres{A}{\Ati}$ any letter represents the trivial element, \ie $\pres{A}{\Ati}=1$.

\subsection{Stallings' techniques beyond free groups}

The tremendous success of the theory of Stallings' automata in the study of properties of subgroups of the free group has led to many attempts to extend it to broader contexts; first by Stallings himself, to groups acting non-freely on graphs and trees   \parencite{stallings_foldings_1991}, and later successively refined and extended in the graph of groups context by multiple authors, see \eg 
\parencite{bestvina_bounding_1991, dunwoody_folding_1998,dunwoody_small_1999, rips_cyclic_1997, sela_acylindrical_1997, sela_diophantine_2001,
dicks_equalizers_1999,dunwoody_groups_1997,guirardel_approximations_1998,guirardel_reading_2000,bowditch_peripheral_2001,bogopolski_uniqueness_2002}. 

Other (mostly algorithmically-oriented) generalization targets of Stallings' automata include 
inverse monoids and semigroups \parencite{margolis_free_1993,margolis_closed_2001,delgado_combinatorial_2002,steinberg_inverse_2002}, groups satisfying certain small-cancellation properties~\parencite{arzhantseva_class_1996}, fully residually free groups~\parencite{myasnikov_fully_2006,kharlampovich_subgroups_2004, nikolaev_finite_2011}, free products~\parencite{ivanov_intersection_1999}, certain graphs of groups ~\parencite{kapovich_foldings_2005}, amalgams of finite groups~\parencite{markus-epstein_stallings_2007}, groups acting freely on $\ZZ^n$-trees~\parencite{nikolaev_membership_2012}, virtually free groups~\parencite{silva_finite_2016}, quasi-convex subgroups~\parencite{kharlampovich_stallings_2017}, free-abelian by free groups~\parencite{delgado_algorithmic_2013, delgado_extensions_2017,Delgado_stallings_FTA_2022}, CAT(0) cube complexes~\parencite{beeker_stallings_2018},
certain relatively hyperbolic groups~\parencite{kharlampovich_generalized_2020}, right-angled Coxeter groups~\parencite{dani_subgroups_2021}, etc.

\medskip

This extensive list of applications and generalizations shows that Stallings' techniques are a very active and modern research topic, well connected with other lines of investigation within Group Theory. It is clear that this field will continue to grow in the near future, providing both new applications to understand more deeply the structure of the lattice of subgroups of the free group and new generalizations to other contexts beyond free groups.


\section*{Acknowledgments}

Both authors acknowledge support from the Spanish \emph{Agencia Estatal de Investigación} through grant PID2021-126851NB-I00 (AEI/FEDER, UE).
The first author also acknowledges support from the \emph{Universitat Politècnica de Catalunya} in the form of a ``María Zambrano'' scholarship.

\newpage
\renewcommand*{\bibfont}{\small}
\printbibliography
\newpage
\printindex

\Addresses

\end{document}